\documentclass[12pt,a4paper]{amsart}
\usepackage{latexsym}
\usepackage{graphics} 
\usepackage{epsfig} 
\usepackage{amsmath} 
\usepackage{amsfonts, amssymb} 
\usepackage{amsthm,enumerate}
\usepackage{color}
\usepackage{mathrsfs}
\usepackage{amscd}      
\usepackage{ulem}
\usepackage{euscript}
\usepackage{arydshln}
\usepackage[all]{xy}

\usepackage[
	a4paper,
	top=25mm,
	bottom=20mm,
	left=25mm,
	right=25mm,
]{geometry}

\theoremstyle{definition}
\newtheorem{Def}{Definition}[section]
\newtheorem{Hyp}{Hypothesis}[section]

\newtheorem{Eg}{Example}[section]
\newtheorem{Rm}{Remark}[section]

\theoremstyle{plain}
\newtheorem{Prop}[Def]{Proposition}
\newtheorem{Lem}[Def]{Lemma}
\newtheorem{Thm}[Def]{Theorem}
\newtheorem{Cor}[Def]{Corollary}

\numberwithin{equation}{section}

\newcommand{\authorfootnotesA}{\renewcommand\thefootnote{$\flat$}}%
\newcommand{\authorfootnotesB}{\renewcommand\thefootnote{$\sharp$}}%
\newcommand{\authorfootnotesC}{\renewcommand\thefootnote{$\diamond$}}%
\newcommand{\authorfootnotesD}{\renewcommand\thefootnote{$\S$}}%

\begin{document}

\begin{center}
	\LARGE 
	Distributional It\^o's Formula and
	Regularization of Generalized Wiener Functionals \par \bigskip

	\normalsize
	\authorfootnotesA
	Takafumi Amaba\footnote{
	This work was supported by JSPS KAKENHI Grant Number 15K17562.
	}\textsuperscript{1}
	and
	\authorfootnotesB
	Yoshihiro Ryu\textsuperscript{2}
	\authorfootnotesC
	\authorfootnotesD

	\textsuperscript{1}Fukuoka University,
	8-19-1 Nanakuma, J\^onan-ku, Fukuoka, 814-0180, Japan.\par \bigskip
	\textsuperscript{2}Ritsumeikan University,
	1-1-1 Nojihigashi, Kusatsu, Shiga, 525-8577, Japan.\par \bigskip

	
	\email{(T. Amaba) fmamaba@fukuoka-u.ac.jp}
	\email{(Y. Ryu) gr0095kp@ed.ritsumei.ac.jp}

	\today
\end{center}

\begin{quote}{\small {\bf Abstract.}
	We investigate Bochner integrabilities
	of generalized Wiener functionals.
	We further formulate an It\^o formula
	for a diffusion in a distributional setting,
	and apply
	it
	to investigate
	differentiability-index $s$
	and integrability-index $p \geqslant 2$
	for which the Bochner integral belongs to $\mathbb{D}_{p}^{s}$.
}
\end{quote}

\renewcommand{\thefootnote}{\fnsymbol{footnote}}



\section{Introduction}
In this paper, we justify the symbol
``$\int_{0}^{T} \delta_{y} (X_{t}) \mathrm{d}t$"
denoting a quantity relating to the local time of a
$d$-dimensional diffusion process
$X=(X_{t})_{t \geqslant 0}$
with $X_{0}$ being deterministic
(in multi-dimensional case, we assume $X_{0} \neq y$),
or more generally,
the object
``$\int_{0}^{T} \Lambda (X_{t}) \mathrm{d}t$"
where $\Lambda$ is a distribution.
Our diffusion process $X=(X_{t})_{t \geqslant 0}$
is assumed to satisfy a $d$-dimensional
stochastic differential equation
$$
\mathrm{d} X_{t}
=
\sigma ( X_{t} ) \mathrm{d} w(t)
+
b ( X_{t} ) \mathrm{d} t,
\quad
X_{0}=x \in \mathbb{R}^{d},
$$
where
$
w = (w^{1}(t), \cdots , w^{d}(t))_{t \geqslant 0}
$
is
a $d$-dimensional Wiener process with $w(0) = 0$.
The main conditions on
$
\sigma = ( \sigma_{j}^{i} )_{1 \leqslant i,j \leqslant d}
$
and
$
b = (b^{i})_{1 \leqslant i \leqslant d}
$
under which we will work are combinations from the following.

\begin{Hyp} 
\begin{itemize}
\item[(H1)]
the coefficients
$\sigma$ and $b$ are $C^{\infty}$,
and have bounded derivatives in all orders $\geqslant 1$.

\item[(H2)]
$( \sigma \sigma^{*} ) (x)$
is strictly positive,
where $x = X_{0}$
and $\sigma^{*}$ is the transposed matrix of $\sigma$.

\item[(H3)]
$\sigma \sigma^{*}$ is uniformly positive definite, i.e.,
there exists $\lambda > 0$ such that
$$
\lambda \vert \xi \vert_{\mathbb{R}^{d}}^{2}
\leqslant
\langle
	\xi ,
	( \sigma \sigma^{*} )
	( y )
	\xi
\rangle_{\mathbb{R}^{d}}
\quad
\text{for all
$
\xi,y
\in \mathbb{R}^{d}
$,}
$$
\end{itemize}
\noindent
where
$
\langle \bullet , \bullet \rangle_{\mathbb{R}^{d}}
$
is the standard inner product on $\mathbb{R}^{d}$,
and
$
\vert \bullet \vert_{\mathbb{R}^{d}}
=
\vert \bullet \vert
$
is the corresponding norm.
\begin{itemize}
\item[(H4)]
$\sigma$ and $b$ are bounded.
\end{itemize}
\end{Hyp} 

We further formulate stochastic integrals
and an It\^o formula
in this distributional setting
and investigate when the local time
belongs to $\mathbb{D}_{2}^{s}$
(the Sobolev space of integrability-index $2$ and
differentiability-index $s \in \mathbb{R}$
with respect to the Malliavin derivative).

In fact, we will formulate
$\int_{0}^{T} \delta_{y} (X_{t}) \mathrm{d}t$
as a Bochner integral in the space of
generalized Wiener functional.
We remark here the Bochner integrability
seems nontrivial
when $y=X_{0}$,
since
$\delta_{y} (X_{t})$
no longer makes sense at $t=0$.
On the other hand,
the local time is usually formulated as
a classical Wiener functional.
Hence, once the Bochner integrability
is proved, a ``smoothing effect" should occur
in the Bochner integral
$\int_{0}^{T} \delta_{y} (X_{t}) \mathrm{d}t$,
i.e.,
the differentiability-index for
$\int_{0}^{T} \delta_{y} (X_{t}) \mathrm{d}t$,
should be greater than that of $\delta_{y} (X_{t})$.

In the case of Brownian motion
$X_{t}=w(t)$,
everything can be explicitly computed
and we can exhibit this phenomenon.
Namely, the following is the prototype of this study.

Let $\mathscr{S}^{\prime} (\mathbb{R}^{d})$
denote the space of all Schwartz distributions
on $\mathbb{R}^{d}$.

\begin{Thm} 
\label{Main_Thm1} 
Assume $d=1$.
Let
$\Lambda \in \mathscr{S}^{\prime} (\mathbb{R})$
and
$s \in \mathbb{R}$.
If $\Lambda (w(T)) \in \mathbb{D}_{2}^{s}$ then
the mapping
$$
(0, T] \ni t \mapsto
\sqrt{\frac{T}{t}}
\Lambda \Big( \sqrt{\frac{T}{t}} w(t) \Big)
\in \mathbb{D}_{2}^{s}
$$
is Bochner integrable in $\mathbb{D}_{2}^{s}$
and we have
$$
\int_{0}^{T}
\sqrt{\frac{T}{t}}
\Lambda \Big( \sqrt{\frac{T}{t}} w(t) \Big)
\mathrm{d}t
\in \mathbb{D}_{2}^{s+1}.
$$
\end{Thm} 

Since it is known that
$\delta_{0} (w(t)) \in \mathbb{D}_{2}^{(-1/2)-}$
(see Watanabe \cite{Wa91}),
we obtain
$
\int_{0}^{T} \delta_{0} (w(t)) \mathrm{d}t
\in \mathbb{D}_{2}^{(1/2)-}
$,
which agrees with the result by
Nualart-Vives \cite{NuVi}
and Watanabe \cite{Wa94}.
The proof of this theorem
is due to
the chaos computations
(which is essentially the same as \cite{NuVi}
but with {\it no} approximations of the integrand).
When $b = \sigma \sigma^{\prime} / 2$,
this computation brings
the H\"older-continuity
of the local time
with respect to space variable.
The norm on $\mathbb{D}_{p}^{s}$ will be denoted by
$
\Vert
\bullet
\Vert_{p,s}
$.

\begin{Thm} 
\label{Main_Thm2} 
Let $d=1$.
Assume {\rm (H1)}, {\rm (H3)} and that
the drift-coefficient is given by
$b = \sigma \sigma^{\prime} / 2$.
Then for each $s < \frac{1}{2}$
and
$\beta \in (0, \min\{ \frac{1}{2}-s,1 \} )$,
there exists a constant $c=c(s, \beta ) > 0$ such that
\begin{equation*}
\begin{split}
\big\Vert
\sigma (y)
\int_{0}^{1} \delta_{y} ( X_{t} ) \mathrm{d}t
-
\sigma (z)
\int_{0}^{1} \delta_{z} ( X_{t} ) \mathrm{d}t
\big\Vert_{2,s}
\leqslant
c
\vert y-z \vert^{\beta}
\end{split}
\end{equation*}
for every $y, z \in \mathbb{R}$.
\end{Thm} 

The proof of this theorem
seems
interesting in its own right.
The study of H\"older continuity of local times had been initiated by
Trotter \cite[inequalities (2.1) and (2.3)]{Tr},
in which the almost-sure H\"older-continuity of the Brownian local time
$\{ l(t,x) : t \geqslant 0, x \in \mathbb{R}\}$
in time-space variable $(t,x)$
was proved
(see also Boufoussi-Roynette \cite{BoRo}).
There are a lot of such studies
(see, e.g., Liang \cite{Li},
Ait~Ouahra-Kissami-Ouahhabi \cite{AKO},
Shuwen-Cheng \cite{ShCh}
and references therein).

Theorem \ref{Main_Thm2} implies immediately
the following.

\begin{Cor} 
\label{Main_Cor1} 
Under the conditions in Theorem \ref{Main_Thm2},
let $p_{t} (x,y)$ be the transition density
of $X_{t}$.
Then the mapping
$$
\mathbb{R} \ni y
\mapsto
\sigma (y)
\int_{0}^{1}
p_{t} (x,y)
\mathrm{d}t
\in \mathbb{R}
$$
is
(globally)
$\beta$-H\"older continuous for every $\beta < 1$.
\end{Cor} 

The latter half of this paper concerns with
an It\^o formula
in a distributional setting.
The classical It\^o-Tanaka formula had been extended
with several formulations
(see
F\"ollmer-Protter-Shiryayev \cite{FPA},
Bouleau-Yor \cite{BoYo},
Wang \cite{Wang},
Kubo \cite{Ku1} and so on).
In particular,
according to results in \cite{Wang} and \cite{BoYo},
the It\^o-Tanaka formula for $f(X_{t})$
is valid in the case where $f$ is just a convex function.
In our case, we obtained
the following.

\begin{Thm} 
\label{Main_Thm3}
Assume {\rm (H1)}, {\rm (H2)} and {\rm (H4)}.
Let
$A_{i} = \sum_{j=1}^{d} \sigma_{i}^{j} \partial / ( \partial x_{j} )$
and
$L$ be the generator of the diffusion process $X$.
Suppose that $f: \mathbb{R}^{d} \to \mathbb{R}$ is
a measurable function such that
\begin{itemize}
\item[(i)]
$f$ is continuous at $x$,

\item[(ii)]
$f$
has
at most exponential growth,

\item[(iii)]
$
\int_{0}^{T}
\Vert
	( A_{i}f ) (X_{t})
\Vert_{2,-k}^{2}
\mathrm{d}t
< +\infty
$
for $i=1, 2, \cdots , d$,

\item[(iv)]
$
\int_{0}^{T}
\Vert
	(Lf) (X_{t})
\Vert_{2,-k}
\mathrm{d}t
< +\infty
$
\end{itemize}
for some $k \in \mathbb{N}$.
Then we have
\begin{equation*}
\begin{split}
&
f(X_{T}) - f(x)
=
\sum_{i=1}^{d}
\int_{0}^{T}
(A_{i}f) (X_{t})
\mathrm{d}w^{i}(t)
+
\int_{0}^{T}
(Lf) (X_{t})
\mathrm{d}t
\quad
\text{in $\mathbb{D}^{-\infty}$.}
\end{split}
\end{equation*}
\end{Thm} 

We can drop the assumption {\rm (H4)}
if $f$ has at most polynomial growth.

The definition of stochastic integral will be
given in
Section~\ref{Stochastic_Integral}
and the time-integral
$
\int_{0}^{T}
(Lf) (X_{t})
\mathrm{d}t
$
is understood in the sense of Bochner integral
in $\mathbb{D}_{2}^{-k}$.
Kubo \cite{Ku1} also obtained an It\^o formula
for Brownian motion in a distributional setting.
However,
his formula does not need to consider the Bochner integrability
because the time-interval of integration is a closed interval
excluding zero.
A generalization to
the case of one-dimensional fractional Brownian motion
was done by Bender
\cite{Be}
(and see references therein),
in which, even the case where
the time-interval of integration is such as $(0,T]$
is considered
(\cite[Theorem 4.4]{Be}),
though the first distributional derivative of $f$ is assumed to be
a regular distribution.
But he did not give a systematic treatment
of
Bochner integrability.
Theorem \ref{Main_Thm3} will be proved in
Section~\ref{Generalized_sec}
and it will be established
in
Section~\ref{Generalized_Ito}
even the case where $f$ itself is a distribution of exponential-type
and furthermore the time-interval of integration is $(0,T]$.

A distribution
$\Lambda \in \mathscr{S}^{\prime} (\mathbb{R}^{d})$
is said to be {\it positive} if
$
\langle
	\Lambda ,
	f
\rangle
\geqslant 0
$
for every
nonnegative
test function
$f \in \mathscr{S}(\mathbb{R}^{d})$.
To include local times for diffusions in our scope,
we prepare the following

\begin{Thm} 
\label{Main_Thm4} 
Assume $d=1$, {\rm (H1)} and {\rm (H2)}.
Let
$\Lambda \in \mathscr{S}^{\prime} (\mathbb{R})$
be positive.
Then
there exists $k \in \mathbb{Z}_{\geqslant 0}$
such that
we have
$
\int_{0}^{T}
\Vert
	\Lambda ( X_{t} )
\Vert_{p,-2k}
\mathrm{d} t
< +\infty
$
for every $p \in (1,\infty)$.
\end{Thm} 

Hence the mapping
$
(0,T] \ni t
\mapsto
\delta_{y} (X_{t}) \in
\mathbb{D}_{p}^{-2 k }
$
is Bochner integrable
in the case of $d=1$.
For multi-dimensional cases, it is sufficient to assume
$x \neq y$
in order to guarantee the Bochner integrability
(Proposition \ref{exclusion_int}).

Finally,
let
$
H_{p}^{s}(\mathbb{R}^{d})
:=
( 1 - \triangle )^{-s/2}
L_{p} (\mathbb{R}^{d}, \mathrm{d}z)
$
for $p \in (1,\infty )$, $s \in \mathbb{R}$,
which are called the Bessel potential spaces
(see \cite{Ab} or \cite{Kr} for details).
We will then apply
the It\^o formula
(Theorem \ref{Ito_Formula})
to derive the following.

\begin{Cor} 
\label{Main_Thm5} 
Assume {\rm (H1)}, \rm{(H3)} and \rm{(H4)}.
Let $p \in (1, \infty )$
and
$s \in \mathbb{R}$.
Then for each
$\Lambda \in H_{p}^{s} (\mathbb{R}^{d})$,
we have
\begin{itemize}
\item[(i)]
$\Lambda (X_{t}) \in \mathbb{D}_{p^{\prime}}^{s}$
for
$t > 0$ and $p^{\prime} \in (1,p)$;

\item[(ii)]
if $p > 2$, we further have
$
\int_{t_{0}}^{T} \Lambda (X_{t}) \mathrm{d}t
\in \mathbb{D}_{p^{\prime}}^{s+1}
$
for $t_{0} \in (0, T]$ and
$p^{\prime} \in [2,p)$.
\end{itemize}
\end{Cor} 

It might be natural to ask about the class to which
$
\int_{t_{0}}^{T} \Lambda (X_{t}) \mathrm{d}t
$
belongs
when $t_{0}=0$.
Some examples are included in
Section~\ref{application}.

The organization of the current paper is as follows:
We first review the classical Malliavin calculus in
Section~\ref{Review}
to introduce several notations.
In particular, the mapping of Watanabe's pull-back
will be extended to the space of distributions of exponential-type.
Section~\ref{sec_Bochner}
is devoted to investigate
Bochner integrability of the mapping
$
(0,T] \ni t \mapsto \Lambda ( X_{t} )
$
where $\Lambda$ is a distribution.
We will illustrate the Brownian case
with detailed computations.
The
methods
there bring
a H\"older continuity in the space variable of the local time
in the case where the stochastic differential equation is
written in a Fisk-Stratonovich symmetric form.
In
Section~\ref{Generalized_sec},
we give a definition of stochastic integrals and
formulate an It\^o formula in this distributional setting.
Corollary \ref{Main_Thm5} and some examples
will be presented in
Section~\ref{application}.
Several estimates necessary for these examples are
wrapped up in Appendix \ref{auxiliary}.

\section{Review of Malliavin Calculus}
\label{Review} 
First, we make a brief review of the
classical Malliavin calculus
on
the $d$-dimensional classical Wiener space
to introduce notations.

Let
$(W, \mathcal{F}, \mathbf{P})$
be the $d$-dimensional Wiener space on $[0,T]$,
that is,
$W$ is the space of all continuous functions
$[0, T] \to \mathbb{R}^{d}$,
$\mathcal{F}$
is the $\sigma$-field generated by
the canonical process
$W \ni w \mapsto w(t) \in \mathbb{R}^{d}$,
$0 \leqslant t \leqslant T$,
and $\mathbf{P}$ is the Wiener measure
with $\mathbf{P} ( w(0) = 0 ) = 1$.
The expectation under $\mathbf{P}$ will be denoted by $\mathbf{E}$.
The space $W$ contains the subspace $H$,
consisting of all absolutely continuous $h \in W$
with $h(0)=0$ and the square-integrable derivative.
The subspace $H$ is called the {\it Cameron-Martin subspace}
and forms a real Hilbert space
with
the inner product
$$
\langle
	h_{1}, h_{2}
\rangle_{H}
:=
\int_{0}^{T}
\langle
	\dot{h}_{1} (t),
	\dot{h}_{2} (t)
\rangle_{\mathbb{R}^{d}}
\mathrm{d}t,
\quad
h_{1}, h_{2} \in H.
$$

It is known that $L_{2} := L_{2}(W, \mathcal{F}, \mathbf{P})$ has the following
orthogonal decomposition,
called the {\it Wiener-It\^o chaos expansion}:
$$
L_{2}
=
\mathbb{R}
\oplus
\mathcal{C}_{1}
\oplus
\mathcal{C}_{2}
\oplus
\cdots ,
$$
where each $\mathcal{C}_{k}$ is a closed linear subspace
of $L_{2}$ spanned by multiple stochastic integrals
$$
\int_{
	0
	\leqslant t_{1}
	< \cdots
	<
	t_{k}
	\leqslant T
}
\langle
	\dot{h}_{1}( t_{1} )
	\otimes
	\cdots
	\otimes
	\dot{h}_{k}(t_{k} ),
	\mathrm{d}w(t_{1})
	\otimes
	\cdots
	\otimes
	\mathrm{d}w(t_{k})
\rangle_{(\mathbb{R}^{d})^{\otimes k}},
$$
for
$
h_{1}, h_{2}, \cdots , h_{k} \in H
$,
of $k$-th degree.
Each $\mathcal{C}_{k}$ is called the subspace of
{\it Wiener's homogeneous chaos of $k$-th order}.
We denote by
$J_{k}$ the orthogonal projection onto $\mathcal{C}_{k}$.
For each separable Hilbert space
$(E, \langle \bullet , \bullet \rangle_{E} )$,
$L_{p}(E)$
denotes the space of
$E$-valued $p$-th
order
integrable random variables $F$
with norm
$\Vert F \Vert_{p} = \mathbf{E} [ \vert F \vert_{E}^{p} ]^{1/p}$.
Each projection $J_{n}$ extends to
$
L_{2}(E) \cong L_{2} \otimes E \to \mathcal{C}_{n} \otimes E
$,
which is still denoted by the same symbol.

For each $s \in \mathbb{R}$ and $p \in (1,\infty )$,
a Sobolev-type space
$\mathbb{D}_{p}^{s} (E)$
(we write this $\mathbb{D}_{p}^{s}$ when $E = \mathbb{R}$)
is defined as the completion of
$
\mathcal{P}
:=
\cup_{n=1}^{\infty} \cap_{m \geqslant n}
\{
F \in L_{2}(E):
J_{m} F = 0
\}
$
under the norm $\Vert \cdot  \Vert_{{p,s}}$ defined by
$\Vert F \Vert_{{p,s}}
=
\Vert
(I-\mathcal{L})^{s/2} F
\Vert_{p}
$
for
$F \in \mathcal{P}$,
where
$\mathcal{L}$
is the Ornstein-Uhlenbeck operator on the Wiener space.
It is known that
\begin{equation}
\label{Norm} 
(I-\mathcal{L})^{s/2}F
=
\sum_{k=0}^{\infty}
(1+k)^{s/2}
J_{k}F,
\quad F \in \mathcal{P} .
\end{equation}
Note that $\mathbb{D}_{p}^{0} = L_{p}$
for $p \in (1, \infty)$, and
\begin{equation}
\label{Norm2} 
\Vert F \Vert_{{2,s}}^{2}
=
\sum_{k=0}^{\infty}
(1+k)^{s}
\Vert J_{k}F \Vert_{2}^{2},
\quad F \in \mathbb{D}_{2}^{s}.
\end{equation}
We further define
\begin{equation*}
\mathbb{D}^{\infty} (E)
:=
\bigcap_{s>0}
\bigcap_{1 < p < \infty}
\mathbb{D}_{p}^{s} (E)
\quad
\text{and}
\quad
\mathbb{D}^{-\infty} (E)
:=
\bigcup_{s<0}
\bigcup_{1 < p < \infty}
\mathbb{D}_{p}^{s} (E) .
\end{equation*}
It is known that
$(\mathbb{D}_{p}^{s}(E))^{\prime} = \mathbb{D}_{q}^{-s}(E)$
if and only if
$1/p + 1/q = 1$
(where ``$\prime$" stands for the ``continuous dual")
for each $s \in \mathbb{R}$,
the space $\mathbb{D}^{\infty} (E)$
is a complete
countably-normed
space
and $\mathbb{D}^{-\infty} (E)$ is its dual
which is called the space of
{\it generalized Wiener functionals}.
The pairing of
$\Phi \in \mathbb{D}^{-\infty} (E)$
and
$F \in \mathbb{D}^{\infty} (E)$
is written as
$
\mathbf{E} [ \Phi F ]
:=
\mbox{}_{\mathbb{D}^{-\infty} (E)}\langle
	\Phi , F
\rangle_{\mathbb{D}^{\infty} (E)}
$,
and then
$
\mathbf{E} [ \Psi ]
=
\mbox{}_{\mathbb{D}^{-\infty}}\langle
	\Psi , 1
\rangle_{\mathbb{D}^{\infty}}
$
is called the
{\it generalized expectation}
of $\Psi \in \mathbb{D}^{-\infty}$.

One can define a (continuous) linear operator
$
D : \mathbb{D}^{-\infty} (E)
\to
\mathbb{D}^{-\infty} ( E \otimes H )
$
such that
(a)
each restriction
$
D : \mathbb{D}_{p}^{s+1} (E)
\to
\mathbb{D}_{p}^{s} ( E \otimes H )
$
and is continuous for every
$s \in \mathbb{R}$ and $p \in (1,\infty)$,
and
(b)
we have
$
\langle
	DF ,
	e \otimes h
\rangle_{E \otimes H}
=
\langle
	D_{h} F , e
\rangle_{E}
$
for $e \in E$, $h \in H$ and $F \in \mathcal{P}$,
where $D_{h}F$ is given by
\begin{equation}
\label{Dif}
\langle
( D_{h}F )(w), e
\rangle_{E}
=
\lim_{\varepsilon \to 0}
\frac{1}{\varepsilon}
\langle
	F( w+h ) - F(w) , e
\rangle_{E}
\quad
\text{for
$w \in W$.}
\end{equation}
The differential operator $D_{h}$ in (\ref{Dif}) is well-defined
for
almost all
$w$ because of the so-called
{\it Cameron-Martin theorem}.
There also exists a (continuous) linear operator
$
D^{*} :
\mathbb{D}^{-\infty} (E \otimes H)
\to
\mathbb{D}^{-\infty}(E)
$
such that
(a)$\mbox{}^{*}$
each restriction
$
D^{*} : \mathbb{D}_{p}^{s+1} ( E \otimes H )
\to
\mathbb{D}_{p}^{s} (E)
$
and is continuous for every
$s \in \mathbb{R}$ and $p \in (1,\infty)$,
and
(b)$\mbox{}^{*}$
we have
\begin{equation*}
D^{*} ( G \otimes h )
=
- D_{h}G
+
\int_{0}^{T}
\langle
	\dot{h}(t),
	\mathrm{d}w(t)
\rangle_{\mathbb{R}^{d}}
G
\end{equation*}
for $h \in H$, $G \in \mathbb{D}_{2}^{1}(E)$.
These operators are related as follows:
For $F,G \in \mathbb{D}_{2}^{1}(E)$
and
$h \in H$, it holds that
\begin{equation*}
\mathbf{E}
[
	\langle DF , G \otimes h  \rangle_{E \otimes H}
]
=
\mathbf{E}
[
	\langle
	F,
	D^{*} ( G \otimes h  )
	\rangle_{E}
] .
\end{equation*}

Let $\mathscr{S}(\mathbb{R}^{d})$ be
the real Schwartz space of rapidly decreasing
$C^{\infty}$-functions on $\mathbb{R}^{d}$.
We denote by
$\mathscr{S}_{2k}(\mathbb{R}^{d})$, $k \in \mathbb{Z}$
the completion of $\mathscr{S}(\mathbb{R}^{d})$
by the norm
$$
\vert \phi \vert_{2k}
:=
\vert
	\big( 1 + x^{2} - \triangle /2 \big)^{k}
	\phi
\vert_{\infty},
\quad
\phi = \phi (x) \in \mathscr{S} (\mathbb{R}^{d}),
$$
where
$
\triangle
=
\sum_{i=1}^{d} \partial^{2} / (\partial x_{i})^{2}
$
and
$
\vert \phi \vert_{\infty}
=
\sup_{x \in \mathbb{R}^{d}}
\vert \phi (x) \vert
$.

\begin{Def} 
\label{non-degeneracy} 
\begin{itemize}
\item[(i)]
A Wiener functional
$
F = (F^{1}, \cdots , F^{d})
\in
\mathbb{D}^{\infty}(\mathbb{R}^{d})
$
is said to be
{\it non-degenerate}
if
$
\Vert
\det
(
\langle
	DF^{i}, DF^{j}
\rangle_{H}
)_{i,j}^{-1}
\Vert_{p}
< \infty
$
for any $p \in (1, \infty )$.

\item[(ii)]
A family of Wiener functionals
$
F_{\alpha} = ( F_{\alpha}^{1}, \cdots , F_{\alpha}^{d} )
\in
\mathbb{D}^{\infty} (\mathbb{R}^{d})
$,
$\alpha \in I$,
where $I$ is an index set,
is said to be
{\it uniformly non-degenerate}
if
$
\sup_{\alpha \in I}
\Vert
\det
(
\langle
	DF_{\alpha}^{i}, DF_{\alpha}^{j}
\rangle_{H}
)_{i,j}^{-1}
\Vert_{p}
< \infty
$
for any $p \in (1, \infty )$.
\end{itemize}
\end{Def} 

If
$F \in \mathbb{D}^{\infty}(\mathbb{R}^{d})$ is
non-degenerate,
then the mapping
$
\mathscr{S}( \mathbb{R}^{d} ) \ni \phi
\mapsto
\phi (F) \in \mathbb{D}^{\infty}
$
extends uniquely to a mapping
$
\mathscr{S}^{\prime}( \mathbb{R}^{d} ) \ni \Lambda
\mapsto
\Lambda (F) \in
\cup_{s>0}
\cap_{1 < p < \infty}
\mathbb{D}_{p}^{-s}
$
such that each restriction maps
$
\mathscr{S}_{-2k} (\mathbb{R}^{d}) \to \mathbb{D}_{p}^{-2k}
$
and is continuous for every
$k \in \mathbb{Z}$
and
$p \in (1, \infty)$
(see e.g., \cite[Chapter~V, Section~9]{IW}).
The generalized Wiener functional
$\Lambda (F)$
is called the
{\it pull-back} of
$\Lambda \in \mathscr{S}^{\prime} (\mathbb{R}^{d})$
by $F \in \mathbb{D}^{\infty}(\mathbb{R}^{d})$.

For $\Lambda \in \mathscr{S}^{\prime} (\mathbb{R})$,
we denote by $\Lambda^{(n)}$
the $n$-th distributional derivative of $\Lambda$.

\begin{Lem} 
\label{delta_unif} 
Let $y \in \mathbb{R}$ and
$\delta_{y}$ be the Dirac delta-function at $y$.
Then $\delta_{y} \in \mathscr{S}_{-2} (\mathbb{R})$,
$\delta_{y}^{(2k)} \in \mathscr{S}_{-2(k+1)} (\mathbb{R})$
for $k \in \mathbb{N}$,
and
$
\sup_{a \in \mathbb{R}}
\vert
	\delta_{a}
\vert_{-2} < +\infty
$.
\end{Lem} 
\begin{proof} 
It is well known that
$\delta_{y} \in \mathscr{S}_{-2} (\mathbb{R})$
and
$\delta_{y}^{(2k)} \in \mathscr{S}_{-2(k+1)} (\mathbb{R})$
for $k \in \mathbb{N}$
(see \cite[Chapter V, section 9, Lemma 9.1, p.380]{IW}).
It is also known that
\begin{equation*}
\big(
	( 1+ x^{2} - \triangle / 2 )^{-1} \delta_{y}
\big)
(x)
\leqslant
\frac{1}{2\pi}
\int_{-\infty}^{+\infty}
\frac{
	\mathrm{e}^{i \xi (x-y)}
}{
	(1+\frac{\xi^{2}}{2})
}
\mathrm{d} \xi
\end{equation*}
for any $x,y \in \mathbb{R}$, from which,
we easily conclude that
\begin{equation*}
\sup_{a \in \mathbb{R}}
\vert \delta_{a} \vert_{-2}
\leqslant
\frac{1}{2\pi}
\int_{-\infty}^{+\infty}
\frac{
	\mathrm{d} \xi
}{
	(1+\frac{\xi^{2}}{2})
}
< +\infty .
\end{equation*}
\end{proof} 

\subsection{Slight extension to exponential-type distributions}

It will be convenient to extend the pull-back procedure
from Schwartz distribution space
to the space of all distributions of exponential-type.
Let
$\partial_{i} := \partial / (\partial x_{i})$,
$i=1, 2, \cdots , d$.

\begin{Def}[Hasumi \cite{Ha}] 
We say $\phi \in C^{\infty}(\mathbb{R}^{d})$
belongs to $\mathscr{E}(\mathbb{R}^{d})$ if
$
\sup_{x \in \mathbb{R}}
\vert
	\exp ( p \vert x \vert )
	\partial_{1}^{k_{1}}
	\cdots
	\partial_{d}^{k_{d}}
	\phi (x)
\vert
$
$
< +\infty
$
for any
$p \in \mathbb{Z}_{\geqslant 0}$
and
$k_{1}, \cdots , k_{d}\in \mathbb{Z}_{\geqslant 0}$.
\end{Def} 

Semi-norms on $\mathscr{E}(\mathbb{R}^{d})$, defined by
$$
\vert
	\phi
\vert_{p}
:=
\sup_{
	\substack{
	k_{1} , \cdots , k_{d} \in \mathbb{Z}_{\geqslant 0}: \\
	0 \leqslant k_{1}+\cdots + k_{d} \leqslant p
	}
}
\sup_{x \in \mathbb{R}^{d}}
\vert
	\exp ( p \vert x \vert )
	\partial_{1}^{k_{1}}
	\cdots
	\partial_{d}^{k_{d}}
	\phi (x)
\vert ,
\quad
p=0,1,2,\cdots
$$
make $\mathscr{E}(\mathbb{R}^{d})$ a locally convex metrizable space
and induces continuous inclusions
\begin{equation*}
\mathscr{D} ( \mathbb{R}^{d} )
\hookrightarrow
\mathscr{E} ( \mathbb{R}^{d} )
\hookrightarrow
\mathscr{S} ( \mathbb{R}^{d} )
\quad
\text{and}
\quad
\mathscr{S}^{\prime} ( \mathbb{R}^{d} )
\hookrightarrow
\mathscr{E}^{\prime} ( \mathbb{R}^{d} )
\hookrightarrow
\mathscr{D}^{\prime} ( \mathbb{R}^{d} )
\end{equation*}
where
$\mathscr{D}^{\prime}(\mathbb{R}^{d})$
is the space of all distributions on $\mathbb{R}^{d}$
with the test function space $\mathscr{D}(\mathbb{R}^{d})$,
and $\mathscr{E}^{\prime} (\mathbb{R}^{d})$
is the continuous dual of $\mathscr{E}(\mathbb{R}^{d})$.
Elements in $\mathscr{E}^{\prime}(\mathbb{R}^{d})$ are referred as
{\it distributions of exponential-type}.
The following is known
(see Hasumi \cite[Proposition 3]{Ha}).

\begin{Thm}
\label{Hasu-Thm} 
For any $\Lambda \in \mathscr{E}^{\prime} (\mathbb{R}^{d})$,
there exist $k \in \mathbb{Z}_{\geqslant 0}$
and a bounded continuous function
$f : \mathbb{R}^{d} \to \mathbb{R}$
such that
$$
\Lambda
=
\frac{\partial^{kd}}{\partial x_{1}^{k} \cdots \partial x_{d}^{k}}
[
	\exp ( k \vert x \vert )
	f (x)
].
$$
Here, the derivatives are understood in the sense of
distributional derivatives.
\end{Thm} 

\begin{Rm} 
The spaces
$\mathscr{E}(\mathbb{R}^{d})$
and
$\mathscr{E}^{\prime} (\mathbb{R}^{d})$
are denoted by $H$ and $\Lambda_{\infty}$
respectively in Hasumi \cite{Ha}.
The space
$\mathscr{E}^{\prime} (\mathbb{R}^{d}) = \Lambda_{\infty}$
is
introduced by
Silva \cite{Se}
in the study of
the Fourier transform of $\Lambda_{\infty}$,
called {\it ultra-distributions}
(see
Silva \cite{Se},
Hasumi \cite{Ha}
and
Yoshinaga \cite{Yo}).
\end{Rm} 

Define
$
e_{k} (x) := \prod_{i=1}^{d} \cosh (k x_{i})
$
for
$k \in \mathbb{Z}_{\geqslant 0}$
and
$
x = (x_{1}, \cdots , x_{d}) \in \mathbb{R}^{d}
$.
We postpone the proof of the next proposition
after Definition~\ref{pull-back(exp)}.

\begin{Prop} 
\label{exp-pull} 
Suppose that
$
F_{\alpha} = ( F_{\alpha}^{1}, \cdots , F_{\alpha}^{d} )
\in \mathbb{D}^{\infty}(\mathbb{R}^{d})
$,
$\alpha \in I$,
where $I$ is an index set,
are uniformly non-degenerate
and satisfy
\begin{equation}
\label{exp-moment} 
M_{r} :=
\sup_{\alpha \in I}
\mathbf{E}
[
	\exp ( r \vert F_{\alpha }\vert_{\mathbb{R}^{d}} )
]
< +\infty
\quad
\text{for each $r > 0$.}
\end{equation}
Then for any $p \in (1, \infty)$,
$k \in \mathbb{Z}_{\geqslant 0}$
and
$r>0$,
there exists a constant
$
c = c(q,k,r,
M_{r}
) >0
$,
where $1/p +1/q = 1$,
such that
\begin{equation*}
\begin{split}
\big\Vert
	\frac{
		\partial^{kd} ( e_{r} \phi )
	}{
		\partial x_{1}^{k} \cdots \partial x_{d}^{k}
	}
	( F_{\alpha} )
\big\Vert_{p,-kd}
\leqslant
c
\sup_{x \in \mathbb{R}^{d}} \vert \phi (x) \vert
\end{split}
\end{equation*}
for all $\phi \in \mathscr{S}(\mathbb{R}^{d})$
and
$\alpha \in I$.
\end{Prop} 

Now, let $F \in \mathbb{D}^{\infty} (\mathbb{R}^{d})$
be non-degenerate.
Take
$
\Lambda
=
\partial_{1}^{k} \cdots \partial_{d}^{k}
[ \exp (k \vert x \vert) f(x) ]
\in \mathscr{E}^{\prime} (\mathbb{R}^{d})
$
(where $f$ is the one associated to $\Lambda$ in Theorem \ref{Hasu-Thm})
and assume $k \geqslant 1$.
Let
$\varepsilon \in (0,1)$
be arbitrary
and put $r := k+\varepsilon > 0$.
Define $\phi \in C_{0}(\mathbb{R}^{d})$
(the space of continuous functions on $\mathbb{R}^{d}$
vanishing at infinity)
by
$
\phi (x)
:=
(
	\mathrm{e}^{ k \vert x \vert }
	/
	\prod_{i=1}^{d} \cosh (rx_{i})
)
f(x)
$,
so that now we have
$
\Lambda
=
\partial_{1}^{k} \cdots \partial_{d}^{k}
( e_{r} \phi )
=
\partial_{1}^{k} \cdots \partial_{d}^{k}
[ ( \prod_{i=1}^{d} \cosh ( r x_{i} ) ) \phi (x) ]
$.
Take any sequence
$\phi_{n} \in \mathscr{S} (\mathbb{R}^{d})$,
$n \in \mathbb{N}$
such that
$
\vert \phi_{n} - \phi \vert_{\infty} \to 0
$.
Then Proposition \ref{exp-pull} tells us that
$
\lim_{n \to \infty}
[ \partial_{1}^{k} \cdots \partial_{d}^{k} ( e_{r} \phi_{n} ) ]
(F)
$
exists in $\mathbb{D}_{p}^{-kd}$
for each $p \in (1,\infty )$.
The limit does not depend on the choice of
$\varepsilon > 0$
and the sequence $\phi_{n} \in \mathscr{S}(\mathbb{R}^{d})$.
Under these notations,
we put the following.

\begin{Def} 
\label{pull-back(exp)} 
We denote the limit by $\Lambda (F)$ and call the
{\it pull-back of $\Lambda \in \mathscr{E}^{\prime} (\mathbb{R}^{d})$ by $F$}.
\end{Def} 

\begin{proof}[Proof of Proposition \ref{exp-pull}] 
Let
$p \in (1,\infty)$,
$k \in \mathbb{Z}_{\geqslant 0}$,
$r>0$,
$\alpha \in I$,
$\phi \in \mathscr{S}(\mathbb{R}^{d})$
and
$J \in \mathbb{D}^{\infty}$ be arbitrary.
Then
\begin{equation*}
\begin{split}
\mathbf{E}
[
	\big(
	\partial_{1}^{k} \cdots \partial_{d}^{k} ( e_{r} \phi )
	\big)
	( F_{\alpha} )
	J
]
=
\mathbf{E}
[
	\Big(
	\prod_{i=1}^{d}
	\cosh ( r F_{\alpha}^{i} )
	\Big)
	\phi ( F_{\alpha} )
	l_{\alpha} (J)
]
\end{split}
\end{equation*}
where $l_{\alpha} (J) \in \mathbb{D}^{\infty}$ is of the form
$$
l_{\alpha} (J)
=
\sum_{j=0}^{kd}
\langle
	P_{j} (w), D^{j} J
\rangle_{H^{\otimes j}}
$$
for some
$P_{j}(w) \in \mathbb{D}^{\infty} ( H^{\otimes j} )$,
$j=0,1, \cdots , kd$
which are polynomials in $F_{\alpha}$,
its derivatives up to the order $kd$,
and
$
\det ( \langle DF_{\alpha}^{i}, DF_{\alpha}^{j} \rangle_{H} )_{ij}^{-1}
$.
Take $q^{\prime} \in (1, q)$,
where $1/p + 1/q = 1$.
Since
$\{ F_{\alpha} \}_{\alpha \in I}$
is uniformly non-degenerate,
there exists $c_{0} > 0$ such that
$$
\Vert
	l_{\alpha} (J)
\Vert_{q^{\prime}}
\leqslant
c_{0} \Vert J \Vert_{q,k}
\quad
\text{for all $\alpha \in I$ and $J \in \mathbb{D}^{\infty}$.}
$$
Therefore by taking
$p^{\prime} \in (1, \infty)$
such that
$1/p^{\prime} + 1/q^{\prime} = 1$,
we have
\begin{equation*}
\begin{split}
&
\big\vert
\mathbf{E}
[
	\Big(
	\prod_{i=1}^{d}
	\cosh ( r F_{\alpha}^{i} )
	\Big)
	\phi ( F_{\alpha} )
	l_{\alpha} (J)
]
\big\vert
\leqslant
c_{0}^{\prime}
\vert \phi \vert_{\infty}
\Vert
\exp \big( r \vert F_{\alpha} \vert \big)
\Vert_{p^{\prime}}
\Vert J \Vert_{q, k},
\end{split}
\end{equation*}
for some constant $c_{0}^{\prime} > 0$, which implies
\begin{equation*}
\begin{split}
\Vert
	( e_{k} \phi )^{(k)}
	( F_{\alpha} )
\Vert_{p,-k}
\leqslant
c_{0}^{\prime}
\vert \phi \vert_{\infty}
\Vert \exp ( r \vert F_{\alpha} \vert ) \Vert_{p^{\prime}}
\leqslant
c_{0}^{\prime}
\sup_{\alpha \in I}
\Vert \exp ( r \vert F_{\alpha} \vert ) \Vert_{p^{\prime}}
\vert \phi \vert_{\infty}.
\end{split}
\end{equation*}
\end{proof} 

\begin{Cor} 
\label{exp-pull-conti} 
Suppose that
$F_{\alpha} \in \mathbb{D}^{\infty}(\mathbb{R}^{d})$
for
$\alpha \in I \subset \mathbb{R}$,
where $I$ is an index set,
satisfy
\begin{itemize}
\item[(i)]
$\{ F_{\alpha} \}_{\alpha \in I}$
is uniformly non-degenerate;

\item[(ii)]
$
\sup_{\alpha \in I}
\mathbf{E}
[
	\exp ( r \vert F_{\alpha }\vert_{\mathbb{R}^{d}} )
]
< +\infty
$
for each $r > 0$;

\item[(iii)]
the mapping
$
I \ni \alpha
\mapsto
F_{\alpha} \in \mathbb{D}^{\infty}(\mathbb{R}^{d})
$
is continuous.
\end{itemize}
Then for any $p \in (1, \infty )$ and
$
\Lambda
=
\partial_{1}^{k} \cdots \partial_{d}^{k}
[ \exp (k\vert x \vert) f (x)]
\in \mathscr{E}^{\prime} (\mathbb{R}^{d})
$
{\rm ($f$ is the one associated to $\Lambda$ in Theorem \ref{Hasu-Thm})},
the mapping
$$
I \ni \alpha \mapsto \Lambda (F_{\alpha}) \in \mathbb{D}_{p}^{-kd}
$$
is continuous.
\end{Cor} 
\begin{proof} 
Let
$p \in (1,\infty)$,
$\alpha \in I$
and
$\varepsilon > 0$
be arbitrary.
Suppose that
$
\Lambda
=
\partial_{1}^{k} \cdots \partial_{d}^{k}
[
	( \prod_{i=1}^{d} \cosh ( r x_{i} ) )
	\phi (x)
]
$
where $r>k \geqslant 1$ and $\phi \in C_{0}(\mathbb{R}^{d})$.
Then by Proposition \ref{exp-pull},
there exists $\psi \in \mathscr{S}(\mathbb{R}^{d})$
such that
$
\Vert
	\Lambda (F_{\beta})
	-
	[ \partial_{1}^{k} \cdots \partial_{d}^{k}
	( e_{r} \psi ) ] (F_{\beta})
\Vert_{p, -k} < \varepsilon /3
$
for every $\beta \in I$.
Furthermore,
by the condition
(ii) and
(iii),
there exists $\delta > 0$ such that
$
\Vert
	[ \partial_{1}^{k} \cdots \partial_{d}^{k}
	( e_{r} \psi ) ] ( F_{\beta} )
	-
	[ \partial_{1}^{k} \cdots \partial_{d}^{k}
	( e_{r} \psi ) ] ( F_{\alpha} )
\Vert_{p}
< \varepsilon /3
$
if $\vert \beta - \alpha \vert < \delta$.
Hence, for each $\beta \in I$
with
$\vert \beta - \alpha \vert < \delta$,
\begin{equation*}
\begin{split}
\Vert
	\Lambda ( F_{\alpha} )
	-
	\Lambda ( F_{\beta} )
\Vert_{p,-k}
&\leqslant
\Vert
	\Lambda ( F_{\alpha} )
	-
	[ \partial_{1}^{k} \cdots \partial_{d}^{k}
	( e_{r} \psi ) ] ( F_{\alpha} )
\Vert_{p,-k} \\
&\hspace{10mm}+
\Vert
	[ \partial_{1}^{k} \cdots \partial_{d}^{k}
	( e_{r} \psi ) ] ( F_{\alpha} )
	-
	[ \partial_{1}^{k} \cdots \partial_{d}^{k}
	( e_{r} \psi ) ] ( F_{\beta} )
\Vert_{p} \\
&\hspace{25mm}+
\Vert
	[ \partial_{1}^{k} \cdots \partial_{d}^{k}
	( e_{r} \psi ) ] ( F_{\beta} )
	-
	\Lambda ( F_{\beta} )
\Vert_{p,-k} \\
&< \varepsilon .
\end{split}
\end{equation*}
The case of $k=0$ is clear.
\end{proof} 

If we assume
{\rm (H1)}, {\rm (H2)} and {\rm (H4)},
then
we have the Gaussian estimate
for the transition density of $X$
(the special case $n=0$ in Lemma~\ref{IbP}--(ii))
and
\begin{equation}
\label{referee-1} 
\begin{split}
\sup_{t \in K}
\mathbf{E}
[
	\exp ( r \vert X(t,x,w) \vert )
]
< +\infty
\end{split}
\end{equation}
for all $r>0$ and compact set $K \subset (0, \infty )$.
Hence by Corollary \ref{exp-pull-conti},
for any
$
\Lambda
=
\partial_{1}^{k} \cdots \partial_{d}^{k}
[ \exp ( k \vert x \vert ) f(x) ]
\in \mathscr{E}^{\prime} (\mathbb{R}^{d})
$,
the mapping
$
(0, \infty ) \ni t
\mapsto
\Lambda ( X(t,x,w) ) \in \mathbb{D}_{p}^{-kd}
$
is continuous for every $p \in (1, \infty )$.

\section{Bochner Integrability of Pull-Backs by Diffusions}
\label{sec_Bochner} 
Let $w = (w(t))_{t \geqslant 0}$
be a $d$-dimensional Wiener process
with $w(0) = 0$.
Let
$
\sigma : \mathbb{R}^{d} \to \mathbb{R}^{d} \otimes \mathbb{R}^{d}
$,
$b: \mathbb{R}^{d} \to \mathbb{R}^{d}$.
We consider the following stochastic differential equation
\begin{equation}
\label{SDE} 
\mathrm{d} X_{t}
=
\sigma (X_{t}) \mathrm{d} w(t)
+
b ( X_{t} ) \mathrm{d} t,
\quad
X_{0} = x \in \mathbb{R}^{d}.
\end{equation}

In this
section,
we assume conditions {\rm (H1)} and {\rm (H2)}.
Under these conditions,
the equation (\ref{SDE}) admits
a unique strong solution.
We denote by
$\{ X(t,x,w) \}_{t \geqslant 0}$ the unique strong solution
$X = (X_{t})_{t \geqslant 0}$ to (\ref{SDE}).
Furthermore, for each $t>0$ and $x \in \mathbb{R}$,
we have
$
X(t,x,w) \in \mathbb{D}^{\infty} (\mathbb{R}^{d})
$
and is non-degenerate.
Henceforward
for $t>0$,
one can define the pull-back $\Lambda (X(t,x,w))$
of $\Lambda \in \mathscr{S}^{\prime} ( \mathbb{R}^{d} )$
by $X(t,x,w)$ as an element of
$\cup_{s>0} \cap_{1 < p < \infty} \mathbb{D}_{p}^{-s}$
(Recall things just after Definition \ref{non-degeneracy}).

Fix $T>0$ be arbitrary.
By the condition {\rm (H2)},
we further know that
$$
\sup_{t \in K}
\Vert
	\det
	(
	\langle
	DX_{t}^{i},
	DX_{t}^{j}
	\rangle_{H}
	)_{ij}^{-1}
\Vert_{p}
<
+\infty
\quad
\text{for $1<p<\infty$}
$$
holds for each closed interval $K \subset (0,T]$,
which implies that for each
$k \in \mathbb{Z}_{\geqslant 0}$,
$p \in (1,\infty)$
and
$\Lambda \in \mathscr{S}_{-2k}(\mathbb{R}^{d})$,
the mapping
$
(0,T] \ni t
\mapsto
\Lambda (X(t,x,w)) \in \mathbb{D}_{p}^{-2k}
$
is continuous
(see e.g. \cite[Remark 2.2]{Wa87}).
In particular,
the mapping
$
[t_{0} ,T] \ni t
\mapsto
\Lambda (X(t,x,w)) \in \mathbb{D}_{p}^{-2k}
$
is Bochner integrable for each
$t_{0} > 0$,
and hence the Bochner integral
$
\int_{t_{0}}^{T}
\Lambda (X(t,x,w))
\mathrm{d}t
$
makes sense as an element in $\mathbb{D}_{p}^{-2k}$
for each $t_{0} > 0$.
If we assume {\rm (H4)} additionally,
then
analogous results follow also for
$\Lambda \in \mathscr{E}^{\prime}(\mathbb{R}^{d})$
by virtue of
the Gaussian estimates
for the transition density function of
$(X_{t})_{t\geqslant 0}$
(or one can refer Lemma~\ref{IbP} below)
and
Corollary \ref{exp-pull-conti}.

Now our problem is the Bochner integrability on $(0,T]$,
i.e.,
whether
$
\int_{0}^{T}
\Vert \Lambda (X(t,x,w)) \Vert_{p,-2k}
\mathrm{d}t
< +\infty
$
holds or not.
The main results of this
section
are
Theorem \ref{pos_Bochner}
and
Proposition \ref{exclusion_int}.
Before seeing this,
we shall start with the Brownian case,
which would be an introductory example.
\subsection{Bochner integrability of $\delta_{0} (w(t))$}

\begin{Prop} 
\label{Brown-Bochner} 
Let $d=1$,
$\Lambda \in \mathscr{S}^{\prime} (\mathbb{R})$
and
$s \in \mathbb{R}$.
If $\Lambda (w(T)) \in \mathbb{D}_{2}^{s}$ then
the mapping
$$
(0, T] \ni t \mapsto
\sqrt{\frac{T}{t}}
\Lambda \Big( \sqrt{\frac{T}{t}} w(t) \Big)
\in \mathbb{D}_{2}^{s}
$$
is Bochner integrable in $\mathbb{D}_{2}^{s}$
and we have
$$
\int_{0}^{T}
\sqrt{\frac{T}{t}}
\Lambda \Big( \sqrt{\frac{T}{t}} w(t) \Big)
\mathrm{d}t
\in \mathbb{D}_{2}^{s+1}.
$$
\end{Prop} 

\begin{Rm} 
Hence the Bochner integral poses a sort of
``smoothing effect", in the sense of
raising the regularity-index $s$,
which might be a common understanding for most of us.
\end{Rm} 

\begin{proof} 
Recall the integration by parts formula
$$
\mathbf{E}
[
	\Lambda^{\prime} (w(t)) H_{n} \Big( \frac{w(t)}{\sqrt{t}} \Big)
]
=
t^{-1/2}
\mathbf{E}
[
	\Lambda (w(t)) H_{n+1} \Big( \frac{w(t)}{\sqrt{t}} \Big)
]
$$
where
$\Lambda^{\prime} = \partial \Lambda$
is the distributional derivative of $\Lambda$,
$H_{n} := \partial^{*n}1$
is the $n$-th Hermite polynomial
and
$\partial^{*} = - \partial + x $.
The family
$\{ \frac{1}{\sqrt{n!}} H_{n} \}_{n=0}^{\infty}$
forms a
complete orthonormal base
of
$
L_{2}
(
	\mathbb{R},
	(2\pi)^{-1/2} \mathrm{e}^{-x^{2}/2} \mathrm{d}x
)
$
(see Lemma \ref{Hermite-facts} in Appendix \ref{auxiliary}).

Then $\Lambda ( (T/t)^{1/2} w(t) )$ has the Fourier expansion
(the Wiener-It\^o chaos expansion)
$$
\Lambda \big( \sqrt{\frac{T}{t}} w(t) \big)
=
\sum_{n=0}^{\infty}
\frac{1}{n!}
\mathbf{E}
[
	\Lambda \big( \sqrt{\frac{T}{t}} w(t) \big)
	H_{n} \Big( \frac{w(t)}{\sqrt{t}} \Big)
]
H_{n} \Big( \frac{w(t)}{\sqrt{t}} \Big) ,
$$
and hence
\begin{equation*}
\begin{split}
&
\Vert
	\Lambda \big( \sqrt{\frac{T}{t}} w(t) \big)
\Vert_{2,s}^{2}
=
\sum_{n=0}^{\infty}
\frac{(1+n)^{s}}{n!}
\mathbf{E}
[
	\Lambda \big( \sqrt{\frac{T}{t}} w(t) \big)
	H_{n} \Big( \frac{w(t)}{\sqrt{t}} \Big)
]^{2}
\end{split}
\end{equation*}
Here we have
\begin{equation*}
\begin{split}
&
\mathbf{E}
[
	\Lambda \big( \sqrt{\frac{T}{t}} w(t) \big)
	H_{n} \Big( \frac{w(t)}{\sqrt{t}} \Big)
]
=
\mathbf{E}
[
	\Lambda ( w(T) )
	H_{n} \Big( \frac{w(T)}{\sqrt{T}} \Big)
] .
\end{split}
\end{equation*}
Therefore
\begin{equation*}
\begin{split}
&
\Vert
	\sqrt{\frac{T}{t}}
	\Lambda \big( \sqrt{\frac{T}{t}} w(t) \big)
\Vert_{2,s}^{2}
=
\frac{T}{t}
\sum_{n=0}^{\infty}
\frac{(1+n)^{s}}{n!}
\mathbf{E}
[
	\Lambda (w(T))
	H_{n} \Big( \frac{w(T)}{\sqrt{T}} \Big)
]^{2}
=
\frac{T}{t}
\Vert \Lambda (w(T)) \Vert_{2,s}^{2} ,
\end{split}
\end{equation*}
that is, we get
$
\Vert
	(T/t)^{1/2}
	\Lambda ( (T/t)^{1/2} w(t) )
\Vert_{2,s}
=
( T/t )^{1/2}
\Vert \Lambda (w(T)) \Vert_{2,s}
$.
Since
$$
\int_{0}^{T}
\big\Vert
	\sqrt{\frac{T}{t}}
	\Lambda \big( \sqrt{\frac{T}{t}} w(t) \big)
\big\Vert_{2,s}
\mathrm{d}t
=
T^{1/2}
\Vert \Lambda (w(T)) \Vert_{2,s}
\int_{0}^{T}
t^{-1/2}
\mathrm{d}t
< +\infty ,
$$
the function
$
(0,T] \ni t \mapsto
(T/t)^{1/2}
\Lambda ( (T/t)^{1/2} w(t) )
\in \mathbb{D}_{2}^{s}
$
is Bochner integrable and
$
\int_{0}^{T}
(T/t)^{1/2}
\Lambda ( (T/t)^{1/2} w(t) )
\mathrm{d}t \in \mathbb{D}_{2}^{s} .
$

Next we show
\begin{equation}
\label{raised} 
\int_{0}^{T}
\sqrt{\frac{T}{t}}
\Lambda
\big(
	\sqrt{\frac{T}{t}}
	w(t)
\big)
\mathrm{d}t \in \mathbb{D}_{2}^{s+1} .
\end{equation}
For this, we note that
\begin{equation*}
\begin{split}
&
\int_{0}^{T}
\sqrt{\frac{T}{t}}
\Lambda
\big( \sqrt{\frac{T}{t}} w(t) \big)
\mathrm{d}t
=
\sum_{n=0}^{\infty}
\frac{1}{n!}
\mathbf{E}
[
	\Lambda (w(T))
	H_{n} \Big( \frac{w(T)}{\sqrt{T}} \Big)
]
\int_{0}^{T}
\sqrt{\frac{T}{t}}
H_{n} \Big( \frac{w(t)}{\sqrt{t}} \Big)
\mathrm{d}t
\end{split}
\end{equation*}
is the chaos expansion for 
$
\int_{0}^{T}
(T/t)^{1/2}
\Lambda
\big( (T/t)^{1/2} w(t) \big)
\mathrm{d}t
$
(more precisely, we have used Corollary \ref{commute} below).
We shall focus on the $L_{2}$-norm of the last factor:
\begin{equation*}
\begin{split}
&
\mathbf{E}
[
\Big\{
\int_{0}^{T}
\frac{1}{\sqrt{t}}
H_{n} \Big( \frac{w(t)}{\sqrt{t}} \Big)
\mathrm{d}t
\Big\}^{2}
] \\
&=
2
\int_{0 < t < s < T}
\frac{1}{\sqrt{ts}}
\mathbf{E}
[
H_{n} \Big( \frac{w(t)}{\sqrt{t}} \Big)
H_{n} \Big( \frac{w(s)}{\sqrt{s}} \Big)
]
\mathrm{d}t \mathrm{d}s \\
&=
2
\int_{0 < t < s < T}
\frac{1}{\sqrt{ts}}
\mathbf{E}
[
H_{n} \Big( \frac{w(t)}{\sqrt{t}} \Big)
H_{n}
\Big(
	\sqrt{\frac{t}{s}}
	\frac{
		w(t)
	}{\sqrt{t}}
	+
	\frac{w(s)-w(t)}{\sqrt{s}}
\Big)
]
\mathrm{d}t \mathrm{d}s \\
&=
2
\int_{0 < t < s < T}
\frac{1}{\sqrt{ts}}
n! \Big( \frac{t}{s} \Big)^{n/2}
\mathrm{d}t \mathrm{d}s \\
&=
2(n!)
\int_{0 < t < s < T}
t^{
	(n-1)/2
}
s^{-(n+1)/2}
\mathrm{d}t \mathrm{d}s
=
\frac{4T}{n+1} n!.
\end{split}
\end{equation*}
In the third equality, we have used the integration by parts formula
and $H_{n}^{\prime} = n H_{n-1}$.
Hence we have
\begin{equation*}
\begin{split}
&
\Vert
	\int_{0}^{T}
	\sqrt{\frac{T}{t}}
	\Lambda \big( \sqrt{\frac{T}{t}} w(t) \big)
	\mathrm{d}t
\Vert_{2,s+1}^{2} \\
&=
T
\sum_{n=0}^{\infty}
\frac{ ( 1+n )^{s+1} }{ (n!)^{2} }
\mathbf{E}
[
	\Lambda (w(T))
	H_{n} \Big( \frac{w(T)}{\sqrt{T}} \Big)
]^{2}
\mathbf{E}
[
\Big\{
\int_{0}^{T}
\frac{1}{\sqrt{t}}
H_{n} \Big( \frac{w(t)}{\sqrt{t}} \Big)
\mathrm{d}t
\Big\}^{2}
] \\
&=
4T^{2}
\sum_{n=0}^{\infty}
\frac{ (1+n)^{s} }{ n! }
\mathbf{E}
[
	\Lambda (w(T))
	H_{n} \Big( \frac{w(T)}{\sqrt{T}} \Big)
]^{2}
=
4T^{2}
\Vert \Lambda (w(T))  \Vert_{2,s}^{2}
< +\infty ,
\end{split}
\end{equation*}
which proves (\ref{raised}).
\end{proof} 

By
Nualart-Vives \cite[Section 2]{NuVi}
and
Watanabe \cite{Wa94},
it is known that
$$
\delta_{0} (w(t))
\in \mathbb{D}_{2}^{(-1/2)-}
\quad
\text{but}
\quad
\delta_{0} (w(t)) \notin \mathbb{D}_{2}^{-1/2}
\quad
\text{for $t > 0$,}
$$
where
$
\mathbb{D}_{2}^{s-}
:=
\cap_{\alpha <s} \mathbb{D}_{2}^{\alpha}
$.
From this fact and Proposition \ref{Brown-Bochner},
we reached the following result by
Nualart-Vives \cite{NuVi}
and
Watanabe \cite{Wa94}.

\begin{Cor} 
If $d=1$, we have
$\displaystyle
\int_{0}^{T}
\delta_{0} ( w(t) )
\mathrm{d}t
\in \mathbb{D}_{2}^{(1/2)-}
$.
\end{Cor} 

\subsection{Bochner integrability of $\Lambda (X_{t})$ where $\Lambda$ is a distribution.}
\label{Boc_pos_sec} 

A distribution
$
\Lambda \in
\mathscr{D}^{\prime} (\mathbb{R}^{d})
$
is said to be {\it positive} if
$
\langle
	\Lambda ,
	f
\rangle
\geqslant 0
$
for all nonnegative
$
f \in
\mathscr{D} ( \mathbb{R}^{d} )
$.
It
is known that
for a positive distribution
$
\Lambda \in
\mathscr{D}^{\prime} (\mathbb{R}^{d})
$,
there exists a Radon measure $\mu$ on $\mathbb{R}^{d}$
such that
$$
\langle \Lambda , f \rangle
=
\int_{\mathbb{R}^{d}}
\langle
	\delta_{y},
	f
\rangle
\mu ( \mathrm{d}y ) ,
\quad
f \in
\mathscr{D} (\mathbb{R}^{d}).
$$

The main objective in this
section
is to prove

\begin{Thm} 
\label{pos_Bochner} 
Let $d=1$ and $x \in \mathbb{R}$.
Suppose that
$\Lambda \in \mathscr{D}^{\prime} ( \mathbb{R} )$
is positive.
\begin{itemize}
\item[(i)]
If
{\rm (H1)},
{\rm (H2)}
and
$\Lambda \in \mathscr{S}_{-2k}(\mathbb{R})$
where $k \in \mathbb{Z}_{\geqslant 0}$,
then
for any
$p \in ( 1, +\infty )$,
$
\int_{0}^{T}
\Vert
	\Lambda ( X(t,x,w) )
\Vert_{p,-2k}
\mathrm{d} t
< +\infty
$.
\vspace{2mm}
\item[(ii)]
If
{\rm (H1)},
{\rm (H2)},
{\rm (H4)}
and
$\Lambda \in \mathscr{E}^{\prime} (\mathbb{R})$
hold, then
\begin{equation}
\label{integral} 
\int_{0}^{T}
\int_{\mathbb{R}}
\Vert
	\delta_{y} ( X( t, x, w ) )
\Vert_{p,-2}
\mu ( \mathrm{d}y )
\mathrm{d} t
< +\infty ,
\end{equation}
where $\mu$ is the Radon measure associated to $\Lambda$
as above.
\end{itemize}
\end{Thm} 

\begin{Rm} 
\begin{itemize}
\item[(a)]
Therefore if we assume {\rm (H1)}, {\rm (H2)}
and
$
\Lambda \in
\mathscr{S}_{-2k} (\mathbb{R})
$
is positive, then
$
(0, T] \ni t
\mapsto
\Lambda
(
	X_{t}
) \in \mathbb{D}_{p}^{-2k}
$
is Bochner integrable.
If we have {\rm (H4)} additionally,
the mapping
$
(0, T] \times \mathbb{R} \ni ( t, y )
\mapsto
\delta_{y}
(
	X_{t}
) \in \mathbb{D}_{p}^{-2}
$
is also Bochner integrable with respect to
$\mathrm{d}t \otimes \mu (\mathrm{d}y)$
(the measurability can be checked from the construction of
$
\delta_{y} ( X_{t} )
$),
and we have
\begin{equation*}
\begin{split}
\int_{0}^{T}
\Lambda
(
	X_{t}
)
\mathrm{d}t
&=
\int_{0}^{T}
\int_{\mathbb{R}}
\delta_{y}
(
	X_{t}
)
\mu (\mathrm{d}y) \mathrm{d}t
=
\int_{\mathbb{R}}
\int_{0}^{T}
\delta_{y}
(
	X_{t}
)
\mathrm{d}t \mu (\mathrm{d}y)
\quad
\text{in $\mathbb{D}_{p}^{-2}$.}
\end{split}
\end{equation*}

\item[(b)]
Under {\rm (H1)} and {\rm (H2)},
Lemma~\ref{delta_unif}
and
Theorem~\ref{pos_Bochner}--(i)
assures
$
\int_{0}^{T}
\Vert
	\delta_{y}
	(
		X_{t}
	)
\Vert_{p,-2}
\mathrm{d}t
$
is finite for any $y \in \mathbb{R}$.
\end{itemize}
\end{Rm} 

\begin{Cor} 
\label{commute} 
Let $d=1$.
Assume {\rm (H1)} and {\rm (H2)}.
For each $y \in \mathbb{R}$
and
$n \in \mathbb{Z}_{\geqslant 0}$,
the mapping
$
(0,T] \ni t \mapsto J_{n} [ \delta_{y} (X_{t}) ] \in L_{2}
$
is Bochner integrable and
$$
J_{n} \big[ \int_{0}^{T} \delta_{y} (X_{t}) \mathrm{d}t \big]
=
\int_{0}^{T} J_{n} [ \delta_{y} (X_{t}) ] \mathrm{d}t .
$$
\end{Cor} 
\begin{proof}[Proof of Corollary \ref{commute}]
By Theorem \ref{pos_Bochner},
we have
\begin{equation*}
\begin{split}
\int_{0}^{T}
\Vert
	J_{n} [ \delta_{y} (X_{t}) ]
\Vert_{L_{2}}
\mathrm{d}t
&=
(1+n)
\int_{0}^{T}
\Vert
	J_{n} [ \delta_{y} (X_{t}) ]
\Vert_{2,-2}
\mathrm{d}t
\leqslant
(1+n)
\int_{0}^{T}
\Vert
	\delta_{y} (X_{t})
\Vert_{2,-2}
\mathrm{d}t
< +\infty .
\end{split}
\end{equation*}
This shows the Bochner integrability of the mapping
$
(0,T] \ni t \mapsto J_{n} [ \delta_{y} (X_{t}) ] \in L_{2}
$,
and hence
$
\int_{0}^{T} J_{n}[ \delta_{y} ( X_{t} ) ] \mathrm{d}t
\in L_{2}
$.

Second,
for each $F \in \mathbb{D}^{\infty}$, the mapping
$
\mathbb{D}_{2}^{s} \ni G \mapsto \mathbf{E} [GF] \in \mathbb{R}
$
is linear and bounded for each $s \in \mathbb{R}$.
Therefore Bochner integrals and (generalized) expectations
may be interchanged, and accordingly we have
\begin{equation*}
\begin{split}
&
\mathbf{E}
\big[
	J_{n}
	[
	\int_{0}^{T} \delta_{y} ( X_{t} ) \mathrm{d}t
	]
	F
\big]
=
\mathbf{E}
\big[
	\int_{0}^{T} \delta_{y} ( X_{t} ) \mathrm{d}t
	J_{n} F
\big] \\
&=
\int_{0}^{T}
\mathbf{E}
[
	\delta_{y} ( X_{t} )
	J_{n} F
]
\mathrm{d}t
=
\int_{0}^{T}
\mathbf{E}
\big[
	J_{n} [ \delta_{y} ( X_{t} ) ]
	F
\big]
\mathrm{d}t
=
\mathbf{E}
\big[
	\int_{0}^{T} J_{n} [ \delta_{y} ( X_{t} ) ] \mathrm{d}t
	F
\big] ,
\end{split}
\end{equation*}
which proves the second assertion.
\end{proof} 

To prove Theorem \ref{pos_Bochner},
we need several implements.

For each $\varepsilon > 0$,
we consider the following
$d$-dimensional stochastic differential equation
\begin{equation}
\label{SDE-per} 
\mathrm{d} X_{t}
=
\varepsilon \sigma (X_{t}) \mathrm{d} w(t)
+
\varepsilon^{2} b ( X_{t} ) \mathrm{d} t.
\end{equation}
Similarly to the equation (\ref{SDE}),
we denote by
$\{ X^{\varepsilon}(t,x,w) \}_{t \geqslant 0}$
a unique strong solution
$X^{\varepsilon} = (X_{t}^{\varepsilon})_{t \geqslant 0}$ to (\ref{SDE-per})
such that $X_{0}^{\varepsilon} = x \in \mathbb{R}^{d}$.
Then it holds that
for each $\varepsilon > 0$,
$\{ X( \varepsilon^{2} t,x,w) \}_{t \geqslant 0}$
is equivalent to
$\{ X^{\varepsilon} (t,x,w) \}_{t \geqslant 0}$
in law.
A more tricky fact which we need is the following.

\begin{Prop} 
\label{trick1} 
Let $\Lambda \in \mathscr{S}^{\prime} (\mathbb{R}^{d})$.
Then for every
$p \in (1,\infty)$,
$s \in \mathbb{R}$
and
$t > 0$,
we have
$
\Vert \Lambda ( X( \varepsilon^{2} t,x,w) ) \Vert_{p,s}
=
\Vert \Lambda ( X^{\varepsilon} ( t,x,w) ) \Vert_{p,s}
$.
\end{Prop} 
\begin{proof} 
For simplicity, we give a proof in the case of $d=1$.
Let
$p \in (1,\infty)$,
$s \in \mathbb{R}$
and
$t > 0$
be arbitrary.
It is enough to prove that
$
\Vert f ( X( \varepsilon^{2} t,x,w) ) \Vert_{p,s}
=
\Vert f ( X^{\varepsilon} ( t,x,w) ) \Vert_{p,s}
$
for each $f \in \mathscr{S}(\mathbb{R})$.
By the Veretennikov-Krylov formula
(see \cite[p.279, Theorem 4]{VeKr}),
we have
\begin{equation*}
\begin{split}
&
J_{n} [ f( X^{\varepsilon} (t,x,w) ) ]
=
\int_{0}^{t}
\cdots
\int_{0}^{r_{2}}
(
	P_{r_{1}}^{\varepsilon}
	Q_{r_{2}-r_{1}}^{\varepsilon}
	\cdots
	Q_{r_{n}-r_{n-1}}^{\varepsilon}
	Q_{t-r_{n}}^{\varepsilon}
	f
)
(x)
\mathrm{d} w(r_{1})
\cdots
\mathrm{d} w(r_{n}),
\end{split}
\end{equation*}
where
$
( P_{r}^{\varepsilon} f )(z)
:=
\mathbf{E} [ f(X^{\varepsilon} (r,z,w)) ]
$
for $z \in \mathbb{R}$,
$
Q_{r}^{\varepsilon} f
:=
A^{\varepsilon} ( P_{r}^{\varepsilon} f )
$
and
$
A^{\varepsilon} := \varepsilon \sigma \frac{ \mathrm{d} }{ \mathrm{d} x }
$.
In the case $\varepsilon =1$,
we will write
$P_{r} := P_{r}^{1}$,
$A := A^{1}$ and $Q_{r} := Q_{r}^{1}$
for simplicity.
Therefore, we have to show that
\begin{equation*}
\begin{split}
&
\Big\{
\int_{0}^{\varepsilon^{2}t}
\cdots
\int_{0}^{\varepsilon^{2}r_{2}}
(
	P_{r_{1}}
	Q_{r_{2}-r_{1}}
	\cdots
	Q_{
		r_{k}-r_{k-1}
	}
	Q_{
		\varepsilon^{2} t-r_{k}
	}
	f
)
(x)
\mathrm{d} w(r_{1})
\cdots
\mathrm{d} w(r_{k})
\Big\}_{k=0}^{n} \\
&=
\Big\{
\int_{0}^{t}
\cdots
\int_{0}^{r_{2}}
(
	P_{r_{1}}^{\varepsilon}
	Q_{r_{2}-r_{1}}^{\varepsilon}
	\cdots
	Q_{
		r_{k}-r_{k-1}
	}^{\varepsilon}
	Q_{
		t-r_{k}
	}^{\varepsilon}
	f
)
(x)
\mathrm{d} w(r_{1})
\cdots
\mathrm{d} w(r_{k})
\Big\}_{k=0}^{n}
\end{split}
\end{equation*}
in law for each $n \in \mathbb{N}$,
and hence
by the scaling property of Brownian motion:
$
( \varepsilon^{-1} w(\varepsilon^{2} t) )_{t \geqslant 0}
=
( w(t) )_{t \geqslant 0}
$
in law,
it suffices to show that
\begin{equation}
\label{law_equiv} 
\begin{split}
&
(
	P_{ \varepsilon^{2} r_{1} }
	Q_{ \varepsilon^{2}r_{2} - \varepsilon^{2}r_{1} }
	\cdots
	Q_{ \varepsilon^{2}r_{n} - \varepsilon^{2}r_{n-1} }
	Q_{ \varepsilon^{2}t - \varepsilon^{2}r_{n} }
	f
)
(x) \\
&=
\varepsilon^{-n}
(
	P_{r_{1}}^{\varepsilon}
	Q_{r_{2}-r_{1}}^{\varepsilon}
	\cdots
	Q_{r_{n}-r_{n-1}}^{\varepsilon}
	Q_{t-r_{n}}^{\varepsilon}
	f
)
(x)
\end{split}
\end{equation}
for each $0 < r_{1} < \cdots < r_{n} < t$.
But this is clear since we have
$
P_{\varepsilon^{2}r} = P_{r}^{\varepsilon}
$,
$A = \varepsilon^{-1} A^{\varepsilon}$
and
$
Q_{\varepsilon^{2}r} = \varepsilon^{-1} Q_{r}^{\varepsilon}
$
for each $r > 0$.
\end{proof} 

For each $\varepsilon > 0$ and $x \in \mathbb{R}^{d}$,
we set
$$
F( \varepsilon , x,  w )
:=
\frac{ X^{\varepsilon} ( 1,x,w ) - x }{ \varepsilon }.
$$
Then the following two conditions are equivalent
(see \cite[Theorem 3.4]{Wa87}):
\begin{itemize}
\item[$\circ$]
{\rm (H2)}, i.e., there exists $\lambda > 0$ such that
$$
\lambda \vert \xi \vert^{2}
\leqslant
\langle
	\xi ,
	( \sigma \sigma^{*} ) (x)
	\xi
\rangle_{\mathbb{R}^{d}},
\quad
\text{for all $\xi \in \mathbb{R}^{d}$.}
$$

\item[$\circ$]
the family
$\{ F(\varepsilon , x, w) \}_{\varepsilon > 0}$
is uniformly non-degenerate.
\end{itemize}

\begin{Prop} 
\label{asymp1} 
Let $d=1$ and $x \in \mathbb{R}$.
Suppose {\rm (H1)} and $\sigma (x)^{2} > 0$.
Then for any $p \in (1, \infty)$,
we have
$
\sup_{0 < \varepsilon \leqslant 1}
\Vert
\delta_{0} ( F( \varepsilon , x, w ) )
\Vert_{p,-2}
< +\infty
$.
\end{Prop} 
\begin{proof} 
Let
$
\phi (z) := ( 1 + z^{2} - \triangle )^{-1} \delta_{0} (z)
\in \mathscr{S}_{0}
$
and take $q \in (1, \infty)$
so that $1/p +1/q = 1$.
Then for each $J \in \mathbb{D}^{\infty}$,
we have
\begin{equation*}
\begin{split}
\mathbf{E}
[
	\delta_{0} ( F( \varepsilon , x, w ) )
	J
]
&=
\mathbf{E}
[
	\big(
		( 1 + z^{2}- \triangle )
		\phi
	\big)
	( F( \varepsilon , x, w ) )
	J
] \\
&=
\mathbf{E}
[
	\phi ( F( \varepsilon , x, w ) )
	l_{\varepsilon} ( J )
]
\end{split}
\end{equation*}
where $l_{\varepsilon} ( J ) \in \mathbb{D}^{\infty}$
is of the form
\begin{equation*}
l_{\varepsilon} ( J )
=
\langle
	P_{0} ( \varepsilon , w ),
	DJ
\rangle_{H}
+
\langle
	P_{1} ( \varepsilon , w ),
	D^{2} J
\rangle_{H \otimes H}
\end{equation*}
for some
$
P_{i}(\varepsilon , w)
\in \mathbb{D}^{\infty} ( H^{\otimes i} )
$,
$i=1,2$,
both of which are polynomials in
$F(\varepsilon , x, w)$,
its derivatives up to the second order
and
$
\Vert DX^{\varepsilon} (1,x,w) \Vert_{H}^{-2}
$
(see e.g., \cite[equation (2.20)]{Wa87}).

Take $q^{\prime} \in (1, q)$.
Since
$\{ F(\varepsilon , x, w) \}_{\varepsilon > 0}$
is uniformly non-degenerate,
there exists $c_{0} > 0$ such that
$$
\Vert
	l_{\varepsilon} (J)
\Vert_{q^{\prime}}
\leqslant
c_{0} \Vert J \Vert_{q,2}
\quad
\text{for all $\varepsilon \in (0,1]$ and $J \in \mathbb{D}^{\infty}$.}
$$
Therefore we have
for each $J \in \mathbb{D}^{\infty}$,
\begin{equation*}
\big\vert
\mathbf{E}
[
	\delta_{0} ( F( \varepsilon , x, w ) )
	J
]
\big\vert
\leqslant
\Vert \phi \Vert_{\infty}
\Vert l_{\varepsilon} ( J ) \Vert_{q^{\prime}}
\leqslant
c_{0}
\Vert \phi \Vert_{\infty}
\Vert J \Vert_{q, 2}
\end{equation*}
which implies
$
\sup_{0 < \varepsilon \leqslant 1}
\Vert
	\delta_{0} ( F( \varepsilon , x, w ) )
\Vert_{p,-2}
\leqslant
c_{0} \Vert \phi \Vert_{\infty}
< +\infty
$.
\end{proof} 

Second,
we recall the next fact
(see \cite[Chapter V, Section 8, p.384]{IW})
to prove Theorem \ref{pos_Bochner}.

\begin{Prop} 
\label{unif_est} 
Let $x \in \mathbb{R}^{d}$
and suppose that {\rm (H1)} and {\rm (H2)}.
Then for any
$p \in (1, \infty)$
and $k \in \mathbb{Z}_{\geqslant 0}$,
there exists $C>0$ such that
\begin{equation*}
\begin{split}
\Vert
	\phi \circ F ( \varepsilon , x, w )
\Vert_{p,-2k}
\leqslant
C
\Vert
	\phi
\Vert_{-2k}
\quad
\text{for $\phi \in \mathscr{S} (\mathbb{R}^{d})$,
$\varepsilon \in (0,T]$.}
\end{split}
\end{equation*}
\end{Prop} 

The last tool we need is the following.

\begin{Lem} 
\label{IbP} 
Let $x \in \mathbb{R}^{d}$.
\begin{itemize}
\item[(i)]
If {\rm (H1)} and {\rm (H2)} hold, then
there exist
$K>0$, $T_{0} \in (0,T]$
with the following property{\rm :}
For each
$k_{1} , \cdots , k_{d} \in \mathbb{Z}_{\geqslant 0}$
and
$p \in (1,\infty)$,
there exist constants
$\nu_{0} , c_{1},c_{2}>0$
such that
\begin{equation*}
\Vert
	\partial_{1}^{k_{1}} \cdots \partial_{d}^{k_{d}}
	\delta_{y} ( X(t,x,w) )
\Vert_{p,-(n+2)}
\leqslant
c_{1} t^{-\nu_{0}}
\exp
\Big\{
	- c_{2}
	\frac{ \vert x-y \vert }{ t }
\Big\}
\end{equation*}
for each
$t \in (0,T_{0}]$
and
$y \in \mathbb{R}$ such that $\vert y \vert \geqslant K$.
\vspace{2mm}
\item[(ii)]
If {\rm (H1)}, {\rm (H2)} and {\rm (H4)} hold, then
for each
$k_{1} , \cdots , k_{d} \in \mathbb{Z}_{\geqslant 0}$,
$K > \vert x \vert$
and
$p \in (1,\infty)$,
there exist constants $\nu_{0} , c_{1},c_{2}>0$
such that
\begin{equation*}
\Vert
	\partial_{1}^{k_{1}} \cdots \partial_{d}^{k_{d}}
	\delta_{y} ( X(t,x,w) )
\Vert_{p,-(n+2)}
\leqslant
c_{1} t^{-\nu_{0}}
\exp
\Big\{
	- c_{2}
	\frac{ \vert x-y \vert^{2} }{ t }
\Big\}
\end{equation*}
for each
$t \in (0,T]$
and
$y \in \mathbb{R}$ such that $\vert y \vert \geqslant K$.
\end{itemize}
Here, $n:=k_{1} + \cdots + k_{d}$ in both cases.
\end{Lem} 
\begin{proof} 
(i)
Assume {\rm (H1)} and {\rm (H2)}.
For simplicity of notation, we prove the case $d=1$.
The case $d \geqslant 2$ can be proved by a similar argument.
It suffices to prove that
there exist $K>0$ and $T_{0} \in (0,T]$
with the following property:
\begin{quote}
For each
$n \in \mathbb{Z}_{\geqslant 0}$
and
$p \in (1,\infty)$,
there exist constants $\nu_{0} , c_{1},c_{2}>0$
such that
\begin{equation*}
\begin{split}
\vert
\mathbf{E}
[
	J
	\delta_{y}^{(n)} ( X_{t} )
]
\vert
\leqslant
c_{1} t^{-\nu_{0}}
\exp
\Big\{
	- c_{2}
	\frac{ \vert x-y \vert }{ t }
\Big\}
\Vert
	J
\Vert_{p,n+2},
\end{split}
\end{equation*}
for each
$J \in \mathbb{D}^{\infty}$,
$t \in (0,T_{0}]$
and
$y \in \mathbb{R}$ with $\vert y \vert \geqslant K$.
\end{quote}

Let $y \in \mathbb{R}$ be arbitrary and
$
\varphi (z)
:= ( 1 + z^{2} - \triangle  )^{-1} \delta_{y} (z)
\in \mathscr{S}_{0}
$.
Take a $C^{\infty}$-function
$\phi : \mathbb{R} \to \mathbb{R}$ such that
$$
\phi (\xi)
=
\left\{\begin{array}{ll}
1 & \text{if $\xi \leqslant 1/3$,} \\
0 & \text{if $\xi \geqslant 2/3$}
\end{array}\right.
$$
and then we set
$
\psi_{y}
(z)
:=
\phi \big( \frac{z-y}{\vert x-y \vert} \big)
$.
Let $p \in (1, \infty)$ be arbitrary
and let $q \in (1, \infty)$ be such that $1/p +1/q = 1$.
Further take $q^{\prime} \in (1,q)$.
Since
$
\psi_{y}
\delta_{y} = \delta_{y}
$,
we have for each $J \in \mathbb{D}^{\infty}$ that
\begin{equation}
\label{cal3} 
\begin{split}
\mathbf{E}
[
	J
	\delta_{y}^{(n)} ( X_{t} )
]
&=
\mathbf{E}
[
	J
	\psi_{y} (X_{t})
	\Big(
		\frac{ \mathrm{d}^{n} }{ \mathrm{d} z^{n} }
		( 1+z^{2}- \triangle ) \varphi
	\Big)
	( X_{t} )
] \\
&=
\mathbf{E}
[
	\varphi ( X_{t} )
	\Big\{
	\sum_{j=0}^{n+2}
	\psi_{y}^{(j)}
	( X_{t} )
	l_{j, t} (J)
	\Big\}
] ,
\end{split}
\end{equation}
where
$
\psi_{y}^{(j)}
$
is the $j$-th derivative of
$
\psi_{y}
$
(with convention that
$
\psi_{y}^{(0)}
=
\psi_{y}
$)
and each $l_{j, t} (J)$ is of the form
\begin{equation}
\label{bear} 
P_{0} ( t, w )
J
+
\langle
	P_{1} ( t, w ),
	D J
\rangle_{H}
+
\cdots
+
\langle
	P_{n+2} ( t, w ),
	D^{n+2}J
\rangle_{H^{\otimes (n+2)}}
\end{equation}
for some
$
P_{i} ( t, w ) \in \mathbb{D}^{\infty} (H^{\otimes i})
$,
$i=0,1,\cdots , n+2$,
which is a polynomial in
$X_{t} = X(t,x,w)$,
its derivatives
and
$\Vert DX(t,x,w) \Vert_{H}^{-2}$,
but does not depend on
$
\psi_{y}
$.

Note that
$
\psi_{y}
(z)  \equiv 0
$
for $\vert z-x \vert < \frac{ \vert y-x \vert }{3}$,
and hence the last
term
in (\ref{cal3}) equals to
\begin{equation*}
\begin{split}
&
\mathbf{E}
[
	\Big(
	\sum_{j=0}^{n+2}
	( \varphi \psi_{y}^{(j)} ) ( X_{t} )
	l_{j, t} (J)
	\Big)
	1_{\text{{\small $
	\{
	\vert X_{t} - x \vert
	\geqslant
	\frac{ \vert y-x \vert }{ 3 }
	\}$}}
	}
].
\end{split}
\end{equation*}
Therefore, by taking
$p^{\prime} \in (1,\infty)$
such that $1/p^{\prime} + 1/q^{\prime} = 1$,
we have
\begin{equation}
\label{3factors} 
\begin{split}
&
\vert
\mathbf{E}
[
	J
	\delta_{y}^{(n)}
	( X_{t} )
]
\vert
\leqslant
\sum_{j=0}^{n+2}
\Vert
\varphi
\psi_{y}^{(j)}
\Vert_{\infty}
\Vert l_{j,t} (J) \Vert_{q^{\prime}}
\mathbf{P}
\Big(
	\vert X_{t} - x \vert
	\geqslant
	\frac{ \vert y-x \vert }{ 3 }
\Big)^{1/p^{\prime}} .
\end{split}
\end{equation}

Henceforth we shall focus on each factor
in the last equation.

First,
since
$
\psi_{y}^{(j)}
(z)
=
\vert x-y \vert^{-j}
\phi^{(j)} ( (z-y) / \vert x-y \vert )
$
for $j=0,1,\cdots , n+2$,
and by Lemma \ref{delta_unif}
we have
\begin{equation}
\label{phi_psi} 
\sup_{
	\substack{
	y \in \mathbb{R} \\
	\vert y \vert \geqslant K
	}
}
\Vert
\varphi
\psi_{y}^{(j)}
\Vert_{\infty}
< +\infty
\quad
\text{for $K>0$ and $j=0,1,\cdots , n+2$.}
\end{equation}

Second,
it is well known that
for each $r \in (1, \infty)$,
there exist
$\nu, c_{1}^{\prime} = c_{1}^{\prime}(r) > 0$
such that
\begin{equation*}
\big\Vert
	\Vert DX_{t} \Vert_{H}^{-2}
\big\Vert_{r}
\leqslant
c_{1}^{\prime}
t^{-\nu}
\quad
\text{for any $t \in (0,T]$}
\end{equation*}
(see \cite[(3.25) Corollary p.22]{KS84b} or
\cite[Chapter V, Section 10, Theorem 10.2]{IW}).
Now, bearing in mind the form (\ref{bear}),
we see that there exist $\nu_{0}, c_{1}^{\prime\prime} >0$ such that
\begin{equation}
\label{el_est}
\Vert
	l_{j,t} (J)
\Vert_{q^{\prime}}
\leqslant
c_{1}^{\prime\prime}
t^{-\nu_{0}}
\Vert J \Vert_{p,2}
\quad
\text{for $j=0,1,\cdots ,n+2$ and $t \in (0,T]$.}
\end{equation}

Third, we shall prove that
there exist $K>0$ and $T_{0} \in (0,T]$
with the following property:
there exist $c_{1}^{\prime\prime\prime}, c_{2} > 0$ such that
\begin{equation}
\label{prob_est} 
\mathbf{P}
\Big(
	\vert X_{t} - x \vert
	\geqslant
	\frac{ \vert y-x \vert }{ 3 }
\Big)
\leqslant
c_{1}^{\prime\prime\prime}
\exp \Big\{ - c_{2} \frac{ \vert x-y \vert }{ t } \Big\}
\end{equation}
for all $t \in (0,T_{0}]$
and
$y \in \mathbb{R}$ with $\vert y \vert \geqslant K$.
For this, we recall the following general fact
(see \cite[Chapter V, section 10, Lemma 10.5]{IW}):
\begin{quote}
{\it
Let $\kappa >0$
and
$X=(X_{t})_{0 \leqslant t \leqslant T}$
be a one-dimensional continuous semimartingale
with its Doob-Mayer decomposition
$
X_{t} = X_{0} + M_{t} + A_{t}
$
such that
$
\langle M \rangle_{t}
=
\int_{0}^{t} \alpha (s) \mathrm{d}s
$,
$
A_{t}
=
\int_{0}^{t} \beta (s) \mathrm{d}s
$
and
$$
\sup_{0 \leqslant t \leqslant T}
\max \{ \vert \alpha (t) \vert , \vert \beta (t) \vert \}
\leqslant \kappa .
$$
Then for any $a>0$ and
$t \in ( 0, \min\{ \frac{a}{2 \kappa }, T \} ]$,
it holds that
\begin{equation*}
\begin{split}
\mathbf{P}
(
	\tau_{a} < t
)
\leqslant
\frac{ 4 }{ \sqrt{ \pi a } }
\exp \Big( - \frac{ a^{2} }{ 8 \kappa t } \Big) ,
\end{split}
\end{equation*}
where
$\tau_{a} := \inf \{ t > 0 : \vert X_{t} - X_{0} \vert > a \}$.
}
\end{quote}

Since we have assumed {\rm (H1)},
$\sigma^{2}$ and $b$ are Lipschitz continuous.
Thus there exists $\alpha , \beta \geqslant 0$
such that
$
\max \{ \vert \sigma^{2} (y) \vert , \vert b(y) \vert \}
\leqslant
\alpha \vert y-x \vert + \beta
$
for any $y \in \mathbb{R}$.
Put $\kappa (z) := \alpha z + \beta$
and fix $y \in \mathbb{R}$ arbitrarily.
Let $T_{0} \in (0,T]$ so that
$1 - 6 \alpha T_{0} > 0$.
Define
$
\tau :=: \tau_{ \vert y-x \vert / 3 }
:=
\inf
\{
	t > 0:
	\vert X_{t} - x \vert > \frac{\vert y-x \vert}{3}
\}
$.
Then the stopped process
$
X^{\tau}
=
( X_{t \wedge \tau} )_{t \geqslant 0}
$
is a continuous semi-martingale satisfying
$$
X_{t \wedge \tau}
=
x
+
\int_{0}^{
	t\wedge \tau
}
\sigma_{y} ( X_{s \wedge \tau} ) \mathrm{d} w(s)
+
\int_{0}^{
	t\wedge \tau
}
b_{y} ( X_{s \wedge \tau} ) \mathrm{d} s ,
$$
where
$
\sigma_{y} (z)
:=
\min \{ \sigma (z) , \kappa ( \vert y-x \vert ) \}
$
and
$
b_{y} (z)
:=
\min \{ b(z) , \kappa ( \vert y-x \vert ) \}
$.

By setting
$a = \frac{\vert y-x \vert}{3}$
in the above,
we see that
if
$
\vert y-x \vert
\geqslant
\xi_{0}
:=
\max \{ \frac{ 6 \beta T_{0} }{ 1 - 6 \alpha T_{0} }, 1 \}
>0
$
and
$
0 < t \leqslant
T_{0}
$
(here, note that
$\frac{ \xi_{0} }{ 6 ( \alpha \xi_{0} + \beta ) }
\geqslant T_{0}
$
since
$
\xi \mapsto \frac{ \xi }{ 6 ( \alpha \xi + \beta ) }
$
is a nondecreasing function)
then
\begin{equation*}
\begin{split}
\mathbf{P}
\Big(
	\vert X_{t} - X_{0} \vert
	>
	\frac{\vert y-x \vert}{3}
\Big)
&\leqslant
\mathbf{P}
\Big(
	\vert X_{t}^{\tau_{a}} - X_{0} \vert
	>
	\frac{\vert y-x \vert}{3}
\Big) \\
&\leqslant
\mathbf{P}
(
	\tau_{a} < t
) \\
&\leqslant
\frac{ 4 }{ \sqrt{ \pi \frac{\vert y-x \vert}{3} } }
\exp
\Big(
	- \frac{ (\frac{\vert y-x \vert}{3})^{2} }{ 8 ( \alpha \vert y-x \vert + \beta ) t }
\Big) \\
&\leqslant
\frac{ 4 \sqrt{3} }{ \sqrt{ \pi \xi_{0} } }
\exp
\Big(
	-
	\frac{
		\vert y-x \vert^{2}
	}{
		72 ( \alpha \vert y-x \vert + \beta \vert y-x \vert ) t
	}
\Big) \\
&=
\frac{ 4 \sqrt{3} }{ \sqrt{ \pi \xi_{0} } }
\exp
\Big(
	- \frac{ \vert y-x \vert }{ 72 ( \alpha + \beta ) t }
\Big).
\end{split}
\end{equation*}
Thus
(\ref{prob_est})
is satisfied if we set
$K := \vert x \vert + \xi_{0}$,
$
c_{1}^{\prime\prime\prime}
:=
\frac{ 4 \sqrt{3} }{ \sqrt{ \pi \xi_{0} } }
$
and
$
c_{2}
:=
( 72 ( \alpha + \beta ) )^{-1}
$.

Now, combining
(\ref{3factors}),
(\ref{phi_psi}),
(\ref{el_est})
and
(\ref{prob_est}),
we obtain the result.

(ii)
Assume {\rm (H1)}, {\rm (H2)} and {\rm (H4)}.
In this case, $\sigma$ and $b$ are bounded,
and hence the following well-known estimate
is available:
There exist $c_{1}^{\prime\prime\prime}, c_{2} > 0$ such that
\begin{equation}
\mathbf{P}
\Big(
	\vert X_{t} - x \vert
	\geqslant
	\frac{ \vert y-x \vert }{ 3 }
\Big)
\leqslant
c_{1}^{\prime\prime\prime}
\exp \Big\{ -c_{2} \frac{ \vert x-y \vert^{2} }{ t } \Big\}
\end{equation}
for all $t \in (0,T]$ and $x,y \in \mathbb{R}$.
By using this instead of (\ref{prob_est}),
and combining with
(\ref{3factors}),
(\ref{phi_psi})
and
(\ref{el_est}),
we obtain the result.
\end{proof} 

We are now in a position to prove
Theorem \ref{pos_Bochner}.

\begin{proof}[Proof of Theorem \ref{pos_Bochner}]
(i)
Let
$K>0$ and $T_{0} \in (0,T]$
be as in Lemma~\ref{IbP}--(i).
Since
\begin{equation*}
\begin{split}
\frac{ \vert x-y \vert^{2} }{ y^{2} }
\geqslant
\frac{ ( \vert x \vert - \vert y \vert )^{2} }{ y^{2} }
=
\frac{ y^{2} - 2 \vert xy \vert + x^{2} }{ y^{2} }
\to
1
\quad
\text{as $\vert y \vert \to +\infty$,}
\end{split}
\end{equation*}
there exists $K^{\prime}>0$ such that
\begin{equation}
\label{simple_est} 
\vert x-y \vert \geqslant \frac{ \vert y \vert }{ 2 },
\quad
\text{for any $\vert y \vert > K^{\prime}$.}
\end{equation}

Let
$
K^{\prime\prime}
:=
\max
\{
	K,
	K^{\prime}
\}
$.
Let $k \in \mathbb{N}$ be such that
$\Lambda \in \mathscr{S}_{-2k} (\mathbb{R}^{d})$.
Let $p \in (1,\infty )$ be arbitrary.
We have
\begin{equation*}
\begin{split}
\int_{0}^{T}
\Vert \Lambda (X_{t}) \Vert_{p,-2k}
\mathrm{d}t
=
\int_{0}^{T_{0}}
\Vert \Lambda (X_{t}) \Vert_{p,-2k}
\mathrm{d}t
+
\int_{T_{0}}^{T}
\Vert \Lambda (X_{t}) \Vert_{p,-2k}
\mathrm{d}t,
\end{split}
\end{equation*}
here, the last term is finite since
$
(0,T] \ni t \mapsto \Lambda ( X_{t} ) \in \mathbb{D}_{p}^{-2k}
$
is continuous.
The other term is estimated as
\begin{equation*}
\begin{split}
\int_{0}^{T_{0}}
\Vert \Lambda (X_{t}) \Vert_{p,-2k}
\mathrm{d}t
\leqslant
\int_{0}^{T_{0}}
\int_{\mathbb{R}}
\Vert \delta_{y} (X_{t}) \Vert_{p,-2}
\mu ( \mathrm{d} y )
\mathrm{d}t
\leqslant
I_{1} + I_{2},
\end{split}
\end{equation*}
where
\begin{equation*}
\begin{split}
I_{1}
&:=
\int_{0}^{T_{0}}
\int_{ \vert y \vert > K^{\prime\prime} }
\Vert \delta_{y} (X_{t}) \Vert_{p,-2}
\mu ( \mathrm{d} y )
\mathrm{d}t,
\quad
I_{2}
:=
\int_{0}^{T_{0}}
\int_{ \vert y \vert \leqslant K^{\prime\prime} }
\Vert \delta_{y} (X_{t}) \Vert_{p,-2}
\mu ( \mathrm{d} y )
\mathrm{d}t .
\end{split}
\end{equation*}

We shall look at the integral $I_{1}$.
By
Lemma \ref{IbP}--(i)
and
(\ref{simple_est}),
there exist $\nu_{0}, c_{1}, c_{2} > 0$ such that
we have
\begin{equation}
\label{ineq:4} 
\Vert
	\delta_{y} ( X_{t} )
\Vert_{p,-2}
\leqslant
c_{1}
t^{-\nu_{0}}
\mathrm{e}^{ \text{{\small $
	- c_{2} \frac{ \vert x-y \vert }{ t }
$}}}
\leqslant
c_{1}
t^{-\nu_{0}}
\exp
\Big(
	- c_{2} \frac{ \vert y \vert }{ 2t }
\Big)
\end{equation}
for $\vert y \vert > K^{\prime\prime}$ and $t \in (0, T_{0}]$.
To dominate the last quantity,
we shall prove that
for some $c_{3}>0$, it holds that
\begin{equation}
\label{ineq:5} 
t^{-\nu_{0}}
\exp
\Big(
	- c_{2} \frac{ \vert y \vert }{ 2t }
\Big)
\leqslant
c_{3}
\exp
\Big(
	- c_{2} \frac{ \vert y \vert }{ 4t }
\Big)
\quad
\text{for $t \in (0,T_{0}]$ and $\vert y \vert > K^{\prime\prime}$.}
\end{equation}
Indeed, we have
\begin{equation*}
\frac{
	t^{-\nu_{0}}
	\exp
	(
		- c_{2} \frac{ \vert y \vert }{ 2t }
	)
}{
	\exp
	(
		- c_{2} \frac{ \vert y \vert }{ 4t }
	)
}
=
t^{-\nu_{0}}
\exp
\Big(
	- c_{2} \frac{ \vert y \vert }{ 4t }
\Big)
\leqslant
t^{-\nu_{0}}
\exp
\Big(
	- c_{2} \frac{ K^{\prime\prime} }{ 4t }
\Big)
\to 0
\end{equation*}
as $t \downarrow 0$, which proves (\ref{ineq:5}).
Combining (\ref{ineq:4}) and (\ref{ineq:5}),
we obtain
\begin{equation*}
\begin{split}
I_{1}
&\leqslant
c_{1} c_{3}
\int_{0}^{T_{0}}
\int_{ \vert y \vert > K^{\prime\prime} }
\exp
\Big\{
	- c_{2}
	\frac{
		\vert y \vert
	}{
		4t
	}
\Big\}
\mu ( \mathrm{d}y )
\mathrm{d} t \\
&\leqslant
c_{1} c_{3} T
\int_{ \vert y \vert > K^{\prime\prime} }
\exp
\Big\{
	- c_{2}
	\frac{
		\vert y \vert
	}{
		4 T
	}
\Big\}
\mu ( \mathrm{d}y )
=
c_{1} c_{3} T
\langle
	\Lambda , f
\rangle
< +\infty
\end{split}
\end{equation*}
where $f \in \mathscr{S}(\mathbb{R})$ is given by
$
f (y) :=
\phi (y)
\mathrm{e}^{
	- c_{2} \vert y \vert / ( 4 T )
}
$,
$y \in \mathbb{R}$
and
$\phi$ is a $C^{\infty}$-function
such that
$\phi = 0$ on a neighbourhood of $0$
and
$\phi (y) = 1$ if $\vert y \vert \geqslant K^{\prime\prime}$.

Next we turn to $I_{2}$.
By Proposition \ref{unif_est}
and
with noting that
$
\sup_{a \in \mathbb{R}}
\Vert
	\delta_{a}
\Vert_{-2}
< +\infty
$
(Proposition \ref{delta_unif}),
we have
\begin{equation*}
\begin{split}
&
t^{1/2}
\sup_{y \in \mathbb{R}}
\Vert \delta_{y} ( X(t,x,w) ) \Vert_{p,-2}
=
t^{1/2}
\sup_{y \in \mathbb{R}}
\Vert \delta_{y} ( X^{\sqrt{t}}(1,x,w) ) \Vert_{p,-2} \\
&=
\sup_{y \in \mathbb{R}}
\Vert
	\delta_{
		(y-x)
		/ \sqrt{t}
	}
	( F(\sqrt{t},x,w) )
\Vert_{p,-2}
\leqslant
C
\sup_{a \in \mathbb{R}}
\Vert
	\delta_{a}
\Vert_{-2}
\end{split}
\end{equation*}
for each $t>0$.
Hence we can conclude that there exists
$C^{\prime} > 0$
such that
\begin{equation*}
\Vert \delta_{y} ( X(t,x,w) ) \Vert_{p,-2}
\leqslant
C^{\prime} t^{-1/2}
\quad
\text{for $\vert y \vert \leqslant K^{\prime\prime}$, $t \in (0,T]$.}
\end{equation*}
Now, it is easy to deduce that
\begin{equation*}
I_{2}
\leqslant
C^{\prime} \mu ( \{ y \in \mathbb{R} : \vert y \vert \leqslant K^{\prime\prime} \} )
\int_{0}^{T} t^{-1/2} \mathrm{d}t
< +\infty .
\end{equation*}

(ii) is proved similarly by using Lemma~\ref{IbP}--(ii)
instead of Lemma~\ref{IbP}--(i).
\end{proof} 

Recall that the support of
$
\Lambda \in \mathscr{D}^{\prime} (\mathbb{R}^{d})
$
is defined as the complement of
$$
\left\{
	y \in \mathbb{R}:
	\begin{array}{c}
	\text{there exists an open neighbourhood} \\
	\text{$U$ of $y$ such that for any $f \in \mathscr{D}(\mathbb{R}^{d})$,} \\
	\text{$\mathrm{supp} f \subset U \Rightarrow \langle \Lambda , f \rangle = 0$.}
	\end{array}
\right\} .
$$

The proof of the following lemma will be given after
Proposition~\ref{exclusion_int} below.

\begin{Lem} 
\label{exclusion} 
Let
$x \in \mathbb{R}^{d}$
and
$\Lambda \in \mathscr{D}^{\prime} (\mathbb{R}^{d})$
be such that
$
\mathrm{supp} \Lambda
\not\hspace{0.3mm} \ni
x
$.
Then
there exists $k \in \mathbb{Z}_{\geqslant 0}$
such that
$$
\lim_{t \downarrow 0}
\Vert
	\Lambda ( X( t, x, w ) )
\Vert_{p,-k}
= 0
\quad
\text{for each $p \in (1,\infty )$,}
$$
if either one of the following holds{\rm :}
\begin{itemize}
\item[(i)]
{\rm (H1)}, {\rm (H2)} and $\Lambda \in \mathscr{S}^{\prime} (\mathbb{R}^{d})$.
\vspace{2mm}
\item[(ii)]
{\rm (H1)}, {\rm (H2)}, {\rm (H4)} and
$\Lambda \in \mathscr{E}^{\prime} (\mathbb{R}^{d})$.
\end{itemize}
\end{Lem} 

From this, the following is immediate.

\begin{Prop} 
\label{exclusion_int} 

Let
$x \in \mathbb{R}^{d}$
and
$\Lambda \in \mathscr{D}^{\prime} (\mathbb{R}^{d})$
be such that
$
\mathrm{supp} \Lambda
\not\hspace{0.3mm} \ni
x
$.
Under the assumption in Lemma~\ref{exclusion},
there exists $k \in \mathbb{Z}_{\geqslant 0}$
such that
$$
\int_{0}^{T}
\Vert
	\Lambda ( X( t, x, w ) )
\Vert_{p,-k}^{p}
\mathrm{d} t
< +\infty
\quad
\text{for each $p \in (1,\infty )$.}
$$
\end{Prop} 

\begin{proof}[Proof of Lemma \ref{exclusion}] 
We shall prove under the assumption (ii).
The proof in the case (i) is omitted since
one can prove similarly.

We assume $d=1$ just for simplicity.
Let $p \in (1, \infty )$ be arbitrary.
By Proposition~\ref{Hasu-Thm},
we can write
$
\Lambda
=
\frac{\mathrm{d}^{k}}{\mathrm{d}x^{k}}
[ \exp ( k \vert x \vert ) f(x) ]
$,
where $k$ is a nonnegative integer
and
$f:\mathbb{R} \to \mathbb{R}$
is a bounded continuous function.

Since $x \notin \mathrm{supp} \Lambda$,
we have
$
r_{0}
:=
\inf
\{
	\vert x - y \vert :
	y \in \mathrm{supp} \Lambda
\} > 0
$.
Let
$
\Omega
:=
\cup_{y \in \mathrm{supp} \Lambda}
B_{r_{0}/2}(y)
$
(i.e., the
$
(r_{0}/2)
$-neighbourhood of $\mathrm{supp} \Lambda$),
where $B_{r}(y)$
is the open ball with center $y$ and radius $r$.
Then, we have
$\Lambda \in \mathscr{D}^{\prime} ( \Omega )$,
and the function $f$ can be rearranged
so that
$\mathrm{supp} f \subset \Omega$.

Now let
$\varepsilon > 0$
and
$J \in \mathbb{D}^{\infty}$
be arbitrary.
Putting
$
\widetilde{e}_{k} (x)
=
e_{k} ( \vert x \vert )
=
\exp ( k \vert x \vert )$,
we have
\begin{equation*}
\begin{split}
\mathbf{E} [ \Lambda (X^{\varepsilon}(1,x,w)) J ]
&=
\mathbf{E}
[
	(
		\widetilde{e}_{k}
		f
	)^{(k)} (X^{\varepsilon}(1,x,w)) J
] \\
&=
\mathbf{E}
[
	\exp ( k \vert X^{\varepsilon}(1,x,w) \vert )
	f ( X^{\varepsilon}(1,x,w) )
	l_{\varepsilon} (J)
]
\end{split}
\end{equation*}
where
$l_{\varepsilon} (J) \in \mathbb{D}^{\infty}$
is of the form
$$
l_{\varepsilon} (J)
=
\sum_{j=0}^{k}
\langle
	P_{j} ( \varepsilon , w ),
	D^{j} J
\rangle_{H^{\otimes j}}
$$
for some
$
P_{j} ( \varepsilon , w )
\in
\mathbb{D}^{\infty} ( H^{\otimes j} )
$,
$j=0,1,\cdots , k$,
which are polynomials in
$F(\varepsilon , x, w)$,
its derivatives up to the order $k$,
and
$\vert DX^{\varepsilon}(t,x,w) \vert_{H}^{-2}$.
Take $q^{\prime} \in (1, q)$.
Since
$\{ F(\varepsilon , x, w) \}_{\varepsilon > 0}$
is uniformly non-degenerate,
there exists $c_{0}, \nu > 0$ such that
$$
\Vert
	l_{\varepsilon} (J)
\Vert_{q^{\prime}}
\leqslant
c_{0} \varepsilon^{-\nu} \Vert J \Vert_{q,k}
\quad
\text{for all $\varepsilon \in (0,T]$ and $J \in \mathbb{D}^{\infty}$.}
$$
Therefore by taking
$p^{\prime} \in (1, \infty)$
such that
$1/p^{\prime} + 1/q^{\prime} = 1$,
we have
\begin{equation*}
\begin{split}
\big\vert
\mathbf{E}
[
	\Lambda ( X^{\varepsilon}( 1 , x, w ) )
	J
]
\big\vert
&\leqslant
\Vert
	\exp ( k \vert X^{\varepsilon}(1,x,w) \vert )
	f( X^{\varepsilon}(1,x,w) )
\Vert_{p^{\prime}}
\Vert l_{\varepsilon} ( J ) \Vert_{q^{\prime}} \\
&\leqslant
c_{0}
\varepsilon^{-\nu}
\Vert
	\exp ( k \vert X^{\varepsilon}(1,x,w) \vert )
	f( X^{\varepsilon}(1,x,w) )
\Vert_{p^{\prime}}
\Vert J \Vert_{q, k},
\end{split}
\end{equation*}
which implies
$
\Vert
	\Lambda ( X^{\varepsilon}( 1 , x, w ) )
\Vert_{p,-k}
\leqslant
c_{0}
\varepsilon^{-\nu}
\Vert
	\exp ( k \vert X^{\varepsilon}(1,x,w) \vert )
	f( X^{\varepsilon}(1,x,w) )
\Vert_{p^{\prime}}
$
for any $\varepsilon > 0$.
By Proposition \ref{trick1}, we obtain
\begin{equation*}
\begin{split}
\Vert
	\Lambda ( X( t , x, w ) )
\Vert_{p,-k}
&=
\Vert
	\Lambda ( X^{\sqrt{t}}( 1 , x, w ) )
\Vert_{p,-k} \\
&\leqslant
c_{0}
t^{-\nu /2}
\Vert
	\exp ( k \vert X^{\sqrt{t}}(1,x,w) \vert )
	f( X^{\sqrt{t}}(1,x,w) )
\Vert_{p^{\prime}} \\
&=
c_{0}
t^{-\nu /2}
\big\Vert
	\exp ( k \vert X_{t} \vert )
	f( X_{t} )
	1_{\text{\small $
	\{ \vert X_{t} - x \vert
	>
	\frac{
		r_{0}
	}{2}
	\}
	$}}
\big\Vert_{p^{\prime}} .
\end{split}
\end{equation*}
By Lemma~\ref{IbP}--(ii),
we find that
for any $r>0$,
$
\mathbf{E}
[
	\exp ( r \vert X (t,x,w) \vert )
]
= O(t^{-\nu^{\prime}})
$
as $t\downarrow 0$ for some $\nu^{\prime} > 0$.
However we have
$
\mathbf{P}
(
\vert X_{t} - x \vert
>
\frac{
	r_{0}
}{2}
)
=
O(
\mathrm{e}^{ - r_{0}^{2} /( c_{1} t ) }
)
$
as $t \downarrow 0$
for some constant $c_{1} > 0$,
so that the last quantity
converges to zero as $t \downarrow 0$,
and hence get the conclusion.
\end{proof} 

\subsection{H\"older continuity of local time in space variable: a special case}

In this
section,
we assume $d=1$, {\rm (H1)} and {\rm (H3)}.
Let $X=(X_{t})_{t \geqslant 0}$ be a
unique strong solution to the following one-dimensional
stochastic differential equation
\begin{equation}
\label{Fisk-Strat-SDE} 
\mathrm{d}X_{t}
=
\sigma ( X_{t} ) \mathrm{d}w(t)
+
\frac{1}{2} \sigma (X_{t}) \sigma^{\prime} (X_{t}) \mathrm{d}t ,
\quad
X_{0} = x \in \mathbb{R},
\end{equation}
or equivalently,
\begin{equation*}
\mathrm{d}X_{t}
=
\sigma ( X_{t} ) \circ \mathrm{d} w(t),
\quad
X_{0} = x.
\end{equation*}

The main purpose in this
section
is to prove
Theorem \ref{Main_Thm2}.

Note that the object
$
\sigma (y)^{2}
\int_{0}^{t} \delta_{y} ( X_{u} ) \mathrm{d}u
$
in Theorem \ref{Main_Thm2}
is identified with the symmetric local time
associated to the diffusion process $(X_{t})_{t\geqslant 0}$.
See Remark \ref{sym-local_time}.

The Hermite polynomials $H_{n}$,
$n \in \mathbb{Z}_{\geqslant 0}$
are defined by
$H_{0} (x) = 1$
and
$H_{n}(x) := \partial^{*n} 1 (x)$
for
$n \in \mathbb{N}$ and $x \in \mathbb{R}$,
where the operator $\partial^{*}$ is given by
$$
\partial^{*}f (x)
:=
- f^{\prime} (x) + x f(x) ,
\quad
x \in \mathbb{R}
$$
for any differentiable function
$f: \mathbb{R} \to \mathbb{R}$.

The proof of Theorem
\ref{Main_Thm2}
starts from this paragraph.
Let $y,z \in \mathbb{R}$ be arbitrary.
Let
$A = A_{z} := \sigma (z) \frac{\mathrm{d}}{\mathrm{d} z}$
and
$p_{t}(z_{1},z_{2})$
be the transition-density function of $X$.
For each $t \geqslant 0$,
the Krylov-Veretennikov formula
tells us that
$
\delta_{a} (X_{t})
=
\sum_{n=0}^{\infty}
J_{n} [ \delta_{a} (X_{t}) ]
$
(the convergence is in $\mathbb{D}_{2}^{-\infty}$)
for every $a \in \mathbb{R}$,
where
\begin{equation}
\begin{split}
\label{KryVer} 
J_{n} [ \delta_{a} (X_{t}) ]
=
\int_{0 \leqslant t_{1} < \cdots < t_{n} \leqslant t}
\Pi_{n}(x;t,a) [t_{1}, \cdots , t_{n}]
\mathrm{d}w(t_{1}) \cdots \mathrm{d}w(t_{n})
\end{split}
\end{equation}
and
\begin{equation*}
\begin{split}
&
\Pi_{n}(x;t,a) [t_{1}, \cdots , t_{n}] \\
&=
\int_{\mathbb{R}} \cdots \int_{\mathbb{R}}
p_{t_{1}}(x,z_{1})
A_{ z_{1} }
[
p_{t_{2}-t_{1}}
(
	z_{1},
	z_{2}
)
]
\cdots
A_{z_{n}}
[
p_{
	t - t_{n}
}
( z_{n}, a )
]
\mathrm{d}z_{1} \cdots \mathrm{d}z_{n} .
\end{split}
\end{equation*}
We remark that the stochastic integral in (\ref{KryVer})
is well defined since the square-integrability of each component
$(t_{1}, \cdots , t_{n}) \mapsto \Pi_{n} (x;t,a) [t_{1}, \cdots , t_{n}]$
of the chaos kernel
$\{ \Pi_{n} (x;t,a) \}_{n=0}^{\infty}$ is established in
\cite[Corollary 4.5]{Wa94b},
in a more general situation.
Since the generator $L=\frac{1}{2} A^{2}$
commutes with $A$,
we have
\begin{equation*}
\begin{split}
&
\Pi_{n}(x;t,a) [t_{1}, \cdots , t_{n}]
=
(
\mathrm{e}^{ t_{1} L }
A \mathrm{e}^{(t_{2}-t_{1}) L}
\cdots
A \mathrm{e}^{(t_{n}-t_{n-1}) L}
A \mathrm{e}^{(t-t_{n}) L}
)
\delta_{a} \\
&=
[
A^{n}
\exp (
	\{ t_{1} + (t_{2}-t_{1}) + \cdots + (t_{n}-t_{n-1}) + (t-t_{n}) \}
	L
)
]
\delta_{a}
=
(
A^{n}
\mathrm{e}^{ t L }
)
( \delta_{a} )
\end{split}
\end{equation*}
in the distributional sense.
By using
$
( \mathrm{e}^{ t L } \delta_{a} )(x)
=
p_{t} ( x , a )
$,
the formula (\ref{KryVer}) reduces to
\begin{equation*}
\begin{split}
J_{n} [ \delta_{a} (X_{t}) ]
=
[ A_{x}^{n} p_{t} (x,a) ]
\int_{0 \leqslant t_{1} < \cdots < t_{n} \leqslant t}
\mathrm{d}w(t_{1}) \cdots \mathrm{d}w(t_{n}) .
\end{split}
\end{equation*}
Therefore,
by using Corollary \ref{commute},
we have
\begin{equation*}
\begin{split}
&
\sigma (a)
\int_{0}^{1} \delta_{a} ( X_{s} ) \mathrm{d}s
=
\sum_{n=0}^{\infty}
\int_{0}^{1}
[
\sigma (a)
A_{x}^{n} p_{t} (x,a)
]
\int_{0 \leqslant t_{1} < \cdots < t_{n} \leqslant t}
\mathrm{d}w(t_{1}) \cdots \mathrm{d}w(t_{n})
\mathrm{d}t .
\end{split}
\end{equation*}
Hence we have
\begin{equation*}
\begin{split}
&
\big\Vert
\sigma (y)
\int_{0}^{1} \delta_{y} ( X_{t} ) \mathrm{d}t
-
\sigma (z)
\int_{0}^{1} \delta_{z} ( X_{t} ) \mathrm{d}t
\big\Vert_{2,s}^{2} \\
&=
\sum_{n=0}^{\infty}
(1+n)^{s}
\mathbf{E}
\big[
\Big\{
\int_{0}^{1}
[
	\sigma (y)
	A_{x}^{n} p_{t} (x,y)
	-
	\sigma (z)
	A_{x}^{n} p_{t} (x,z)
]
\int_{0 \leqslant t_{1} < \cdots < t_{n} \leqslant t}
\mathrm{d}w(t_{1}) \cdots \mathrm{d}w(t_{n})
\mathrm{d}t
\Big\}^{2}
\big] ,
\end{split}
\end{equation*}
and so we need to compute
\begin{equation}
\label{I_n-definition} 
\begin{split}
&
I_{n}
:=
\mathbf{E}
\big[
\Big\{
\int_{0}^{1}
[
	\sigma (y)
	A_{x}^{n} p_{t} (x,y)
	-
	\sigma (z)
	A_{x}^{n} p_{t} (x,z)
]
\int_{0 \leqslant t_{1} < \cdots < t_{n} \leqslant t}
\mathrm{d}w(t_{1}) \cdots \mathrm{d}w(t_{n})
\mathrm{d}t
\Big\}^{2}
\big] \\
&= 
2
\int_{0 \leqslant s < t \leqslant 1}
\mathrm{d}s\mathrm{d}t
\big\{
	\sigma (y)
	A_{x}^{n} p_{s} (x,y)
	-
	\sigma (z)
	A_{x}^{n} p_{s} (x,z)
\big\}
\big\{
	\sigma (y)
	A_{x}^{n} p_{t} (x,y)
	-
	\sigma (z)
	A_{x}^{n} p_{t} (x,z)
\big\} \\
&\hspace{10mm}\times
\mathbf{E}
\big[
\int_{0 \leqslant t_{1} < \cdots < t_{n} \leqslant s}
\mathrm{d}w(t_{1}) \cdots \mathrm{d}w(t_{n})
\int_{0 \leqslant t_{1} < \cdots < t_{n} \leqslant t}
\mathrm{d}w(t_{1}) \cdots \mathrm{d}w(t_{n})
\big] \\
&= 
\frac{ 2 }{ n! }
\int_{0 \leqslant s < t \leqslant 1}
s^{n}
\mathrm{d}s\mathrm{d}t
\big\{
	\sigma (y)
	A_{x}^{n} p_{s} (x,y)
	-
	\sigma (z)
	A_{x}^{n} p_{s} (x,z)
\big\}
\big\{
	\sigma (y)
	A_{x}^{n} p_{t} (x,y)
	-
	\sigma (z)
	A_{x}^{n} p_{t} (x,z)
\big\} .
\end{split}
\end{equation}

Since $\sigma$ is
Lipschitz continuous
(which is because of (H1)),
the vector field $A = \sigma (z) ( \mathrm{d}/\mathrm{d} z )$ is
{\it complete}, so that
one can associate the one-parameter group
of diffeomorphisms
$\{ \mathrm{e}^{sA} \}_{s \in \mathbb{R}}$.
Note that $\mathrm{e}^{sA}(x)$ for each $x \in \mathbb{R}$
is defined by
$
\frac{ \mathrm{d} }{ \mathrm{d}s }
\mathrm{e}^{ sA } (x)
=
\sigma ( \mathrm{e}^{ sA } (x) )
$
and
$
\mathrm{e}^{ sA } (x) \vert_{s=0} = x
$.

\begin{Lem} 
\label{intertwine} 
$
\frac{ \mathrm{d} }{ \mathrm{d}u }
\mathrm{e}^{uA} (x)
=
\sigma (x) \frac{ \partial }{ \partial x }
\mathrm{e}^{uA} (x)
$
for every $u\in \mathbb{R}$ and $x \in \mathbb{R}$.
\end{Lem} 
\begin{proof} 
By the homomorphism property:
$
\mathrm{e}^{(s+u)A}
=
\mathrm{e}^{sA} \circ \mathrm{e}^{uA}
$,
we have
$
\frac{ \mathrm{d} }{ \mathrm{d}s }
\mathrm{e}^{sA} (x)
=
\left.
\frac{ \mathrm{d} }{ \mathrm{d}u }
\right\vert_{u=0}
\mathrm{e}^{(s+u)A} (x)
=
\left.
\frac{ \mathrm{d} }{ \mathrm{d}u }
\right\vert_{u=0}
\mathrm{e}^{sA} ( \mathrm{e}^{uA}(x) )
=
\sigma (x) \frac{ \partial }{ \partial x }
\mathrm{e}^{sA} (x)
$.
\end{proof} 

To continue the calculation of $I_{n}$,
we shall
look at $A_{x} p_{t} (x,a)$.
The unique strong solution
$X=(X_{t})_{t \geqslant 0}$
is now expressed by
$X_{t} = \mathrm{e}^{w(t)A}(x)$
(To check this, just apply the It\^o formula for $\mathrm{e}^{w(t)A}(x)$).
Therefore by using Lemma \ref{intertwine}, we have
\begin{equation*}
\begin{split}
&
A_{x} p_{t} (x,a)
=
A_{x}
\mathbf{E} [ \delta_{a} ( \mathrm{e}^{w(t) A} (x) ) ] \\
&=
\sigma (x) \frac{\partial}{\partial x}
\int_{\mathbb{R}}
\delta_{a} ( \mathrm{e}^{\sqrt{t} u A} (x) )
\frac{ \mathrm{e}^{- u^{2}/2 } }{ \sqrt{2\pi} }
\mathrm{d}u \\
&=
\int_{\mathbb{R}}
\Big(
\sigma (x)
\frac{\partial}{\partial x}
\mathrm{e}^{\sqrt{t} u A} (x)
\Big)
\delta_{a}^{\prime}
(
	\mathrm{e}^{\sqrt{t} u A} (x)
)
\frac{ \mathrm{e}^{- u^{2}/2 } }{ \sqrt{2\pi} }
\mathrm{d}u \\
&=
\int_{\mathbb{R}}
\Big(
\frac{1}{\sqrt{t}}
\frac{\mathrm{d}}{\mathrm{d} u}
\mathrm{e}^{\sqrt{t} u A} (x)
\Big)
\delta_{a}^{\prime} ( \mathrm{e}^{\sqrt{t} u A} (x) )
\frac{ \mathrm{e}^{- u^{2}/2 } }{ \sqrt{2\pi} }
\mathrm{d}u \\
&=
\frac{1}{\sqrt{t}}
\int_{\mathbb{R}}
\Big(
\frac{\mathrm{d}}{\mathrm{d} u}
\delta_{a} ( \mathrm{e}^{\sqrt{t} u A} (x) )
\Big)
\frac{ \mathrm{e}^{- u^{2}/2 } }{ \sqrt{2\pi} }
\mathrm{d}u \\
&=
\int_{\mathbb{R}}
\delta_{a} ( \mathrm{e}^{\sqrt{t} u A} (x) )
\Big(
\frac{-1}{\sqrt{t}}
\frac{\mathrm{d}}{\mathrm{d} u}
\frac{ \mathrm{e}^{- u^{2}/2 } }{ \sqrt{2\pi} }
\Big)
\mathrm{d}u ,
\end{split}
\end{equation*}
where the integral is understood as the coupling
of the Schwartz distribution and the test function.
By the repetition of the above procedure,
we obtain
\begin{equation*}
\begin{split}
A_{x}^{n} p_{t} (x,a)
&= 
\int_{\mathbb{R}}
\delta_{a} ( \mathrm{e}^{\sqrt{t} u A} (x) )
\Big\{
\frac{ (-1)^{n} }{ t^{n/2} }
\frac{\mathrm{d}^{n}}{\mathrm{d} u^{n}}
\frac{ \mathrm{e}^{- u^{2}/2 } }{ \sqrt{2\pi} }
\Big\}
\mathrm{d}u \\
&= 
\int_{\mathbb{R}}
\delta_{a} ( \mathrm{e}^{\sqrt{t} u A} (x) )
\frac{ 1 }{ t^{n/2} }
H_{n} (u)
\frac{ \mathrm{e}^{- u^{2}/2 } }{ \sqrt{2\pi} }
\mathrm{d}u .
\end{split}
\end{equation*}
For each $t > 0$,
define a mapping
$\varphi : \mathbb{R} \to \mathbb{R}$
by
$
\varphi (u) := \mathrm{e}^{\sqrt{t} uA} (x)
$.
Since $\varphi$ is continuously differentiable and
$
\vert \varphi^{\prime} (u) \vert
=
\sqrt{t}
\vert \sigma ( \varphi (u) ) \vert
\geqslant
\sqrt{t} \lambda^{1/2}
$
for every $u$,
where $\lambda >0$ is
the constant
appeared in ({\rm H3}),
we see that $\varphi$ is
a diffeomorphism.
Hence we can apply the change of variables
$
b
= \varphi (u)
$
\footnote{when $t=1$, we see that
$
\varphi^{-1} (z)
=
\int_{x}^{z} \frac{\mathrm{d}a}{\sigma (a)}
$
and the stochastic process
$\varphi^{-1}(X_{s})$,
which is nothing but the Wiener process $w(s)$,
is known as the Lamperti transformation of $X$.}
(and then
$
\mathrm{d}u / \mathrm{d}
b
=
( \sqrt{t} \sigma ( b ) )^{-1}
$)
in the above integral to get
\begin{equation*}
\begin{split}
A_{x}^{n} p_{t} (x,a)
&= 
\int_{\mathbb{R}}
\delta_{a} ( b )
\frac{ 1 }{ t^{(n+1)/2} }
H_{n} ( \varphi^{-1}( b ) )
\frac{
	\mathrm{e}^{- (\varphi^{-1}( b ) )^{2}/2 }
}{
	\sqrt{ 2\pi }
	\sigma ( b )
}
\mathrm{d} b \\
&= 
\frac{ 1 }{ t^{(n+1)/2} }
H_{n} ( \varphi^{-1}(a) )
\frac{
	\mathrm{e}^{- (\varphi^{-1}(a))^{2}/2 }
}{
	\sqrt{ 2\pi }
	\sigma (a)
}.
\end{split}
\end{equation*}
By using
Lemma \ref{Hermite-facts}--(i),
we obtain the formula
\begin{equation}
\label{A^n_formula} 
A_{x}^{n} p_{t} (x,a)
=
\frac{
	(-1)^{n} t^{-(n+1)/2}
}{
	2\pi \sigma (a)
}
\int_{-\infty}^{\infty}
( i\xi )^{n}
\mathrm{e}^{-\xi^{2} /2}
\mathrm{e}^{ i \xi \varphi^{-1} (a) }
\mathrm{d}\xi .
\end{equation}

Therefore by using (\ref{A^n_formula}),
\begin{equation}
\label{Ap-difference} 
\begin{split}
&
\sigma (y)
A_{x}^{n} p_{t} (x,y)
-
\sigma (z)
A_{x}^{n} p_{t} (x,z) \\
&= 
\frac{ (-1)^{n} t^{-(n+1)/2} }{ 2\pi }
\int_{-\infty}^{\infty}
( i\xi )^{n}
\mathrm{e}^{-\xi^{2} /2}
\big\{
	\mathrm{e}^{ i \xi \varphi^{-1} (y) }
	-
	\mathrm{e}^{ i \xi \varphi^{-1} (z) }
\big\}
\mathrm{d}\xi
=: 
I\!\!I_{n}.
\end{split}
\end{equation}

\begin{Lem} 
\label{gimmick} 
For any $\alpha \in [0,1]$ and $\theta \in \mathbb{R}$, we have
$
\vert
	\mathrm{e}^{i\theta}
	-
	1
\vert
\leqslant
2\vert \theta \vert^{\alpha}
$
with the convention $0^{0} := 1$.
\end{Lem} 
\begin{proof} 
Let
$\alpha \in [0,1]$. If $\theta \in \mathbb{R} \setminus (-1,1)$,
then we see that
$
\vert
	\mathrm{e}^{i\theta}
	-
	1
\vert
\leqslant
2
\leqslant
2\vert \theta \vert^{\alpha} .
$
On the other hand, for $\theta \in (-1,1)$,
$
\vert
	\mathrm{e}^{i\theta}
	-
	1
\vert
\leqslant
\vert \theta \vert
\leqslant
\vert \theta \vert^{\alpha}
\leqslant
2\vert \theta \vert^{\alpha}
$
because $\alpha \in [0,1]$.
Thus we obtain
$
\vert
	\mathrm{e}^{i\theta}
	-
	1
\vert
\leqslant
2\vert \theta \vert^{\alpha}
$
for all $\theta \in \mathbb{R}$.
\end{proof} 

Let $\beta \in [0,1)$ be arbitrary.
By Lemma \ref{gimmick}, we have
\begin{equation*}
\begin{split}
\vert I\!\!I_{n} \vert
&
\leqslant
\frac{ t^{-(n+1)/2} }{ \pi }
\vert
	\varphi^{-1} (y)
	-
	\varphi^{-1} (z)
\vert^{\beta}
\int_{-\infty}^{\infty}
\vert \xi \vert^{n+\beta}
\mathrm{e}^{-\xi^{2} /2}
\mathrm{d}\xi \\
&=
\frac{
	t^{-(n+1)/2} 2^{(n+\beta -1)/2}
}{
	\pi
}
\Gamma \Big( \frac{n+\beta +1}{2} \Big)
\vert
	\varphi^{-1} (y)
	-
	\varphi^{-1} (z)
\vert^{\beta} .
\end{split}
\end{equation*}
Since
$
\frac{
	\mathrm{d}
}{
	\mathrm{d}
	b
}
\varphi^{-1} ( b )
=
\frac{1}{
	\sqrt{t}
	\sigma ( b )
}
$,
we have
$
\vert
\varphi^{-1}(y) - \varphi^{-1}(z)
\vert
=
t^{-1/2}
\vert
\int_{y}^{z}
\sigma ( b )^{-1}
\mathrm{d} b
\vert
$,
so that
\begin{equation}
\label{II_1-est} 
\begin{split}
\vert I\!\!I_{n} \vert
&
\leqslant
\frac{
	t^{-(n+\beta +1)/2} 2^{(n+\beta -1)/2}
}{
	\pi \lambda^{\beta /2}
}
\Gamma \Big( \frac{n+\beta +1}{2} \Big)
\vert
	y - z
\vert^{\beta} ,
\end{split}
\end{equation}
where it is recalled again that $\lambda > 0$ is what appeared in {\rm (H3)}.

Substituting (\ref{II_1-est})
into (\ref{Ap-difference}),
\begin{equation*}
\begin{split}
&
\vert
	\sigma (y)
	A_{x}^{n} p_{t} (x,y)
	-
	\sigma (z)
	A_{x}^{n} p_{t} (x,z)
\vert
\leqslant 
c_{2}
t^{-(n+\beta +1)/2} 2^{(n+\beta -1)/2}
\Gamma \Big( \frac{n+\beta +1}{2} \Big)
\vert y-z \vert^{\beta}
\end{split}
\end{equation*}
where
$
c_{2}
:=
( \pi \lambda^{ \beta /2 } )^{-1}
$.
Note that $c_{2}$ does not depend on $y$ and $z$.
With putting
$$
c_{3} (n)
:=
\Gamma \Big( \frac{n+\beta +1}{2} \Big) ,
$$
we have obtained
\begin{equation}
\label{Ap-difference'} 
\begin{split}
&
\vert
	\sigma (y)
	A_{x}^{n} p_{t} (x,y)
	-
	\sigma (z)
	A_{x}^{n} p_{t} (x,z)
\vert
\leqslant 
c_{2}
t^{-(n+\beta +1)/2} 2^{(n+\beta -1)/2}
c_{3}(n)
\vert y-z \vert^{\beta} .
\end{split}
\end{equation}

Now, by substituting
(\ref{Ap-difference'})
into
(\ref{I_n-definition}),
we have
\begin{equation*}
\begin{split}
I_{n}
&\leqslant 
\frac{ (c_{2})^{2} }{ n! }
2^{n+\beta}
c_{3} (n)^{2}
\vert y-z \vert^{2\beta}
\int_{0}^{1}
s^{n}
s^{-(n+\beta +1)/2}
\mathrm{d}s
\int_{s}^{1} t^{-(n+\beta +1)/2} \mathrm{d}t .
\end{split}
\end{equation*}
Note that the last iterated integral
is finite because $\beta < 1$,
and gives
$$
\int_{0}^{1}
s^{n}
s^{-(n+\beta +1)/2}
\mathrm{d}s
\int_{s}^{1} t^{-(n+\beta +1)/2} \mathrm{d}t
=
\frac{
	2
}{
	( 1 - \beta )
	( n-\beta +1 )
}.
$$
Hence we have
\begin{equation*}
\begin{split}
I_{n}
&\leqslant 
\frac{ (c_{2})^{2} }{ ( 1 - \beta ) }
\vert y-z \vert^{2\beta}
\frac{
	2^{n+\beta +1}
	c_{3} (n)^{2}
}{
	n! ( n-\beta +1 )
} .
\end{split}
\end{equation*}

Finally we have
\begin{equation*}
\begin{split}
&
\big\Vert
	\sigma (y)
	\int_{0}^{1} \delta_{y} ( X_{t} ) \mathrm{d}t
	-
	\sigma (z)
	\int_{0}^{1} \delta_{z} ( X_{t} ) \mathrm{d}t
\big\Vert_{2,s}^{2} \\
&=
\sum_{n=0}^{\infty}
(1+n)^{s}
I_{n}
\leqslant
\frac{ (c_{2})^{2} }{ ( 1 - \beta ) }
\vert y-z \vert^{2\beta}
\sum_{n=0}^{\infty}
(1+n)^{s}
\frac{
	2^{n+\beta +1}
	c_{3} (n)^{2}
}{
	n! ( n-\beta +1 )
}.
\end{split}
\end{equation*}
By  Stirling's formula, we see that
the quantity
\begin{equation*}
\begin{split}
&
(1+n)^{s}
\frac{
	2^{n+\beta +1}
	c_{3} (n)^{2}
}{
	n! ( n-\beta +1 )
}
=
(1+n)^{s}
\frac{
	2^{n+\beta +1}
}{
	n! ( n-\beta +1 )
}
\Gamma \Big( \frac{n+\beta +1}{2} \Big)^{2}
\end{split}
\end{equation*}
behaves like
\begin{equation*}
\begin{split}
&
(1+n)^{s}
\frac{
	2^{n+\beta +1}
}{
	( n-\beta +1 )
	\sqrt{2\pi n} ( \frac{n}{\mathrm{e}} )^{n}
}
\times
\big\{
\sqrt{\pi (n+\beta -1)}
	\Big(
		\frac{n+\beta -1}{2\mathrm{e}}
	\Big)^{\frac{n+\beta -1}{2}}
\big\}^{2} \\
&= 
\frac{
	(1+n)^{s}
	2^{n+\beta +1}
}{
	( n-\beta +1 )
	\sqrt{2\pi n}
	( \frac{n}{\mathrm{e}} )^{n}
}
\pi (n+\beta -1)
\Big(
	\frac{n+\beta -1}{2\mathrm{e}}
\Big)^{n+\beta -1}
=
O(n^{s+\beta - \frac{3}{2}})
\end{split}
\end{equation*}
as $n\to \infty$ for each $\beta$.
Hence the sum converges if $s+\beta - \frac{3}{2} < -1$,
i.e.,
$s+ \beta < \frac{1}{2}$.
The proof of Theorem
\ref{Main_Thm2}
finishes.

\begin{proof}[Proof of Corollary \ref{Main_Cor1}]
This is clear from
Theorem
\ref{Main_Thm2}
and the inequality
\begin{equation*}
\begin{split}
\big\vert
\sigma (y)
\int_{0}^{1} p_{t} (x,y) \mathrm{d}t
-
\sigma (z)
\int_{0}^{1} p_{t} (x,z) \mathrm{d}t
\big\vert
&=
\big\vert
\mathbf{E}
[
	\sigma (y)
	\int_{0}^{1} \delta_{y} ( X_{t} ) \mathrm{d}t
	-
	\sigma (z)
	\int_{0}^{1} \delta_{z} ( X_{t} ) \mathrm{d}t
]
\big\vert \\
&\leqslant
\Vert
	\sigma (y)
	\int_{0}^{1} \delta_{y} ( X_{t} ) \mathrm{d}t
	-
	\sigma (z)
	\int_{0}^{1} \delta_{z} ( X_{t} ) \mathrm{d}t
\Vert_{2, -1/2} .
\end{split}
\end{equation*}
\end{proof} 

\section{It\^o's Formula for Generalized Wiener Functionals}
\label{Generalized_sec} 

Let $w = (w(t))_{t \geqslant 0}$
be the $d$-dimensional Wiener process
with $w(0) = 0$
and
let $(\mathcal{F}_{t}^{w})_{t \geqslant 0}$
be the filtration generated by $w$:
$
\mathcal{F}_{t}^{w}
:=
\sigma ( w(s) : 0 \leqslant s \leqslant t )
$,
for $t \geqslant 0$.
Similarly to the previous
section,
we fix $x \in \mathbb{R}^{d}$ and
consider the following $d$-dimensional
stochastic differential equation
\begin{equation}
\label{SDE2} 
\mathrm{d} X_{t}
=
\sigma (X_{t}) \mathrm{d} w(t)
+
b ( X_{t} ) \mathrm{d} t,
\quad
X_{0} = x \in \mathbb{R}^{d}.
\end{equation}
Also in this
section,
we assume the conditions
{\rm (H1)} and {\rm (H2)}.
We denote by $\{ X(t,x,w) \}_{t \geqslant 0}$
a unique strong solution
$X = (X_{t})_{t \geqslant 0}$ to (\ref{SDE2}).

\subsection{Stochastic integrals of pull-backs by diffusion}
\label{Stochastic_Integral} 

For each $J \in \mathbb{D}^{\infty}$,
we will denote by $(DJ)^{i}$
the $i$-th component of $DJ \in \mathbb{D}^{\infty}(H)$:
$
DJ = ( (DJ)^{1}, \cdots , (DJ)^{d} )
$
(Recall $H$ is the Cameron-Martin subspace
of the $d$-dimensional Wiener space).
For each $t \in [0,T]$,
the evaluation map
$\mathrm{ev}_{t}: H \ni h \mapsto h(t) \in \mathbb{R}^{d}$
naturally induces a map
$
\mathrm{id} \otimes \mathrm{ev}_{t} :
L_{2}(H) \cong L_{2} \otimes H \to L_{2} \otimes \mathbb{R}^{d}
$
and then we write
$
D_{t}J
:=: ( (D_{t}J)^{1}, \cdots , (D_{t}J)^{d} )
:= \frac{\mathrm{d}}{\mathrm{d}t}( \mathrm{id} \otimes \mathrm{ev}_{t} ) (DJ)
$
for a.a. $t \in [0,T]$.

Let
$
\Lambda \in \mathscr{D}^{\prime} (\mathbb{R}^{d})
$.
Throughout this
section,
we assume either one of
\begin{itemize}
\item[$\circ$]
{\rm (H1)}, {\rm (H2)} and $\Lambda \in \mathscr{S}^{\prime} (\mathbb{R}^{d})$;
\vspace{2mm}
\item[$\circ$]
{\rm (H1)}, {\rm (H2)}, {\rm (H4)} and $\Lambda \in \mathscr{E}^{\prime} (\mathbb{R}^{d})$.
\end{itemize}
Then $\Lambda ( X(t,x,w) )$ is defined as a generalized Wiener functional.
If
$
\int_{0}^{T}
\Vert \Lambda ( X(t,x,w) ) \Vert_{2,-k}^{2}
\mathrm{d}t
$
is finite for some $k \in \mathbb{N}$,
we
define the {\it stochastic integrals}
$
\int_{0}^{T}
\Lambda ( X(t,x,w) )
\mathrm{d} w^{i}(t)
$,
$i=1, \cdots , d$
as elements in $\mathbb{D}^{-\infty}$
via the pairing
\begin{equation}
\label{stoc_int} 
\begin{split}
\mathbf{E}
[
	\Big(
	\int_{0}^{T}
	\Lambda ( X(t,x,w) )
	\mathrm{d} w^{i}(t)
	\Big)
	J
]
&=
\int_{0}^{T}
\mathbf{E}
[
	\Lambda ( X(t,x,w) )
	( D_{t} J )^{i}
]
\mathrm{d}t,
\end{split}
\end{equation}
for $J \in \mathbb{D}^{\infty}$ and $i=1, 2, \cdots , d$.
We define the stochastic integral
$
\int_{s}^{T}
\Lambda ( X(t,x,w) )
\mathrm{d}w^{i}(t)
$
similarly
for each $0 < s \leqslant T$.

This pairing is well defined because of the following:

\begin{Prop} 
\label{ver_pair} 
For each $k\in \mathbb{N}$ and $i =1, 2, \cdots , d$,
there exists a constant $C>0$ such that
\begin{equation*}
\begin{split}
&
\int_{0}^{T}
\vert
\mathbf{E}
[
	\Lambda ( X(t,x,w) )
	( D_{t} J )^{i}
]
\vert
\mathrm{d}t
\leqslant
C
\Big\{
	\int_{0}^{T}
	\Vert
		\Lambda ( X(t,x,w) )
	\Vert_{2,-k}^{2}
	\mathrm{d}t
\Big\}^{1/2}
\Vert
	J
\Vert_{2,k+1}
\end{split}
\end{equation*}
for all $J \in \mathbb{D}^{\infty}$.
\end{Prop} 
\begin{proof} 
For simplicity of notation, we prove in the case $d=1$.
For each $k \in \mathbb{N}$ and $J \in \mathbb{D}^{\infty}$,
\begin{equation}
\label{schw_ineq} 
\begin{split}
\int_{0}^{T}
\vert
\mathbf{E}
[
	\Lambda ( X_{t} )
	D_{t} J
]
\vert
\mathrm{d}t
&\leqslant
\int_{0}^{T}
\Vert
	\Lambda ( X_{t} )
\Vert_{2,-k}
\Vert
	D_{t} J
\Vert_{2,k}
\mathrm{d}t \\
&\leqslant
\Big\{
	\int_{0}^{T}
	\Vert
		\Lambda ( X_{t} )
	\Vert_{2,-k}^{2}
	\mathrm{d}t
\Big\}^{1/2}
\Big\{
	\int_{0}^{T}
	\Vert
		D_{t} J
	\Vert_{2,k}^{2}
	\mathrm{d}t
\Big\}^{1/2} .
\end{split}
\end{equation}
By Meyer's inequality, there exist constants
$c^{\prime}, C^{\prime}>0$ such that
\begin{equation*}
c^{\prime}
\Vert D^{k^{\prime}} F \Vert_{2}
\leqslant
\Vert F \Vert_{2,k^{\prime}}
\leqslant
C^{\prime}
\sum_{l=0}^{k^{\prime}}
\Vert D^{l}F \Vert_{2},
\quad
\text{for all $F \in \mathbb{D}_{2}^{k^{\prime}}$}
\end{equation*}
for $k^{\prime} = 1,2, \cdots , k+1$.
Therefore we have
\begin{equation*}
\begin{split}
&
\Vert D_{t} J \Vert_{2,k}^{2}
\leqslant
( C^{\prime} )^{2}
\Big\{
\sum_{l=0}^{k}
\Vert D^{l} D_{t} J \Vert_{2}
\Big\}^{2}
\leqslant
( C^{\prime} )^{2}
(k+1)
\sum_{l=0}^{k}
\Vert D^{l} D_{t} J \Vert_{2}^{2} \\
&=
( C^{\prime} )^{2}
(k+1)
\sum_{l=0}^{k}
\mathbf{E}
[
	\Vert D^{l} D_{t} J \Vert_{H^{\otimes l}}^{2}
] \\
&=
C^{\prime\prime}
\Big\{
\mathbf{E}
[
	( D_{t} J )^{2}
]
+
\sum_{l=1}^{k}
\int_{0}^{T} \cdots \int_{0}^{T}
\mathbf{E}
[
	( D_{s_{l}} \cdots D_{s_{1}} D_{t} J )^{2}
]
\mathrm{d}s_{1} \cdots \mathrm{d}s_{l}
\Big\},
\end{split}
\end{equation*}
where $C^{\prime\prime} := ( C^{\prime} )^{2} (k+1)$,
so that
\begin{equation*}
\begin{split}
\int_{0}^{T}
\Vert D_{t} J \Vert_{2,k}^{2}
\mathrm{d}t
&\leqslant
C^{\prime\prime}
\sum_{l=1}^{k+1}
\int_{0}^{T} \cdots \int_{0}^{T}
\mathbf{E}
[
	( D_{u_{l}} \cdots D_{u_{1}} J )^{2}
]
\mathrm{d}u_{1} \cdots \mathrm{d}u_{l} \\
&\leqslant
C^{\prime\prime}
\sum_{l=0}^{k+1}
\Vert D^{l} J \Vert_{2}^{2}
\leqslant
( c^{\prime} )^{-2}
C^{\prime\prime}
\Vert J \Vert_{2,k+1}^{2} .
\end{split}
\end{equation*}
Hence by substituting this into (\ref{schw_ineq}),
we get the result.
\end{proof} 

\begin{Rm} 
\label{referee-2} 
(a)
As is easily seen,
the stochastic integral
$
\int_{0}^{T} \Lambda (X_{t}) \mathrm{d}w^{i}(t)
$
has another expression:
$$
\int_{0}^{T} \Lambda (X_{t}) \mathrm{d}w^{i}(t)
=
D^{*}
[
(
	0, \cdots ,
	\underbrace{
	\int_{0}^{\bullet} \Lambda (X_{u}) \mathrm{d}u
	}_{\text{$i$-th position}} ,
	\cdots , 0
)
].
$$
Hence it coincides with the Skorokhod integral
as long as the object
standing on the right of
$D^{*}$ lies in the domain
$\mathbb{D}_{2}^{1}(H)$
and then automatically coincides
with the It\^o integral because of the adaptedness.
In fact, in view of the Clark-Ocone formula,
every $J \in L_{2}$ can be written as
$
J
=
\mathbf{E}[J]
+
\sum_{i=1}^{d}
\int_{0}^{T}
\mathbf{E} [ (D_{t}J)^{i} \vert \mathcal{F}_{t}^{w} ]
\mathrm{d} w^{i}(t)
$
and then (\ref{stoc_int}) is just the It\^o isometry.
Therefore, our stochastic integral
may be
a natural extension of
the classical anticipative stochastic integral
to this distributional setting.

(b)
By Proposition \ref{ver_pair},
we have
$
\int_{0}^{T} \Lambda ( X_{t} ) \mathrm{d}w^{i}(t)
\in
\mathbb{D}_{2}^{-(k+1)}
$
for $i=1, 2, \cdots , d$
provided
$
\int_{0}^{T}
\Vert
	\Lambda ( X_{t} )
\Vert_{2,-k}^{2}
\mathrm{d}t
< +\infty
$.
\end{Rm} 

The following is the main result in this
section.
This is a version of
a result by
Uemura \cite[Proposition 1]{Ue}.

\begin{Thm} 
\label{reg-pres} 
Let $s \in \mathbb{R}$, $p \geqslant 2$
and assume that
$
(0,T] \ni t
\mapsto
\Lambda ( X(t,x,w) ) \in \mathbb{D}_{p}^{s}
$
is continuous.
Then we have
$$
\int_{0}^{T}
\Lambda ( X(t,x,w) )
\mathrm{d}w^{i}(t)
\in
\mathbb{D}_{p}^{s},
\quad
\text{for $i=1, 2, \cdots , d$}
$$
provided either one of the following
\begin{itemize}
\item[(i)]
$
\lim_{t \downarrow 0}
\Vert
	\Lambda ( X(t,x,w) )
\Vert_{p,s}
= 0
$.

\item[(ii)]
$s \geqslant 0$ and
$
\int_{0}^{T}
\Vert
	\Lambda ( X(t,x,w) )
\Vert_{p,s}^{2}
\mathrm{d}t
< \infty
$.
\end{itemize}
\end{Thm} 

\begin{Rm} 
\begin{itemize}
\item[(a)]
See also a remark just after Lemma \ref{frac_dom}
for verification of the continuity assumption.

\item[(b)]
From the proof of Theorem \ref{reg-pres},
we would find that
$$
\int_{t_{0}}^{T}
\Lambda
(
	X_{t}
)
\mathrm{d}w^{i}(t)
\in
\mathbb{D}_{p}^{s}
\quad
\text{for any $t_{0} > 0$}
$$
and $i=1, 2, \cdots , d$ if
$
(0,T] \ni t
\mapsto
\Lambda (X_{t}) \in \mathbb{D}_{p}^{s}
$
is continuous.
\end{itemize}
\end{Rm} 

The proof of Theorem \ref{reg-pres}
mainly consists of the following series of Propositions
\ref{discrete-est},
\ref{dyad-Cauchy}
and
\ref{int_as_limit}.
We will give the proof at the last of this
section.

Before the next definition, we note that
$
\mathbf{E} [ F \vert \mathcal{F}_{t}^{w} ]
\in \mathbb{D}^{\infty}
$
for every $t \geqslant 0$
if $F \in \mathbb{D}^{\infty}$.

\begin{Def} 
Let $t \geqslant 0$.
We say that a generalized Wiener functional
$F \in \mathbb{D}^{-\infty}$
is
{\it $\mathcal{F}_{t}^{w}$-measurable}
if it holds that
$
\mathbf{E}
[ FG ]
=
\mathbf{E}
[
	F \mathbf{E}[G \vert \mathcal{F}_{t}^{w}]
]
$
for any $G \in \mathbb{D}^{\infty}$.
\end{Def} 

\begin{Prop} 
\label{discrete-est} 
Let $s \in \mathbb{R}$ and $p \geqslant 2$.
Then there exists
$
c = c(p,s,T) > 0
$
such that,
for any
mapping
$
F: (0,T] \ni t \mapsto F(t) \in \mathbb{D}_{p}^{s}
$
with
$F(t)$ is $\mathcal{F}_{t}^{w}$-measurable
for any $t \in (0,T]$,
any division
$
0 = t_{0} < t_{1} < \cdots < t_{n} = T
$,
and any
$i=1, 2, \cdots , d$,
we have
\begin{equation*}
\begin{split}
\Vert
\sum_{k=2}^{n}
F(t_{k-1})
( w^{i}(t_{k}) - w^{i}(t_{k-1}) )
\Vert_{p,s}^{p}
\leqslant
c
\sum_{k=2}^{n}
\Vert F(t_{k-1}) \Vert_{p,s}^{p}
( t_{k} - t_{k-1} ).
\end{split}
\end{equation*}
\end{Prop} 
\begin{proof} 
Let
$
0 = t_{0} < t_{1} < \cdots < t_{n} = T
$
be any division of $[0,T]$
and set
$$
\Phi
:=
\sum_{k=2}^{n}
F ( t_{k-1} )
( w^{i}(t_{k}) - w^{i}(t_{k-1}) ) .
$$
	To calculate
$
\Vert
\Phi
\Vert_{p,s}
$,
we begin with the chaos expansion of each
$
F ( t_{k-1} )
( w^{i}(t_{k}) - w^{i}(t_{k-1}) )
$.
Noting
$
\mathbf{E}
[
F ( t_{k-1} )
( w^{i}(t_{k}) - w^{i}(t_{k-1}) )
] = 0
$
and
\begin{equation}
\label{reduction} 
J_{m} [ F ( t_{k-1} )
( w^{i}(t_{k}) - w^{i}(t_{k-1}) ) ]
=
J_{m-1} [ F ( t_{k-1} ) ]
( w^{i}(t_{k}) - w^{i}(t_{k-1}) )
\end{equation}
for $m \geqslant 1$
(here we have used the condition that $F(t_{k-1})$ is
$\mathcal{F}_{t_{k-1}}^{w}$-measurable),
one finds that the chaos expansion is given by
\begin{equation*}
\begin{split}
&
F ( t_{k-1} )
( w^{i}(t_{k}) - w^{i}(t_{k-1}) )
=
\sum_{m=1}^{\infty}
J_{m-1}[ F ( t_{k-1} ) ]
( w^{i}(t_{k}) - w^{i}(t_{k-1}) ),
\end{split}
\end{equation*}
where
$
F ( t_{k-1} )
=
\sum_{m=0}^{\infty}
J_{m}[ F ( t_{k-1} ) ]
$
is the chaos expansion of $F ( t_{k-1} )$.
Hence, by using (\ref{reduction}),
we have
\begin{equation*}
\begin{split}
&
( I - \mathcal{L} )^{s/2}
\sum_{k=2}^{n}
F ( t_{k-1} )
( w^{i}(t_{k}) - w^{i}(t_{k-1}) ) \\
&=
\sum_{k=2}^{n}
\sum_{m=1}^{\infty}
( 1 + m )^{s/2}
J_{m-1}[ F ( t_{k-1} ) ]
( w^{i}(t_{k}) - w^{i}(t_{k-1}) ) \\
&=
\sum_{m=0}^{\infty}
\frac{ ( 2 + m )^{s/2} }{ ( 1 + m )^{s/2} }
\sum_{k=2}^{n}
( 1 + m )^{s/2}
J_{m}[ F ( t_{k-1} ) ]
( w^{i}(t_{k}) - w^{i}(t_{k-1}) ) .
\end{split}
\end{equation*}
By Meyer's $L_{p}$-multiplier theorem
(see e.g. \cite[Chapter V, Section 8, Lemma 8.2]{IW}),
there exists $c^{\prime} = c^{\prime} (p,s) > 0$ such that
\begin{equation*}
\begin{split}
&
\Vert
\sum_{k=2}^{n}
F ( t_{k-1} )
( w^{i}(t_{k}) - w^{i}(t_{k-1}) )
\Vert_{p,s} \\
&=
\Vert
( I - \mathcal{L} )^{s/2}
\sum_{k=2}^{n}
F ( t_{k-1} )
( w^{i}(t_{k}) - w^{i}(t_{k-1}) )
\Vert_{p} \\
&\leqslant
c^{\prime}
\Vert
\sum_{m=0}^{\infty}
\sum_{k=2}^{n}
( 1 + m )^{s/2}
J_{m}[ F ( t_{k-1} ) ]
( w^{i}(t_{k}) - w^{i}(t_{k-1}) )
\Vert_{p} \\
&=
c^{\prime}
\big\Vert
\sum_{k=2}^{n}
\big[
(I-\mathcal{L})^{s/2}
F ( t_{k-1} )
\big]
( w^{i}(t_{k}) - w^{i}(t_{k-1}) )
\big\Vert_{p}.
\end{split}
\end{equation*}
Note that
$
(I-\mathcal{L})^{s/2}
F ( t_{k-1} )
\in L_{2}
$
and is $\mathcal{F}_{t_{k-1}}^{w}$-measurable.
Hence the last quantity in the $L_{p}$-norm
can be written as
$$
\int_{0}^{T}
\sum_{k=2}^{n}
\big[
(I-\mathcal{L})^{s/2}
F ( t_{k-1} )
\big]
1_{ (t_{k-1}, t_{k}] } (t)
\mathrm{d}w^{i}(t) .
$$
Thus by using the Burkholder-Davis-Gundy inequality,
we get
\begin{equation*}
\begin{split}
&
\Vert
\sum_{k=2}^{n}
F ( t_{k-1} )
( w^{i}(t_{k}) - w^{i}(t_{k-1}) )
\Vert_{p,s}^{p}
\leqslant
c^{\prime\prime}
\mathbf{E}
\big[
\Big\{
	\int_{0}^{T}
	\sum_{k=2}^{n}
	\big[
	(I-\mathcal{L})^{s/2}
	F ( t_{k-1} )
	\big]^{2}
	1_{ (t_{k-1}, t_{k}] } (t)
\mathrm{d}t
\Big\}^{p/2}
\big]
\end{split}
\end{equation*}
for some constant
$c^{\prime\prime} > 0$.
Finally, using the assumption $p \geqslant 2$
and the Jensen inequality,
we reached
\begin{equation*}
\begin{split}
\Vert
\sum_{k=2}^{n}
F ( t_{k-1} )
( w^{i}(t_{k}) - w^{i}(t_{k-1}) )
\Vert_{p,s}^{p}
\leqslant
c
\sum_{k=2}^{n}
\Vert
	F ( t_{k-1} )
\Vert_{p,s}^{p}
( t_{k} - t_{k-1} )
\end{split}
\end{equation*}
for some constant
$c > 0$.
\end{proof} 

Given $n \in \mathbb{N}$, $i=1, 2, \cdots , d$
and the dyadic division
$
\{ t_{k} := kT/2^{n} \}_{k=0}^{2^{n}}
$
of $[0,T]$,
we define
$$
\Phi_{n}
:=
\sum_{k=2}^{2^{n}}
\Lambda ( X ( t_{k-1}, x, w ) )
( w^{i}(t_{k}) - w^{i}(t_{k-1}) ).
$$

\begin{Prop} 
\label{dyad-Cauchy} 
Let $s \in \mathbb{R}$, $p \geqslant 2$
and suppose that
\begin{itemize}
\item[(i)]
$
(0,T] \ni t
\mapsto
\Lambda ( X(t,x,w) ) \in \mathbb{D}_{p}^{s}
$
is continuous and

\item[(ii)]
$
\lim_{t \downarrow 0}
\Vert
	\Lambda ( X(t,x,w) )
\Vert_{p,s}
= 0
$.
\end{itemize}
Then we have
$
\Vert \Phi_{n} - \Phi_{m} \Vert_{p,s}
\to 0
$
as $n,m \to \infty$.
\end{Prop} 
\begin{proof} 
Suppose that $n<m$ and let
$t_{k} := kT/2^{n}$
and
$u_{l} := lT/2^{m}$.
Then we have
$$
\Phi_{n} - \Phi_{m}
=
\sum_{k=2}^{2^{n}}
\sum_{
	\substack{
		l \in \{ 0,1, \cdots , 2^{m} \} : \\
		t_{k-1} < u_{l} \leqslant t_{k}
	}
}
[ \Lambda ( X_{t_{k-1}} ) - \Lambda ( X_{u_{l}} ) ]
( w^{i}(u_{l}) - w^{i}(u_{l-1}) )
$$
and hence by Proposition \ref{discrete-est},
we obtain
\begin{equation*}
\begin{split}
\Vert
	\Phi_{n} - \Phi_{m}
\Vert_{p,s}^{p}
\leqslant
c
\sum_{k=2}^{2^{n}}
\sum_{
	\substack{
		l \in \{ 0,1, \cdots , 2^{m} \} : \\
		t_{k-1} < u_{l} \leqslant t_{k}
	}
}
\Vert
	\Lambda ( X_{t_{k-1}} ) - \Lambda ( X_{u_{l}} )
\Vert_{p, s}^{p}
( u_{l} - u_{l-1} ),
\end{split}
\end{equation*}
for some constant $c>0$.
By the assumption,
the mapping
$
( 0, T ] \ni t \mapsto \Lambda (X_{t}) \in \mathbb{D}_{p}^{s}
$
is uniformly continuous,
from which, we easily get
$
\Vert
	\Phi_{n} - \Phi_{m}
\Vert_{p,s}
\to 0
$
as $n \to \infty$.
\end{proof} 

\begin{Prop} 
\label{int_as_limit} 
Suppose $k \in \mathbb{Z}_{\geqslant 0}$
and
\begin{itemize}
\item[(i)]
$
(0,T] \ni t
\mapsto
\Lambda ( X(t,x,w) ) \in \mathbb{D}_{2}^{-k}
$
is continuous;

\item[(ii)]
$
\lim_{t \downarrow 0}
\Vert \Lambda ( X(t,x,w) ) \Vert_{2,-k}
= 0
$.
\end{itemize}
Then we have
$$
\Vert
	\Phi_{n}
	-
	\int_{0}^{T}
	\Lambda ( X(t,x,w) )
	\mathrm{d}w^{i}(t)
\Vert_{2,-(k+1)}
\to 0
\quad
\text{as $n \to \infty$.}
$$
\end{Prop} 
\begin{proof} 
The same argument in Proposition \ref{ver_pair}
leads us to
\begin{equation*}
\begin{split}
&
\Vert
	\Phi_{n}
		-
		\int_{0}^{T}
		\Lambda ( X_{t} )
		\mathrm{d}w^{i}(t)
\Vert_{2,-(k+1)} \\
&\leqslant
\Big\{
	\int_{0}^{t_{1}}
	\Vert
		\Lambda (X_{t})
	\Vert_{2,-k}^{2}
	\mathrm{d}t
	+
	\sum_{k=2}^{2^{n}}
	\int_{t_{k-1}}^{t_{k}}
	\Vert
		\Lambda (X_{t_{k-1}}) - \Lambda (X_{t})
	\Vert_{2,-k}^{2}
	\mathrm{d}t
\Big\}^{1/2}.
\end{split}
\end{equation*}
By the assumption, we have
$
(0,T] \ni t
\mapsto
\Lambda (X_{t}) \in \mathbb{D}_{2}^{-k}
$
is uniformly continuous,
and hence the above quantity
converges to zero as $n \to \infty$.
\end{proof} 

\begin{proof}[Proof of Theorem \ref{reg-pres}]
(i)
Suppose that
$
\lim_{t \downarrow 0}
\Vert \Lambda (X_{t}) \Vert_{p,s}
= 0
$.
In the first,
suppose that $s<0$.
Let $k$ be the
largest
integer, not exceeding $s$, i.e.,
$k = \max \{ k^{\prime} \in \mathbb{Z} : k^{\prime} \leqslant s \}$.
By the assumption,
the stochastic integral
$\int_{0}^{T} \Lambda (X_{t}) \mathrm{d}w^{i}(t)$
is defined as an element in $\mathbb{D}_{2}^{k-1}$
(see Remark \ref{referee-2}--(b)).
On the other hand,
Proposition \ref{dyad-Cauchy}
tells us that
$\{ \Phi_{n} \}_{n=1}^{\infty}$
is a Cauchy sequence in $\mathbb{D}_{p}^{s}$
and hence converges to some $\Phi \in \mathbb{D}_{p}^{s}$.
Hence $\Phi_{n}$ converges to $\Phi$
also in $\mathbb{D}_{2}^{k-1}$.
Now by Proposition \ref{int_as_limit},
it must be
$
\int_{0}^{T} \Lambda (X_{t}) \mathrm{d}w^{i}(t)
=
\lim_{n \to \infty} \Phi_{n}
=
\Phi
$.
Thus we have
$
\int_{0}^{T} \Lambda (X_{t}) \mathrm{d}w^{i}(t)
\in
\mathbb{D}_{p}^{s}
$.

Second,
suppose that $s \geqslant 0$.
In this case,
it is clearly satisfied that
$
\int_{0}^{T}
\Vert \Lambda (X_{t}) \Vert_{p,s}^{p}
\mathrm{d}t
< +\infty
$,
and hence it reduces to the case (ii).

(ii)
Suppose that $s \geqslant 0$
and
$
\int_{0}^{T}
\Vert \Lambda (X_{t}) \Vert_{p,s}^{p}
\mathrm{d}t
< +\infty
$.
Then by Proposition \ref{discrete-est},
we have
\begin{equation*}
\begin{split}
\Vert \Phi_{n} \Vert_{p,s}^{p}
&\leqslant
\mathrm{const.}
\sum_{k=2}^{2^{n}}
\Vert \Lambda (X_{t_{k-1}}) \Vert_{p,s}^{p}
( t_{k} - t_{k-1} )
\to
\int_{0}^{T}
\Vert \Lambda (X_{t}) \Vert_{p,s}^{p}
\mathrm{d}t,
\end{split}
\end{equation*}
as $n \to \infty$.
Here, $t_{k} = kT/2^{n}$, $k=0,1,\cdots , 2^{n}$.
Hence $\{ \Phi_{n} \}_{n=1}^{\infty}$
forms a bounded family in $\mathbb{D}_{p}^{s}$,
so that by Alaoglu's theorem
(for dual spaces of separable normed spaces),
there exists a subsequence
$\{ \Phi_{n_{l}} \}_{l=1}^{\infty}$
and $\Phi \in \mathbb{D}_{p}^{s}$
such that
$
\Phi_{n_{l}} \to \Phi
$
weakly in $\mathbb{D}_{p}^{s}$.
In particular, since $s \geqslant 0$ and $p \geqslant 2$,
this convergence is still valid in the weak topology on $L_{2}$.
On the other hand,
it is clearly satisfied that
$\int_{0}^{T} \Vert \Lambda (X_{t}) \Vert_{2}^{2} \mathrm{d}t < \infty$,
and hence
$\{ \Lambda (X_{t}) \}_{t \geqslant 0}$
is now a square-integrable
$
(\mathcal{F}_{t})_{t \geqslant 0}
$-adapted
process.
The classical stochastic analysis
proves that
$
\Phi_{n}
\to
\int_{0}^{T} \Lambda (X_{t}) \mathrm{d}w^{i}(t)
$
in $L_{2}$.
Therefore, it must be
$
\int_{0}^{T} \Lambda (X_{t}) \mathrm{d}w^{i}(t)
=
\lim_{l\to\infty} \Phi_{n_{l}}
= \Phi
$, and thus
$
\int_{0}^{T} \Lambda (X_{t}) \mathrm{d}w^{i}(t)
\in
\mathbb{D}_{p}^{s}
$.
\end{proof} 

\subsection{Distributional It\^o's formula}
\label{Generalized_Ito} 

Let
$A_{i}$, $i=1, 2, \cdots , d$
and $L$
be the vector fields and
the second-order differential operator
given by
\begin{equation*}
\begin{split}
\left.\begin{array}{ll}
\displaystyle
( A_{i}f )(z)
:=
\sum_{k=1}^{d}
\sigma_{i}^{k} (z) \frac{\partial f}{\partial z_{k}} (z), \\
\displaystyle
( Lf )(z)
:=
\frac{1}{2}\sum_{i,j=1}^{d}
( \sigma \sigma^{*} )_{i}^{j} (z)
\frac{\partial^{2} f}{\partial z_{i} \partial z_{j}}
(z)
+
\sum_{i=1}^{d}
b^{i} (z)
\frac{\partial f}{\partial z_{i}} (z)
\end{array}\right.
\end{split}
\end{equation*}
for $f \in \mathscr{S}(\mathbb{R}^{d})$ and $z \in \mathbb{R}^{d}$.
In the case of $d=1$, the vector field $A_{1}$
will be denoted by $A$.
Under the assumption {\rm (H1)},
these operators naturally act on
$\mathscr{S}^{\prime}(\mathbb{R}^{d})$
and
$\mathscr{E}^{\prime}(\mathbb{R}^{d})$.

\begin{Thm}[cf. Kubo \cite{Ku1}] 
\label{Ito_Formula} 
Let $x \in \mathbb{R}^{d}$ and assume
{\rm (H1)} and {\rm (H2)}.
Then for each
$\Lambda  \in \mathscr{S}^{\prime} (\mathbb{R}^{d})$
and
$t_{0} \in (0,T]$, we have
\begin{equation}
\label{Ito_Formula0} 
\begin{split}
&
\Lambda (X(T,x,w)) - \Lambda (X(t_{0},x,w)) \\
&=
\sum_{i=1}^{d}
\int_{t_{0}}^{T}
(A_{i}\Lambda ) (X(t,x,w))
\mathrm{d}w^{i}(t)
+
\int_{t_{0}}^{T}
(L\Lambda ) (X(t,x,w))
\mathrm{d}t
\quad
\text{in $\mathbb{D}^{-\infty}$.}
\end{split}
\end{equation}
Similarly, we have
(\ref{Ito_Formula0})
for $\Lambda \in \mathscr{E}^{\prime} (\mathbb{R}^{d})$
if we further assume {\rm (H4)}.
\end{Thm} 
\begin{proof} 
For simplicity of notation, we assume $d=1$.
The case $d \geqslant 2$ is similar.
Fix $t_{0}>0$.
Suppose that $\Lambda \in \mathscr{S}_{-2k}$.
Then there exist
$\phi_{n} \in \mathscr{S}(\mathbb{R})$,
$n \in \mathbb{N}$
such that
$
\Lambda = \lim_{n \to \infty} \phi_{n}
$
in
$\mathscr{S}_{-2k}$.
By It\^o's formula, we clearly have
\begin{equation*}
\begin{split}
&
\phi_{n} (X_{T}) - \phi_{n} (X_{t_{0}})
=
\int_{t_{0}}^{T}
(A \phi_{n} ) (X_{t})
\mathrm{d}w(t)
+
\int_{t_{0}}^{T}
(L \phi_{n} ) (X_{t})
\mathrm{d}t
\end{split}
\end{equation*}
for each $n \in \mathbb{N}$.
What we have to prove is the following:
As $n \to \infty$,
\begin{itemize}
\item[(a)]
$
\Vert
	\Lambda (X_{t})
	-
	\phi_{n} (X_{t})
\Vert_{2,-2k}
\to 0
$
for $t=t_{0}$ and $T$,

\item[(b)]
$\displaystyle
\Vert
	\int_{t_{0}}^{T}
	[
	A \Lambda (X_{t})
	-
	A \phi_{n} (X_{t})
	]
	\mathrm{d}w(t)
\Vert_{
	2,
	-(2k+2)
}
\to 0
$,

\item[(c)]
$\displaystyle
\Vert
	\int_{t_{0}}^{T}
	[
	L \Lambda (X_{t})
	-
	L \phi_{n} (X_{t})
	]
	\mathrm{d}t
\Vert_{
	2,
	-(2k+2)
}
\to 0
$.
\end{itemize}

It is easy to show (a).
In fact,
for any $p \in (1, +\infty )$,
we have the inequality
\begin{equation}
\label{aux-ineq} 
\Vert
	\Lambda ( X_{t} ) - \phi_{n} (X_{t})
\Vert_{
	p,
	-2k
}
\leqslant
c_{0}
t^{-K}
\vert \Lambda - \phi_{n} \vert_{-2k}
\end{equation}
where the
constants $c_{0}>0$ and $K>0$
can depend on
$p$
but not on
$t$ and
$\phi_{n}$'s
(See \cite[Chapter~V, Section~9, Theorem~9.1
and Section~10, Theorem~10.2]{IW}).
We shall prove (b).
By Proposition \ref{ver_pair}
and (\ref{stoc_int}), we have
\begin{equation*}
\begin{split}
&
\Vert
	\int_{t_{0}}^{T}
	[
	A \Lambda (X_{t})
	-
	A \phi_{n} (X_{t})
	]
	\mathrm{d}w(t)
\Vert_{
	2,
	-(2k+2)
}^{2}
\leqslant
C
\int_{t_{0}}^{T}
\Vert
	[ A ( \Lambda - \phi_{n} ) ] (X_{t})
\Vert_{
	2,
	-(2k+1)
}^{2}
\mathrm{d}t .
\end{split}
\end{equation*}
Next we shall show that
there exists $c, \nu > 0$ such that
\begin{equation}
\label{auxi1} 
\begin{split}
&
\Vert
	[ A ( \Lambda - \phi_{n} ) ] (X_{t})
\Vert_{
	2,
	-(2k+1)
}^{2}
\leqslant
c
t^{-\nu}
\Vert
	(\Lambda - \phi_{n}) ( X_{t} )
\Vert_{
	4,
	-2k
}^{2}
\end{split}
\end{equation}
for every $t \in [t_{0},T]$,
and then the above quantities converge to zero
uniformly in $t \in [t_{0},T]$ as $n \to \infty$,
and hence
(b)
is proved.
To prove (\ref{auxi1}),
it suffices to show that:
there exist $c,\nu >0$ such that
\begin{equation}
\label{auxi2} 
\begin{split}
\big\vert
\mathbf{E}
[
	( \sigma \Psi^{\prime} ) ( X_{t} )
	J
]
\big\vert
\leqslant
c
t^{-\nu}
\Vert
	\Psi ( X_{t} )
\Vert_{
	4,
	-2k
}
\Vert
	J
\Vert_{
	2,
	2k+1
}
\end{split}
\end{equation}
for each
$\Psi \in \mathscr{S}_{-2k}$
and
$J \in \mathbb{D}^{\infty}$.
In fact, we have
\begin{equation*}
\begin{split}
&
\mathbf{E}
[
	( \sigma \Psi^{\prime} ) ( X_{t} )
	J
]
=
\mathbf{E}
\big[
	\Psi ( X_{t} )
	\Big\{
		P_{0}(t) \sigma (X_{t}) J
		+
		\langle
			P_{1}(t),
			D \big( \sigma (X_{t}) J \big)
		\rangle_{H}
	\Big\}
\big],
\end{split}
\end{equation*}
for some
$P_{i}(t) \in \mathbb{D}^{\infty} (H^{\otimes i})$,
$i=0,1$
which are polynomials in $X_{t}=X(t,x,w)$,
its derivatives and
$
\Vert
	DX_{t}
\Vert_{H}^{-2}
$.
Hence
\begin{equation*}
\begin{split}
&
\big\vert
\mathbf{E}
[
	( \sigma \Psi^{\prime} ) ( X_{t} )
	J
]
\big\vert
\leqslant
\Vert
	\Psi ( X_{t} )
\Vert_{
	4,
	-2k
}
\Big\{
	\Vert
		P_{0}(t) \sigma (X_{t}) J
	\Vert_{
		4/3,
		2k
	}
	+
	\Vert
	\langle
		P_{1}(t),
		D \big( \sigma (X_{t}) J \big)
	\rangle_{H}
	\Vert_{
		4/3,
		2k
	}
\Big\}.
\end{split}
\end{equation*}
Noting that
$
\frac{1}{2} + \frac{1}{4} = \frac{3}{4} < 1
$,
we can make estimates
\begin{equation*}
\begin{split}
&
\Vert
	P_{0}(t) \sigma (X_{t}) J
\Vert_{
	4/3,
	2k
}
\leqslant
c^{\prime}
t^{-\nu}
\Vert
	J
\Vert_{
	2,
	2k
}, \\
&
\Vert
\langle
	P_{1}(t),
	D \big( \sigma (X_{t}) J \big)
\rangle_{H}
\Vert_{
	4/3,
	2k
}
\leqslant
c^{\prime}
t^{-\nu}
\Vert
	J
\Vert_{
	2,
	2k+1
}
\end{split}
\end{equation*}
for each $t>0$,
and for some constants $c^{\prime}, \nu > 0$
(where, $c^{\prime}$ may depend on the derivatives of
$\sigma$ up to the
($
2k+1
$)-th
order,
which are assumed to be bounded by {\rm (H1)}).
Now (\ref{auxi2}) follows.

(c) is proved similarly.
The statement for $\Lambda \in \mathscr{E}^{\prime} (\mathbb{R})$
is also proved similarly.
\end{proof} 

\begin{Thm} 
\label{Ito_Formula1} 
Let $x \in \mathbb{R}^{d}$.
Suppose {\rm (H1)},
{\rm (H2)} and {\rm (H4)}.
Let $f: \mathbb{R}^{d} \to \mathbb{R}$ be a
locally-integrable
function such that
\begin{itemize}
\item[(i)]
$f$ is continuous at $x$,

\item[(ii)]
$f \in \mathscr{E}^{\prime} (\mathbb{R}^{d})$,

\item[(iii)]
$
\int_{0}^{T}
\Vert
	( A_{i}f ) (X(t,x,w))
\Vert_{2,-k}^{2}
\mathrm{d}t
< +\infty
$
for $i=1, 2, \cdots , d$,

\item[(iv)]
$
\int_{0}^{T}
\Vert
	(Lf) (X(t,x,w))
\Vert_{2,-k}
\mathrm{d}t
< +\infty
$
\end{itemize}
for some $k \in \mathbb{N}$.
Then we have
\begin{equation}
\label{Ito_Formula2} 
\begin{split}
&
f(X(T,x,w)) - f(x) \\
&=
\sum_{i=1}^{d}
\int_{0}^{T}
(A_{i}f) (X(t,x,w))
\mathrm{d}
w^{i}
(t)
+
\int_{0}^{T}
(Lf) (X(t,x,w))
\mathrm{d}t
\quad
\text{in $\mathbb{D}^{-\infty}$.}
\end{split}
\end{equation}
The equation
(\ref{Ito_Formula2})
still holds without
{\rm (H4)}
if
$f \in \mathscr{S}^{\prime} (\mathbb{R}^{d})$.
\end{Thm} 
\begin{proof} 
By the conditions (ii), (iii), (iv) and Theorem \ref{Ito_Formula},
we have for each $t_{0}>0$ that
\begin{equation*}
\begin{split}
&
f(X_{T}) - f(X_{t_{0}})
=
\sum_{i=1}^{d}
\int_{t_{0}}^{T}
(A_{i}f) (X_{t})
\mathrm{d}
w^{i}
(t)
+
\int_{t_{0}}^{T}
(Lf) (X_{t})
\mathrm{d}t
\end{split}
\end{equation*}
in $\mathbb{D}^{-\infty}$.
Letting
$t_{0} \downarrow 0$
with using
(i),
we have
$
f (X_{t_{0}})
\to
f(x)
$
in probability.
Moreover, again by (i),
we can take $\delta > 0$
such that for any $y \in \mathbb{R}^{d}$,
$\vert y - x \vert < \delta$
implies
$\vert f(y) \vert < \vert f(x) \vert + 1$.
Thus
$
\mathbf{E} [ f( X_{t} )^{4} ]
\leqslant
( \vert f(x) \vert + 1 )^{4}
+
\mathbf{E}
[
	f( X_{t} )^{4} ;
	\vert X_{t} - x \vert > \delta
]
$,
where
$
\limsup_{t \downarrow 0}
\mathbf{E}
[
	f( X_{t} )^{4} ;
	\vert X_{t} - x \vert > \delta
] < +\infty
$
in view of (ii) and Lemma~\ref{IbP}--(ii).
Therefore $\{ f(X_{t})^{2} \}_{0 < t \leqslant T}$
is $L_{2}$-bounded, so that
$\{ ( f(X_{t}) - f(x) )^{2} \}_{0 < t \leqslant T}$
is uniformly integrable.
Hence $f(X_{t_{0}}) \to f(x)$ in $L_{2}$ as $t_{0} \downarrow 0$
and
(\ref{Ito_Formula2}) is proved.

If $f \in \mathscr{S}^{\prime}(\mathbb{R}^{d})$,
then $f$ has at most polynomial growth.
In this case,
$\{ f(X_{t})^{2} \}_{0 < t \leqslant T}$
is $L_{2}$-bounded without assuming {\rm (H4)}
and
then (\ref{Ito_Formula2}) can be proved similarly.
\end{proof} 

By Proposition \ref{exclusion_int},
Lemma \ref{exclusion}
and
Theorem \ref{Ito_Formula},
we obtain
the following.

\begin{Cor} 
Let
$x \in \mathbb{R}^{d}$
and
$\Lambda \in \mathscr{E}^{\prime} (\mathbb{R}^{d})$
be such that
$\mathrm{supp} \Lambda \not\hspace{0.4mm}\ni x$.
Assume {\rm (H1)}, {\rm (H2)}
and {\rm (H4)}.
Then we have in $\mathbb{D}^{-\infty}$,
\begin{equation*}
\begin{split}
&
\Lambda (X(T,x,w))
=
\sum_{i=1}^{d}
\int_{0}^{T}
( A_{i} \Lambda ) (X(t,x,w))
\mathrm{d}
w^{i}
(t)
+
\int_{0}^{T}
( L \Lambda ) (X(t,x,w))
\mathrm{d}t.
\end{split}
\end{equation*}
This still holds without
{\rm (H4)}
if
$f \in \mathscr{S}^{\prime} (\mathbb{R}^{d})$.
\end{Cor} 

\subsection{An application and examples}
\label{application} 

In the sequel, we denote by
$X = (X_{t})_{t \geqslant 0}$
the
unique strong solution to
the $d$-dimensional stochastic differential equation
\begin{equation}
\label{ref-SDE} 
\mathrm{d}X_{t}
=
\sigma ( X_{t} ) \mathrm{d} w(t)
+
b(X_{t}) \mathrm{d}t,
\quad
X_{0} = x \in \mathbb{R}^{d},
\end{equation}
where
$w=(w(t))_{t \geqslant 0}$
is a
$d$-dimensional Wiener process.
We denote by $L$ the associated generator,
i.e.,
$
L
=
\frac{1}{2}
\sum_{i,j=1}^{d}
( \sigma \sigma^{*} )_{j}^{i}
\partial_{i}\partial_{j}
+
\sum_{i=1}^{d}
b^{i}
\partial_{i}
$.

Let
$
\{ F_{ \varepsilon } \}_{\varepsilon \in I}
\subset \mathbb{D}^{\infty}(\mathbb{R}^{d})
$
be a bounded and uniformly non-degenerate family
(see Definition \ref{non-degeneracy}).
Here, the boundedness is used in the sense of
$
\{ F_{\varepsilon} \}_{\varepsilon \in I}
$
is bounded in $\mathbb{D}_{p}^{k}(\mathbb{R}^{d})$
for each $p \in (1, \infty)$ and $k \in \mathbb{Z}_{\geqslant 0}$.

\begin{Lem} 
\label{frac_dom} 
For any $s \in \mathbb{R}$,
$p \in ( 1, \infty )$
and $p^{\prime} > p$,
there exists $c = c(s,p,p^{\prime}) > 0$ such that
$$
\Vert
	\Lambda ( F_{\varepsilon} )
\Vert_{p,s}
\leqslant
c
\Vert
	( 1 - \triangle )^{s/2} \Lambda
\Vert_{L_{p^{\prime}}(\mathbb{R}^{d}, \mathrm{d}x)}
$$
for every $\varepsilon \in I$
and
$
\Lambda \in \mathscr{S}^{\prime}(\mathbb{R}^{d})
$
with
$
( 1 - \triangle )^{s/2} \Lambda
\in
L_{p^{\prime}}(\mathbb{R}^{d}, \mathrm{d}x)
$.
\end{Lem} 

We note that
one can take $p^{\prime} = p$
when $F_{\varepsilon} = \varepsilon^{-1} w(\varepsilon^{2}T)$
as mentioned in Remark \ref{p'=p}.
We don't give a proof of Lemma~\ref{frac_dom}.
However, we prove a similar inequality
(Lemma~\ref{fractional_domin})
in Section~\ref{frac_ineq}.
The same techniques there are available
to show Lemma~\ref{frac_dom}
if we replace the harmonic oscillator $H$
by the Bessel potential $(1-\triangle )$.
(Although the step (b)
in the proof of Lemma~\ref{fractional_domin}
uses results by Bongioanni-Torrea~\cite{BoTo},
we don't need for the proof of Lemma~\ref{frac_dom}.
The steps (a) and (b) for the proof of Lemma~\ref{frac_dom}
can be established analogously
by a standard integration-by-parts techniques in
the Malliavin calculus.
The steps (c)--(e) also run in exactly the same way.)

Let
$
H_{p}^{s}(\mathbb{R}^{d})
:=
( 1 - \triangle )^{-s/2}
L_{p} (\mathbb{R}^{d}, \mathrm{d}z)
$,
$p \in (1,\infty )$,
$s \in \mathbb{R}$
be the Bessel potential spaces
(see \cite{Ab} and \cite{Kr} for details).
By
imitating the proof of Corollary~\ref{exp-pull-conti}
with using Lemma~\ref{frac_dom}
(instead of Proposition~\ref{exp-pull}),
we can show that for each
$p \in (1, \infty )$,
$s \in \mathbb{R}$
and
$\Lambda \in H_{p}^{s} ( \mathbb{R}^{d} )$,
the mapping
$
(0, \infty ) \ni t
\mapsto
\Lambda ( X_{t} ) \in \mathbb{D}_{p^{\prime}}^{s}
$
is continuous for $p^{\prime} \in (1,p)$
under assumptions
{\rm (H1)} and {\rm (H2)}.

We are now in a position to prove Corollary \ref{Main_Thm5}.

\begin{proof}[Proof of Corollary \ref{Main_Thm5}]
(i) is clear by Lemma \ref{frac_dom}.

(ii)
The operator $L$ is a uniformly elliptic operator
of the second order satisfying {\rm (H4)}.
Hence the elliptic regularity theorem
(see e.g.,
\cite[Chapter III, Section 7.3, Theorem 7.13]{Ab})
assures that
$f := (1-L)^{-1} \Lambda \in H_{p}^{s+2} ( \mathbb{R}^{d} )$
and
$\sigma_{j}^{k} \partial_{k} f \in H_{p}^{s+1} ( \mathbb{R}^{d} )$
for each $j,k=1, 2, \cdots , d$.
On the other hand,
Theorem \ref{Ito_Formula} gives
\begin{equation*}
f(X_{T}) - f(X_{t_{0}})
=
\sum_{j,k=1}^{d}
\int_{t_{0}}^{T}
( \sigma_{j}^{k} \partial_{k} f ) ( X_{t} )
\mathrm{d} w^{j}(t)
+
\int_{t_{0}}^{T}
( Lf ) (X_{t})
\mathrm{d}t.
\end{equation*}

Let $p^{\prime} \in [2,p)$ be arbitrary.
We find that
$f(X_{T}) \in \mathbb{D}_{p^{\prime}}^{s+2}$
by (i)
and
$
\int_{t_{0}}^{T}
( \sigma_{j}^{k} \partial_{k} f )
( X_{t} )
\mathrm{d}w^{j}
\in \mathbb{D}_{p^{\prime}}^{s+1}
$
for $T > 0$
by Theorem \ref{reg-pres},
so that we have
$
\int_{t_{0}}^{T}
( Lf ) (X_{t})
\mathrm{d}t
\in \mathbb{D}_{p^{\prime}}^{s+1}
$.
Furthermore, since
$
[t_{0}, T] \ni t
\mapsto
f (X_{t}) \in \mathbb{D}_{p^{\prime}}^{s+2}
$
is continuous
(recall the remark just before the proof),
this mapping is Bochner integrable and
$
\int_{t_{0}}^{T} f(X_{t}) \mathrm{d}t
\in \mathbb{D}_{p^{\prime}}^{s+2}
$.
Therefore
\begin{equation*}
\begin{split}
\int_{t_{0}}^{T}
\Lambda (X_{t})
\mathrm{d}t
=
\int_{t_{0}}^{T}
f (X_{t})
\mathrm{d}t
-
\int_{t_{0}}^{T}
( Lf ) (X_{t})
\mathrm{d}t
\in \mathbb{D}_{p^{\prime}}^{s+1}.
\end{split}
\end{equation*}
\end{proof} 

\begin{Rm} 
By the above remark
(just after Lemma \ref{frac_dom}),
one would see that
in Corollary \ref{Main_Thm5},
we can take $p^{\prime} = p$
when $\sigma = (\text{identity matrix})$ and $b=0$,
i.e.,
$X_{t} =  w(t)$.
\end{Rm} 

Second,
we investigate
the class $\mathbb{D}_{p}^{s}$ to which
$
\int_{t_{0}}^{T} \Lambda (X_{t}) \mathrm{d}t
$
belongs when $t_{0}=0$ for several cases of
$\Lambda \in \mathscr{S}^{\prime} (\mathbb{R}^{d})$.

\begin{Eg} 
\label{Eg.1} 
Assume $d=1$, {\rm (H1)},
{\rm (H3)} and {\rm (H4)}.
Let
\begin{equation*}
\begin{split}
\left.\begin{array}{l}
\displaystyle
s(x)
:=
\int_{0}^{x}
\exp
\Big\{
	- \int_{0}^{z}
	\frac{ 2b(\eta ) }{ \sigma (\eta )^{2} }
	\mathrm{d}\eta
\Big\}
\mathrm{d}z, \\
\displaystyle
m(x)
:=
2
\int_{0}^{x}
\exp
\Big\{
	\int_{0}^{z}
	\frac{ 2b(\eta ) }{ \sigma (\eta )^{2} }
	\mathrm{d}\eta
\Big\}
\frac{ \mathrm{d}z }{ \sigma (z)^{2} },
\end{array}\right.
\quad
\text{for $x \in \mathbb{R}$.}
\end{split}
\end{equation*}
The function $s(x)$ is called the
{\it scale function} of $L$
and the measure
$m (\mathrm{d}x ) = m^{\prime} (x) \mathrm{d}x$
is called the
{\it speed measure}
of $L$.
We fix $y \in \mathbb{R}$ and define
$u: \mathbb{R} \to \mathbb{R}$
by
\begin{equation}
\label{fund_sol} 
u(x)
:=:
u(x,y)
:=
\frac{ m^{\prime} (y) }{2}
\vert
	s(x) - s(y)
\vert ,
\quad
x \in \mathbb{R}.
\end{equation}
Then it is easily checked that
$
L u = \delta_{y}
$
and
$$
( Au )(x)
=
\mathrm{sgn} ( x-y )
\big[
\exp
\Big\{
	- \int_{y}^{x}
	\frac{ 2 b( \eta )  }{ \sigma ( \eta )^{2} }
	\mathrm{d}\eta
\Big\}
\frac{ \sigma (x) }{ \sigma (y)^{2} }
\big],
\quad
x \in \mathbb{R}.
$$
in the distributional sense.
Now we shall prove
\begin{equation}
\label{Eg.4-1} 
\left\{\begin{array}{l}
\text{$s < \frac{1}{2}$ ;} \\
\text{$p \in ( 1,\frac{1}{s} )$}
\end{array}\right.
\Rightarrow
\left\{\begin{array}{l}
1_{ \{ X(t,x,w) < y \} }
\in \mathbb{D}_{p}^{s}\text{ ;} \\
\limsup_{t \downarrow 0}
\Vert
	1_{ \{ X(t,x,w) < y \} }
\Vert_{p,s}
< \infty .
\end{array}\right.
\end{equation}

In fact, let $s \in ( 0, 1/2 )$ and $p \in (1, \frac{1}{s})$.
Take $p^{\prime} > p$ so that $sp^{\prime} < 1$.
Then by
Proposition \ref{trick1}
and
Lemma \ref{fractional_domin},
there exists $c>0$ such that
\begin{equation*}
\begin{split}
\Vert
	1_{ \{ X(t,x,w) < y \} }
\Vert_{p,s}
&=
\Vert
	1_{ (-\infty , (y-x)/\sqrt{t}) }
	( \widetilde{F}_{t} )
\Vert_{p,s} \\
&\leqslant
c
\Vert
	( z^{2} - \triangle )^{s/2}
	1_{ (-\infty , (y-x)/\sqrt{t}) }
\Vert_{L_{p^{\prime}} (\mathbb{R}, \mathrm{d}z)}
\quad
\text{for all $t \in (0,T]$.}
\end{split}
\end{equation*}
where
$
\widetilde{F}_{t}
=
( X^{\sqrt{t}}(1,x,w) - x ) / \sqrt{t}
$.
Hence we can conclude (\ref{Eg.4-1}) by
Proposition~\ref{Heaviside_class}--(iii).

By the condition
{\rm (H3)} and {\rm (H4)},
we see that
$u \in \mathscr{E}^{\prime} (\mathbb{R})$.
Hence by
Theorem \ref{pos_Bochner}
and
Theorem \ref{Ito_Formula1},
we have
\begin{equation*}
\begin{split}
&
\int_{0}^{T}
\delta_{y} (X_{t})
\mathrm{d}t
=
u ( X_{T} ) - u (x)
-
\int_{0}^{T}
( Au ) ( X_{t} )
\mathrm{d}w(t) .
\end{split}
\end{equation*}
It is easy to see that
$
u ( X_{T} ) \in \cap_{p>1} \mathbb{D}_{p}^{1}
$
and by {\rm (H3)} that
$$
\exp
\Big\{
	- \int_{y}^{X_{t}}
	\frac{2b(\eta )}{\sigma (\eta )^{2}}
	\mathrm{d}\eta
\Big\}
\frac{\sigma (X_{t})}{\sigma (y)^{2}}
\in \mathbb{D}_{p}^{\infty}
\quad
\text{for all $p \in (1,\infty )$}.
$$
From this and (\ref{Eg.4-1}),
we find
$
\limsup_{t \downarrow 0}
\Vert
	( Au )(X_{t})
\Vert_{p,s}
< \infty
$
for $s < 1/2$ and $p \in (1, \frac{1}{s})$.
Then by Theorem \ref{reg-pres}--(ii),
$
\int_{0}^{T} ( Au ) (X_{t}) \mathrm{d}w(t)
\in \mathbb{D}_{p}^{s}
$
for $s < \frac{1}{2}$ and $p \in [ 2, \frac{1}{s} )$,
and thus we reached:
under
{\rm (H1)}, {\rm (H3)} and {\rm (H4)},
$$
\int_{0}^{T} \delta_{y} (X_{t}) \mathrm{d}t
\in \mathbb{D}_{p}^{s}
\quad
\text{for $s < \frac{1}{2}$ and $p \in \big[ 2, \frac{1}{s} \big)$.}
$$
See Airault-Ren-Zhang \cite{ARZ} for a more general and stronger result.

\end{Eg} 

\begin{Rm} 
\label{sym-local_time} 
In particular, we see from Example \ref{Eg.1} that
$
\int_{0}^{T}
\delta_{y} (X_{t})
\mathrm{d}t
$
is a classical Wiener functional when $d=1$,
which is related to the local time at $y$.
This can be seen as follows:
Classically, the {\it symmetric local time}
$
\{ \widetilde{l} (y,t) : y \in \mathbb{R}, t \geqslant 0 \}
$
is defined as a unique increasing process
such that
$$
\vert
	X_{t} - y
\vert
=
\vert
	x - y
\vert
+
\int_{0}^{t}
\mathrm{sgn} ( X_{s} - y )
\mathrm{d}X_{s}
+
\widetilde{l} (y,t)
$$
and equivalently given by
$$
\widetilde{l} (y,t)
=
\lim_{\varepsilon \downarrow 0}
\frac{1}{2 \varepsilon}
\int_{0}^{t}
1_{
	( y - \varepsilon , y + \varepsilon )
}
(X_{s})
\mathrm{d} \langle X \rangle_{s}
$$
(see \cite{ReYo}).
Since it holds
$
\lim_{\varepsilon \downarrow 0}
( 2\varepsilon )^{-1}
\sigma^{2}
1_{
	( y - \varepsilon , y + \varepsilon )
}
=
\sigma \delta_{y}
=
\sigma (y) \delta_{y}
$
in $\mathscr{S}^{\prime} (\mathbb{R})$,
we have
$$
\widetilde{l} (y,t)
=
\sigma (y)^{2}
\int_{0}^{t} \delta_{y} (X_{u}) \mathrm{d}u .
$$
\end{Rm} 

In the sequel,
we denote
$
\mathbb{D}_{2}^{s-}
:=
\cap_{\varepsilon > 0}
\mathbb{D}_{2}^{s-\varepsilon}
$.

\begin{Eg} 
\label{delta_prime} 
Assume $d=1$, {\rm (H1)},
{\rm (H3)} and {\rm (H4)}.
Let $x, y \in \mathbb{R}$ be such that $x \neq y$.
By
using Lemma~\ref{frac_dom},
we see that
the mapping
$
(0,T] \ni t
\mapsto
\delta_{y}^{\prime} ( X(t,x,w) ) \in \mathbb{D}_{2}^{-k}
$
is
continuous
for some $k \in \mathbb{Z}_{\geqslant 0}$.

Define $u(z_{1},z_{2})$ by (\ref{fund_sol}), and then
\begin{equation*}
\begin{split}
&
v(z)
:= - (\partial_{z_{2}} u) (z,y)
=
\frac{1}{2}
\Big\{
	m^{\prime} (y)
	s^{\prime} (y)
	\mathrm{sgn} ( z - y )
	-
	m^{\prime\prime}(y)
	\vert s(z) - s(y) \vert
\Big\}
\in \mathscr{E}^{\prime} (\mathbb{R})
\end{split}
\end{equation*}
satisfies
$
Lv = \delta_{y}^{\prime}
$
and
\begin{equation*}
\begin{split}
( Av )(z)
=
\frac{ \sigma (z) }{2}
\Big(
	2 m^{\prime} (y) s^{\prime} (y)
	\delta_{y} (z)
	-
	m^{\prime\prime} (y)
	s^{\prime} (z)
	\mathrm{sgn} ( z-y )
\Big) .
\end{split}
\end{equation*}
Since $X_{t} = X(t,x,w)$ is non-degenerate for each $t>0$,
it is known by Watanabe \cite{Wa91}
that
$
\delta_{y} ( X_{t} ) \in \mathbb{D}_{p}^{s}
$
if $s \in (-1, -\frac{1}{2} )$ and $p \in (1, \frac{1}{1+s})$.
Hence by virtue of conditions
{\rm (H1)} and {\rm (H3)},
for each $t_{0}\in (0,T)$ and $p\geqslant 2$,
we find
$
\int_{t_{0}}^{T}
( Au )( X_{t} )
\mathrm{d}w(t)
\in \mathbb{D}_{p}^{(\frac{1}{p}-1)-}
$,
and thus
$$
\int_{t_{0}}^{T}
\delta_{y}^{\prime} ( X_{t} )
\mathrm{d}t
\in \mathbb{D}_{p}^{(\frac{1}{p} - 1)-}
\quad
\text{for any $t_{0}>0$ and $p \geqslant 2$.}
$$
\end{Eg} 

\begin{Eg} 
In the above Example \ref{delta_prime},
we shall try to investigate
the class to which
$
\int_{0}^{T}
\delta_{y}^{\prime} ( X_{t} )
\mathrm{d}t
$
belongs.
Let $p \geqslant 2$ be arbitrary.
Since
$
\delta_{y} ( x + \varepsilon \bullet )
=
\lim_{a \downarrow 0}
( 2\pi a )^{-1/2}
\exp
\{
	- ( x + \varepsilon \bullet - y )^{2} / (2a)
\}
=
\varepsilon^{-1}
\lim_{a \downarrow 0}
( 2\pi \varepsilon^{-2} a )^{-1/2}
\exp
\{
	- ( \frac{ y - x }{ \varepsilon } - \bullet )^{2} / (2 \varepsilon^{-2} a)
\}
=
\varepsilon^{-1}
\delta_{ (y-x)/\varepsilon } ( \bullet )
$
in the distributional sense,
we have
\begin{equation*}
\begin{split}
\sigma (X_{t})
( m^{\prime} s^{\prime} )(y)
\delta_{y} (X_{t})
&=
( \sigma m^{\prime} s^{\prime} )(y)
\varepsilon^{-1}
\delta_{(y-x)/\varepsilon} (\widetilde{F}_{\varepsilon}),
\end{split}
\end{equation*}
where $\varepsilon := \sqrt{t}$,
$X_{t} = X(t,x,w)$
and
$
\widetilde{F}_{\varepsilon}
=
(X_{t} - x)/ \varepsilon
$.
We shall prove
\begin{equation}
\label{Eg.3-a} 
\varepsilon^{-1}
\Vert
	\delta_{(y-x)/\varepsilon} (\widetilde{F}_{\varepsilon})
\Vert_{p,s}
\to 0
\quad
\text{as $\varepsilon = \sqrt{t} \downarrow 0$ for every $s < -1$.}
\end{equation}
Note that the probability density function
$p_{\widetilde{F}_{\varepsilon}}$
of
$\widetilde{F}_{\varepsilon}$
has an estimate
$
\sup_{0 < \varepsilon \leqslant 1}
p_{\widetilde{F}_{\varepsilon}} (z)
\leqslant
C
p(z)
$
for some $C>0$,
where
$
p(z)
:=
( 2 \pi c^{\prime} )^{-1/2} \mathrm{e}^{ - z^{2} / (2c^{\prime}) }
$
and $c^{\prime} > 0$ is a constant.
By refining the argument
in the proof of Lemma \ref{fractional_domin},
one finds that for every $p^{\prime} > p$,
$$
\Vert
	\delta_{(y-x)/\varepsilon} ( \widetilde{F}_{\varepsilon} )
\Vert_{p,s}
\leqslant
\mathrm{const.}
\Vert
	( 1 - \triangle )^{s/2}
	\delta_{(y-x)/\varepsilon}
\Vert_{L_{p^{\prime}}(\mathbb{R},\mu)}
$$
for all $\varepsilon \in (0,1]$,
where $\mu (\mathrm{d}z) = p(z) \mathrm{d}z$.
To estimate the last quantity,
we apply the Fourier transformation
and integration by parts
to get
\begin{equation*}
\begin{split}
\varepsilon^{-2}
( 1 - \triangle )^{s/2} \delta_{(y-x)/\varepsilon}
(z)
&=
\frac{-1}{2\pi}
\int_{-\infty}^{\infty}
\frac{
	\mathrm{e}^{i\xi z}
}{
	( 1 + \xi^{2} )^{-s/2}
}
\varepsilon^{-2}
\mathrm{e}^{i\xi \frac{y-x}{\varepsilon}}
\mathrm{d}\xi \\
&=
\frac{1}{2\pi (y-x)^{2}}
\int_{-\infty}^{\infty}
\Big\{
\frac{\mathrm{d}^{2}}{\mathrm{d}\xi^{2}}
\frac{
	\mathrm{e}^{i\xi z}
}{
	( 1 + \xi^{2} )^{-s/2}
}
\Big\}
\mathrm{e}^{i\xi \frac{y-x}{\varepsilon}}
\mathrm{d}\xi ,
\end{split}
\end{equation*}
in which, one finds that there exists $c=c(s)>0$ such that
$$
\big\vert
\frac{\mathrm{d}^{2}}{\mathrm{d}\xi^{2}}
\frac{
	\mathrm{e}^{i\xi z}
}{
	( 1 + \xi^{2} )^{-s/2}
}
\big\vert
\leqslant
c
( 1 + z^{2} )
( 1 + \xi^{2} )^{s/2}.
$$
Hence we obtain
\begin{equation*}
\begin{split}
&
\vert
\varepsilon^{-2}
( 1 - \triangle )^{s/2} \delta_{(y-x)/\varepsilon}
(z)
\vert
\leqslant
\frac{c}{2\pi (y-x)^{2}}
( 1 + z^{2} )
\int_{-\infty}^{\infty}
( 1 + \xi^{2} )^{s/2}
\mathrm{d}\xi .
\end{split}
\end{equation*}
Note that
$
\int_{-\infty}^{\infty}
( 1 + \xi^{2} )^{s/2}
\mathrm{d}\xi
< +\infty
$
since
$s<-1$.
We thus have
\begin{equation*}
\begin{split}
&
\sup_{\varepsilon > 0}
\varepsilon^{-2p^{\prime}}
\Vert
( 1 - \triangle )^{s/2} \delta_{(y-x)/\varepsilon}
\Vert_{ L_{p^{\prime}}( \mathbb{R}, \mu ) }^{p^{\prime}}
\leqslant
c^{\prime}
\int_{\mathbb{R}}
( 1 + z^{2} )^{p^{\prime}}
p(z)
\mathrm{d}z
< +\infty ,
\end{split}
\end{equation*}
where the constant
$c^{\prime} > 0$
may depend on $y-x$, and
which gives (\ref{Eg.3-a}).

By {\rm (H1)}, {\rm (H3)} and {\rm (H4)},
we see that
$
\lim_{t \downarrow 0}
\sigma ( X_{t} )
=
\sigma (x)
$
in $\mathbb{D}^{\infty}$.
Combining this with (\ref{Eg.3-a}) and
Theorem~\ref{reg-pres}--(i),
we obtain
$
\int_{0}^{T}
( \sigma \delta_{y} ) ( X_{t} )
\mathrm{d}w(t)
\in \mathbb{D}_{p}^{(-1)-}
$.
On the other hand, we have
$
\int_{0}^{T}
( \sigma  s^{\prime} ) ( X_{t} )
\mathrm{sgn} ( X_{t} - y )
\mathrm{d}w(t)
\in L_{2} \subset \mathbb{D}_{p}^{(-1)-}
$.
Therefore
\begin{equation*}
\begin{split}
&
\int_{0}^{T}
( Au )( X_{t} )
\mathrm{d}w(t) \\
&=
m^{\prime} (y) s^{\prime} (y)
\int_{0}^{T}
( \sigma \delta_{y} )( X_{t} )
\mathrm{d}w(t)
-
\frac{ m^{\prime\prime} (y) }{2}
\int_{0}^{T}
( \sigma s^{\prime} ) (X_{t}) \mathrm{sgn} ( X_{t} - y )
\mathrm{d}w(t)
\in \mathbb{D}_{p}^{(-1)-}.
\end{split}
\end{equation*}
Hence
by Theorem \ref{Ito_Formula1},
$$
\int_{0}^{T}
\delta_{y}^{\prime} ( X_{t} )
\mathrm{d}t
\in \mathbb{D}_{p}^{(-1)-}
\quad
\text{for any $p \geqslant 2$.}
$$
\end{Eg} 

\begin{Eg}[Cauchy's principal value of $1/x$]
\label{PV1/x} 
Let $d=1$.
The tempered distribution
$\mathrm{p.v.}\frac{1}{x}$
is defined by
$$
\langle
	\mathrm{p.v.}\frac{1}{x},
	f
\rangle
:=
\lim_{\varepsilon \downarrow 0}
\int_{\vert x \vert > \varepsilon}
\frac{f(x)}{x}
\mathrm{d}x,
\quad
\text{$f \in \mathscr{S}(\mathbb{R})$.}
$$
Let
$\Lambda \in \mathscr{S}^{\prime}(\mathbb{R})$
be the regular distribution
(i.e., a distribution associated with a locally integrable function)
given by
$
\Lambda = x \log \vert x \vert - x
$.
Then the distributional derivatives are:
$\Lambda^{\prime} = \log \vert x \vert$
and
$\Lambda^{\prime\prime} = \mathrm{p.v.} \frac{1}{x}$.

We shall first show the Bochner integrability
of the pull-back of
$\big( \mathrm{p.v.} \frac{1}{x} \big)$
by a Brownian motion $w(t)$:
$
(0,T] \ni t \mapsto \big( \mathrm{p.v.} \frac{1}{x} \big) (w(t))
$.
For each $J \in \mathbb{D}^{\infty}$
and
$p,q > 1$ such that $1/p + 1/q = 1$,
we obtain
\begin{equation*}
\begin{split}
&
\mathbf{E}
[
	\big( \mathrm{p.v.} \frac{1}{x} \big) (w(t))
	J
]
=
\mathbf{E}
[
	\big( x \log \vert x \vert - x \big)^{\prime\prime} (w(t))
	J
] \\
&=
\mathbf{E}
[
	\big( w(t) \log \vert w(t) \vert - w(t) \big)
	l_{t}(J)
]
\leqslant
\Vert \big( w(t) \log \vert w(t) \vert - w(t) \big) \Vert_{p}
\Vert l_{t}(J) \Vert_{q}.
\end{split}
\end{equation*}
Here
$
l_{t}(J)
=
\sum_{i=0}^{2}
\langle
	P_{i}(t), D^{i} J
\rangle_{H^{\otimes i}}
\in \mathbb{D}^{\infty}
$
for some
$P_{i}(t) \in \mathbb{D}^{\infty} ( H^{\otimes i} )$,
$i=0,1,2$
which are polynomials in $w(t)$,
its derivatives and
$\Vert Dw(t) \Vert_{H}^{-2} = t^{-1}$.
Since we have
\begin{equation*}
\begin{split}
\Vert w(t) \log \vert w(t) \vert \Vert_{p}^{p}
&=
\int_{\mathbb{R}}
\big\vert
	( \sqrt{t} x )
	\log
	(
		\sqrt{t}
		\vert
		x
		\vert
	)
\big\vert^{p}
\frac{
	\mathrm{e}^{ - x^{2} / 2 }
}{
	\sqrt{2\pi}
}
\mathrm{d}x \\
&\leqslant
2^{p}
\vert \sqrt{t} \log \sqrt{t} \vert^{p}
\int_{\mathbb{R}}
	\vert
	x
	\vert^{p}
\frac{
	\mathrm{e}^{ - x^{2} / 2 }
}{
	\sqrt{2\pi}
}
\mathrm{d}x
+
2^{p}
t^{p/2}
\int_{\mathbb{R}}
\big\vert
	x
	\log
	\vert
	x
	\vert
\big\vert^{p}
\frac{
	\mathrm{e}^{ - x^{2} / 2 }
}{
	\sqrt{2\pi}
}
\mathrm{d}x,
\end{split}
\end{equation*}
which tends to zero as $t \downarrow 0$,
and hence
$
\limsup_{t \downarrow 0}
\Vert \big( \mathrm{p.v.} \frac{1}{x} \big) (w(t)) \Vert_{p,-2}
< +\infty
$
for $p\geqslant 2$.
This proves that
$
(0,T] \ni t
\mapsto
\big( \mathrm{p.v.} \frac{1}{x} \big) (w(t)) \in \mathbb{D}_{2}^{-2}
$
is Bochner integrable.
On the other hand, it is clear that
$
\int_{0}^{T}
\Vert \log \vert w(t) \vert \Vert_{2}^{2}
\mathrm{d}t
< +\infty
$.

Now, by Theorem \ref{Ito_Formula1}, we have
\begin{equation}
\label{log-Ito} 
w(T) \log \vert w(T) \vert
=
\int_{0}^{T} \log \vert w(t) \vert \mathrm{d}w(t)
+
\frac{1}{2}
\int_{0}^{T}
\Big( \mathrm{p.v.} \frac{1}{x} \Big) (w(t))
\mathrm{d}t .
\end{equation}
The
chaos expansion of
$\log \vert w(1) \vert$
is
given by
\begin{equation*}
\begin{split}
\log \vert w(1) \vert
&=
\mathbf{E} [ \log \vert w(1) \vert ]
+
\sum_{n=1}^{\infty}
\frac{1}{n!}
\lim_{\varepsilon \downarrow 0}
\int_{ \vert x \vert \geqslant \varepsilon }
\frac{1}{x} H_{n-1} (x)
\frac{
	\mathrm{e}^{-x^{2}/2}
}{
	\sqrt{2\pi}
}
\mathrm{d}x
H_{n} (w(1)) .
\end{split}
\end{equation*}
Put
$
\mu (\mathrm{d}x)
:=
(2\pi )^{-1/2}
\exp ( - x^{2}/2 )
\mathrm{d}x
$.
It is clear that
$
\lim_{\varepsilon \downarrow 0}
\int_{ \vert x \vert > \varepsilon }
\frac{1}{x}
H_{1}(x)
\mu ( \mathrm{d} x )
=
\int_{ \mathbb{R} }
\mu ( \mathrm{d} x )
= 1
$.
For $n \geqslant 2$,
by using Lemma \ref{Hermite-facts}--(ii),
we have
\begin{equation*}
\begin{split}
\lim_{\varepsilon \downarrow 0}
\int_{ \vert x \vert > \varepsilon }
\frac{1}{x}
H_{n}(x)
\mu ( \mathrm{d} x )
&= 
\lim_{\varepsilon \downarrow 0}
\int_{ \vert x \vert > \varepsilon }
\frac{1}{x}
\{ x H_{n-1} (x) - (n-1) H_{n-2} (x) \}
\mu ( \mathrm{d} x ) \\
&= 
\int_{ \mathbb{R} }
H_{n-1} (x)
\mu ( \mathrm{d} x )
-
(n-1)
\lim_{\varepsilon \downarrow 0}
\int_{ \vert x \vert > \varepsilon }
\frac{1}{x} H_{n-2} (x)
\mu ( \mathrm{d} x ) \\
&= 
-
(n-1)
\lim_{\varepsilon \downarrow 0}
\int_{ \vert x \vert > \varepsilon }
\frac{1}{x} H_{n-2} (x)
\mu ( \mathrm{d} x ).
\end{split}
\end{equation*}
Therefore if $n=2k+1$, then
\begin{equation*}
\begin{split}
&
\lim_{\varepsilon \downarrow 0}
\int_{ \vert x \vert > \varepsilon }
\frac{1}{x}
H_{n}(x)
\mu ( \mathrm{d} x ) \\
&=
(-2k)
(-(2k-2))
(-(2k-4))
\cdots
(-2)
\cdot
\lim_{\varepsilon \downarrow 0}
\int_{ \vert x \vert > \varepsilon }
\frac{1}{x}
H_{1}(x)
\mu ( \mathrm{d} x ) \\
&=
(-1)^{k}
(2k)!! .
\end{split}
\end{equation*}
Similarly, if $n$ is even,
$
\lim_{\varepsilon \downarrow 0}
\int_{ \vert x \vert > \varepsilon }
x^{-1}
H_{n}(x)
\mu ( \mathrm{d} x )
= 0
$.
Thus we get
\begin{equation*}
\begin{split}
\log \vert w(1) \vert
&=
\mathbf{E} [ \log \vert w(1) \vert ]
+
\sum_{n=0}^{\infty}
(-1)^{n}
\frac{( 2n )!!}{(2n+2)!}
H_{2n+2} (w(1)),
\end{split}
\end{equation*}
from which, we find that
\begin{equation*}
\begin{split}
&
\Vert \log \vert w(t) / \sqrt{t} \vert \Vert_{2,s}^{2}
=
\Vert \log \vert w(1) \vert \Vert_{2,s}^{2}
=
\mathbf{E} [ \log \vert w(1) \vert ]^{2}
+
\sum_{n=0}^{\infty}
(1+n)^{s}
\frac{ ( ( 2n )!! )^{2} }{(2n+2)!}.
\end{split}
\end{equation*}
Noting $(2n)!! = 2^{n} n!$,
the Stirling formula tells us
\begin{itemize}
\item[$\circ$]
$
\Vert \log \vert w(t)/ \sqrt{t} \vert \Vert_{2,s}
=
\Vert \log \vert w(1) \vert \Vert_{2,s}
< +\infty
$
iff $s < \frac{1}{2}$.

\item[$\circ$]
A similar computation shows that
$$\Vert ( \mathrm{p.v.} \frac{1}{x} ) (w(1)) \Vert_{2,s}
< +\infty
\quad
\text{iff $s<-\frac{1}{2}$.}
$$
\end{itemize}
Now, for $s < 1/2$, we have
\begin{equation*}
\begin{split}
&
\int_{0}^{T}
\Vert \log \vert w(t) \vert \Vert_{2,s}^{2}
\mathrm{d}t
\leqslant
4
\Big(
\int_{0}^{T}
( \log \sqrt{t} )^{2}
\mathrm{d}t
+
T\Vert \log \vert w(1) \vert \Vert_{2,s}^{2}
\Big)
< +\infty ,
\end{split}
\end{equation*}
so that
$
\int_{0}^{T}
\log \vert w(t) \vert
\mathrm{d}w(t)
\in \mathbb{D}_{2}^{(1/2)-}
$
by Theorem \ref{reg-pres}--(ii).
Hence
by (\ref{log-Ito}),
$$
\int_{0}^{T}
\Big( \mathrm{p.v.} \frac{1}{x} \Big) (w(t))
\mathrm{d}t
\in \mathbb{D}_{2}^{(1/2)-}.
$$

\end{Eg} 

\begin{Eg} 
\label{multi-dim_local_time} 
Let $d \geqslant 2$ and $x,y \in \mathbb{R}^{d}$.
Suppose {\rm (H1)}, {\rm (H3)} and
{\rm (H4)}.
Since
$\delta_{y} \in H_{p}^{s} ( \mathbb{R}^{d} )$
if
$s < - (p-1)d/p$
(Lemma \ref{resol_lem}),
we find from Corollary
\ref{Main_Thm5},
$$
\int_{t_{0}}^{T}
\delta_{y} ( X_{t} )
\mathrm{d}t
\in
\mathbb{D}_{p}^{(1-\frac{(p-1)d}{p})-}
\quad
\text{for $t_{0} > 0$ and $p\geqslant 2$,}
$$
where $X_{t} = X(t,x,w)$.

Assume that $x \neq y$, and then
we shall further investigate the class to which
$
\int_{0}^{T}
\delta_{y} (X_{t}) \mathrm{d}t
$
belongs.
Let
$
f := (1-L)^{-1} \delta_{y}
$
and let
$\phi : \mathbb{R}^{d} \to \mathbb{R}$
be a $C^{\infty}$-function such that
(1)
$
x
\notin \mathrm{supp} \phi
$,
(2)
$\phi \equiv 1$
on a neighbourhood of
$
y
$,
and
(3)
$\mathrm{supp} \phi$
is compact.
By Theorem \ref{Ito_Formula1},
we have
\begin{equation*}
\begin{split}
(\phi f)( X_{T} )
- (\phi f)( x )
=
\sum_{i,j=1}^{
	d
}
\int_{0}^{T}
( \sigma_{j}^{i} \partial_{i} (\phi f) )
(X_{t})
\mathrm{d}w^{j}(t)
+
\int_{0}^{T}
(L (\phi f))( X_{t} )
\mathrm{d}t ,
\end{split}
\end{equation*}
in which we note that
$$
L ( \phi f )
=
( L \phi ) f
+
\langle
	\sigma \nabla \phi ,
	\sigma \nabla f
\rangle_{\mathbb{R}^{d}}
+
\phi f - \delta_{y}
$$
and
$
( L \phi ) f,
\langle
	\sigma \nabla \phi ,
	\sigma \nabla f
\rangle_{\mathbb{R}^{d}}
\in \mathscr{S}(\mathbb{R}^{d})
$.
By Lemma \ref{der-conv}--(ii),
$$
\left\{\begin{array}{l}
p \in (1, \infty) \text{;} \\
-\frac{d}{2} < s < \min \{ \frac{p}{p-1} - d, 0 \}
\end{array}\right.
\Rightarrow
\left\{\begin{array}{l}
\lim_{t \downarrow 0}
\Vert
	( \phi f ) (X_{t})
\Vert_{p,s}
= 0 \text{;} \\
\lim_{t \downarrow 0}
\Vert
	\partial_{i} ( \phi f ) (X_{t})
\Vert_{p,s}
= 0,
\end{array}\right.
$$
and then we have
$
\lim_{t \downarrow 0}
\Vert
	( \sigma_{j}^{i} \partial_{i} (\phi f) ) (X_{t})
\Vert_{p,s}
= 0
$
since
$
\lim_{t \downarrow 0}
\sigma_{j}^{i} (X_{t})
=
\sigma_{j}^{i} ( x )
$
in $\mathbb{D}^{\infty}$
and
$
\Vert
	( \sigma_{j}^{i} \partial_{i} (\phi f) ) (X_{t})
\Vert_{p,s}
\leqslant
\mathrm{const.}
\Vert
	\sigma_{j}^{i} (X_{t})
\Vert_{q,-s^{\prime}}
\Vert
	\partial_{i} (\phi f) (X_{t})
\Vert_{r,s^{\prime}}
$
(this follows by taking the dual of \cite[inequality (1.6)]{Wa93}),
where $q,r \in (1, \infty)$ and $s^{\prime} \in \mathbb{R}$
are such that $1/p = 1/q + 1/r$
and
$s < s^{\prime} < \min \{ \frac{p}{p-1} - d, 0 \}$
($\leqslant \min \{ \frac{q}{q-1} - d, 0 \}$).
Now by Theorem \ref{reg-pres},
we obtain
$
\int_{0}^{T}
( \sigma_{j}^{i} \partial_{i} (\phi f) ) (X_{t})
\mathrm{d}w^{j}(t)
\in \mathbb{D}_{p}^{s}
$
for $p \geqslant 2$,
so that
$$
\int_{0}^{T}
\delta_{y} (X_{t})
\mathrm{d}t
\in \mathbb{D}_{p}^{(\frac{p}{p-1}-d)-}
\quad
\text{for any $d \geqslant 2$ and $p \geqslant 2$.}
$$
\end{Eg} 

\begin{Rm} 
In Examples \ref{delta_prime} and \ref{multi-dim_local_time},
the conclusions
seem
not the best possible given
$p \geqslant 2$,
because
recalling that
$
\int_{t_{0}}^{T} \delta_{y}^{\prime} ( X_{t} ) \mathrm{d}t
$
and
$
\int_{t_{0}}^{T} \delta_{y} ( X_{t} ) \mathrm{d}t
$
belong to
$\mathbb{D}_{p}^{(\frac{1}{p}-1)-}$
and
$\mathbb{D}_{p}^{(1 - \frac{(p-1)d}{p})-}$
respectively for each $t_{0}>0$
(Note that $X_{t}$'s in each case are different
though we are using the same symbol),
it is natural to ask that
$
\int_{0}^{T} \delta_{y}^{\prime} ( X_{t} ) \mathrm{d}t
$
and
$
\int_{0}^{T} \delta_{y} ( X_{t} ) \mathrm{d}t
$
also do.
To get further results for $p \in (1,2)$,
one would have to investigate Theorem \ref{reg-pres}
for $p \in (1,2)$, which we could not.
\end{Rm} 

\appendix
\section{Auxiliary Lemmas}
\label{auxiliary} 
\subsection{Some knowledge of Hermite polynomials}

The {\it Hermite polynomials} $H_{n}$,
$n \in \mathbb{Z}_{\geqslant 0}$
is defined by
$H_{0} (x) = 1$
and
$H_{n}(x) := \partial^{*n} 1 (x)$
for
$n \in \mathbb{N}$ and $x \in \mathbb{R}$,
where
$$
\partial^{*}f (x)
:=
- f^{\prime} (x) + x f(x) ,
\quad
x \in \mathbb{R}
$$
for any differentiable function
$f: \mathbb{R} \to \mathbb{R}$.
The {\it Hermite functions} are now defined by
$$
\phi_{n} (x)
:=
H_{n} ( \sqrt{2} x )
\mathrm{e}^{-\frac{x^{2}}{2}},
\quad
\text{$
x \in \mathbb{R}$
and
$n \in \mathbb{Z}_{\geqslant 0}$.}
$$

Some facts about
$\{ H_{n} \}_{n=0}^{\infty}$
and
$\{ \phi_{n} \}_{n=0}^{\infty}$
are summarized as follows.

\begin{Lem} 
\label{Hermite-facts} 
\begin{itemize}
\item[(i)]
For each $n \in \mathbb{Z}_{\geqslant 0}$,
$$
H_{n}(x)
=
(-1)^{n}
\mathrm{e}^{\frac{x^{2}}{2}}
\frac{\mathrm{d}^{n}}{\mathrm{d}x^{n}}
\mathrm{e}^{-\frac{x^{2}}{2}}
=
\frac{ (-1)^{n} \mathrm{e}^{\frac{x^{2}}{2}} }{ \sqrt{2\pi} }
\int_{-\infty}^{+\infty}
( i \xi )^{n}
\mathrm{e}^{-\frac{\xi^{2}}{2}}
\mathrm{e}^{i\xi x}
\mathrm{d}\xi .
$$

\item[(ii)]
$
H_{n}^{\prime} = n H_{n-1}
$
and
$
H_{n}(x) = x H_{n-1}(x) - (n-1) H_{n-2} (x)
$.

\item[(iii)]
$\{ \frac{1}{\sqrt{n!}} H_{n} \}_{n \geqslant 0}$
is a complete orthonormal basis of
$L_{2}(\mathbb{R}, (2\pi)^{-1/2} \mathrm{e}^{-x^{2}/2} \mathrm{d}x )$.

\item[(iv)]
$\{ ( \sqrt{\pi} n! )^{-1/2} \phi_{n} \}_{n \geqslant 0}$
is a complete orthonormal
system
of
$L_{2}(\mathbb{R}, \mathrm{d}x )$.

\item[(v)]
$
( 1 + x^{2} - \triangle )
\phi_{n}
=
2 (n+1) \phi_{n}
$
for $n \in \mathbb{Z}_{\geqslant 0}$,
where $\triangle = \frac{\mathrm{d}^{2}}{\mathrm{d}x^{2}}$.
\end{itemize}
\end{Lem} 

\subsection{Some knowledge of Heaviside function}

We begin with introducing some results by
Bongioanni-Torrea \cite{BoTo}.
Let
$H := x^{2} - \triangle$,
$A := \frac{\mathrm{d}}{\mathrm{d}x} + x$
and
$B := - \frac{\mathrm{d}}{\mathrm{d}x} + x$.
From the point of view of Lemma \ref{Hermite-facts}--(iii), (iv),
A family of linear operators
$( \lambda + H )^{s}$, $s \in \mathbb{R}$
and
$\lambda \geqslant 0$
can
also
be defined by the relation
$
( \lambda + H )^{s} \phi_{n}
=
(2(n+ \frac{\lambda +1}{2}))^{s}
\phi_{n}
$
for
$n \in \mathbb{Z}_{\geqslant 0}$.

Then for $s<0$,
the operator $H^{s/2}$ has an integral representation
(see \cite[Proposition 2]{BoTo})
\begin{equation}
\label{fractional_H} 
( H^{s/2}f )(x)
=
\int_{\mathbb{R}}
K_{s/2}(x,y) f(y) \mathrm{d}y
\quad
\text{for $f \in \mathscr{S}(\mathbb{R})$,}
\end{equation}
where the kernel $K_{s/2}(x,z)$ has an estimate
$
K_{s/2} (x,y)
\leqslant
c \Phi_{s/2} ( \vert x-y \vert )
$
for some constant $c>0$ and
\begin{equation*}
\Phi_{s/2} ( x )
=
\left\{\begin{array}{ll}
\vert x \vert^{- (1+s)}
1_{\{ \vert x \vert < 1 \}}
+
\mathrm{e}^{ -x^{2}/4 }
1_{\{ \vert x \vert \geqslant 1 \}}
&
\text{if $s > -1$,} \\
( 1 - \log \vert x \vert )
1_{\{ \vert x \vert < 1 \}}
+
\mathrm{e}^{ -x^{2}/4 }
1_{\{ \vert x \vert \geqslant 1 \}}
&
\text{if $s = -1$,} \\
1_{\{ \vert x \vert < 1 \}}
+
\mathrm{e}^{ -x^{2}/4 }
1_{\{ \vert x \vert \geqslant 1 \}}
&
\text{if $s < -1$.}
\end{array}\right.
\end{equation*}

\begin{Lem} 
\label{delta-class} 
If
$s \in ( -1, - \frac{1}{2} )$
and
$p \in (1, \frac{1}{1+s})$,
$
H^{s/2} \delta_{y} (x)
\in L_{p} ( \mathbb{R}, \mathrm{d}x )
$.
\end{Lem} 
\begin{proof} 
Since $s > -1$, we have
\begin{equation*}
\begin{split}
( H^{s/2} \delta_{y} )(x)
&=
K_{s/2} (x,y)
\leqslant
c
\Big(
\vert x-y \vert^{- (1+s)}
1_{\{ \vert x-y \vert < 1 \}}
+
\mathrm{e}^{- \frac{\vert x-y \vert^{2}}{4} }
1_{\{ \vert x-y \vert \geqslant 1 \}}
\Big) .
\end{split}
\end{equation*}
Hence we can easily conclude the result.
\end{proof} 

By using
\cite[Theorem 4, Lemma 4, Theorem 7]{BoTo}
and
$L_{p}$-multiplier theorem
(see Thangavelu \cite[Chapter 4, Section 4.2, Theorem 4.2.1]{Th})
for operators
$[H^{-1}(H+2)]^{s/2}$,
$s \in \mathbb{R}$,
we can deduce that
for each $s \in \mathbb{R}$,
there exists a constant $c_{1} = c_{1}(s) > 0$ such that
\begin{equation}
\label{eq:xx} 
\begin{split}
&
\Vert H^{(s+1)/2} f \Vert_{L_{p}(\mathbb{R}, \mathrm{d}x)}
\leqslant
c_{1}
\Big(
	\Vert H^{s/2} \phi \Vert_{L_{p}(\mathbb{R}, \mathrm{d}x)}
	+
	\Vert x H^{s/2} f \Vert_{L_{p}(\mathbb{R}, \mathrm{d}x)}
	+
	\Vert H^{s/2} f \Vert_{L_{p}(\mathbb{R}, \mathrm{d}x)}
\Big)
\end{split}
\end{equation}
for all $f = f(x) \in \mathscr{S}(\mathbb{R})$,
where $\phi := Af$.
By a standard argument,
this inequality extends
and is still valid for
$f = f (x) \in \mathscr{S}^{\prime}(\mathbb{R})$
such that
$
(H^{s/2} Af),
(H^{s/2} f),
(x H^{s/2} f) \in L_{p}(\mathbb{R}, \mathrm{d}x)
$.

\begin{Prop} 
\label{Heaviside_class} 
Let $y \in \mathbb{R}$.
For
each
$s < 1/2$
and
$p \in (1, 1/s)$,
we have
\begin{itemize}
\item[(i)]
$
H^{s /2} 1_{ ( -\infty , y ) }
\in
L_{p} ( \mathbb{R}, \mathrm{d}x )
$,

\item[(ii)]
$
\lim_{y \to - \infty}
\Vert
	H^{s/2} 1_{ ( -\infty , y ) }
\Vert_{ L_{p} ( \mathbb{R}, \mathrm{d}x ) }
= 0
$
and

\item[(iii)]
$
\sup_{y \in \mathbb{R}}
\Vert
	H^{s/2} 1_{ ( -\infty , y ) }
\Vert_{ L_{p} ( \mathbb{R}, \mathrm{d}x ) }
< \infty
$.
\end{itemize}
\end{Prop} 
\begin{proof} 
Let $f(x) := 1_{(-\infty , y)} (x)$ and
$
\phi (x) := Af (x)
=
\delta_{y} (x) + x 1_{(-\infty , y)} (x)
$.
Then by (\ref{eq:xx}), we have
\begin{equation*}
\begin{split}
&
\Vert
	H^{s/2} 1_{ ( -\infty , y ) }
\Vert_{ L_{p} ( \mathbb{R}, \mathrm{d}x ) } \\
&\leqslant
2c_{1}
\Big(
	\Vert
	H^{(s-1)/2} \delta_{y}
	\Vert_{L_{p}(\mathbb{R}, \mathrm{d}x)}
	+
	\Vert
	H^{
		(s-1)/2
	}
	( x 1_{(-\infty , y)} )
	\Vert_{L_{p}(\mathbb{R}, \mathrm{d}x)} \\
&\hspace{20mm}
	+
	\Vert
	x
	H^{
		(s-1)/2
	}
	1_{(-\infty , y)}
	\Vert_{L_{p}(\mathbb{R}, \mathrm{d}x)}
	+
	\Vert
	H^{
		(s-1)/2
	}
	1_{(-\infty , y)}
	\Vert_{L_{p}(\mathbb{R}, \mathrm{d}x)}
\Big) ,
\end{split}
\end{equation*}
which is finite by virtue of Lemma \ref{delta-class} and
by a similar argument in Lemma \ref{delta-class}
with using
the integral expression
(\ref{fractional_H}).
Furthermore again by (\ref{fractional_H}),
we find also that
the last four terms converge to zero as $y \to -\infty$,
and
uniformly bounded in $y \in \mathbb{R}$.
\end{proof} 

\subsection{Fractional inequalities}
\label{frac_ineq} 

Let
$
\{ F_{ \varepsilon } \}_{\varepsilon \in I}
\subset \mathbb{D}^{\infty}
$
be a bounded and uniformly non-degenerate family
(see Definition \ref{non-degeneracy}).
Here, the boundedness is used in the sense of
$
\{ F_{\varepsilon} \}_{\varepsilon \in I}
$
is bounded in $\mathbb{D}_{p}^{k}$
for each $p \in (1, \infty)$ and $k \in \mathbb{Z}_{\geqslant 0}$.

\begin{Lem} 
\label{fractional_domin} 

For any $s \in \mathbb{R}$,
$p \in ( 1, \infty )$
and $p^{\prime} > p$,
there exists
$
c = c( s,p,p^{\prime}, \{ F_{\varepsilon} \}_{\varepsilon \in I}) > 0
$
such that
$$
\Vert
	\Lambda ( F_{\varepsilon} )
\Vert_{p,s}
\leqslant
c
\Vert
	( x^{2} - \triangle )^{s/2} \Lambda
\Vert_{L_{p^{\prime}}(\mathbb{R}, \mathrm{d}x)}
$$
for every $\varepsilon \in I$
and
$
\Lambda \in \mathscr{S}^{\prime}(\mathbb{R})
$
with
$
( x^{2} - \triangle )^{s/2} \Lambda
\in
L_{p^{\prime}}(\mathbb{R}, \mathrm{d}x)
$.
\end{Lem} 

See also Remark \ref{p'=p}.
The following proof is based
on the technique in \cite{Wa91}.

\begin{proof} 
We show in the case of
$-2 \leqslant s \leqslant 1$.
Other cases are similarly proved.
Define a linear operator
$T_{\alpha} ( \varepsilon )$
for
$-2 \leqslant \alpha \leqslant 1$
and
$\varepsilon \in I$
by
$$
T_{\alpha} ( \varepsilon ) \phi
:=
( I - \mathcal{L} )^{\alpha /2}
\big[
	( x^{2} - \triangle )^{ - \alpha /2 }
	\phi
\big]
( F_{\varepsilon} )
$$
for $\phi = \phi (x) \in \mathscr{S}( \mathbb{R} )$.
Since
$\{ F_{\varepsilon} \}_{\varepsilon \in I}$
is uniformly non-degenerate,
the density function $p_{F_{\varepsilon}} (x)$
is uniformly bounded in
$
( \varepsilon , x ) \in I \times \mathbb{R}
$:
$$
c_{0}
:=
\sup_{\varepsilon \in I}
\sup_{x \in \mathbb{R}}
p_{F_{\varepsilon}} (x) < \infty .
$$
Take $p^{\prime} > p > 1$ arbitrary.
We divide the proof into
five
steps.

(a)
In the first place,
when $\alpha = -2$,
we shall show that
$
T_{-2}(\varepsilon) :
L_{p^{\prime}} ( \mathbb{R}, \mathrm{d}x )
\to
\mathbb{D}_{p}^{0} = L_{p}
$
and is a continuous linear operator
with an estimate
$$
\Vert
	T_{-2}( \varepsilon )
\Vert_{L_{p^{\prime}} (\mathbb{R}, \mathrm{d}x) \to L_{p}}
\leqslant
c_{1}
\quad
\text{for any $\varepsilon \in I$}
$$
for some constant $c_{1}>0$.
Let
$\phi \in \mathscr{S}(\mathbb{R})$ be arbitrary.
Then
$$
\Vert
	T_{-2}(\varepsilon) \phi
\Vert_{p}
=
\Vert
	(I-\mathcal{L})^{-1}
	\big[
		( x^{2} - \triangle ) \phi
	\big]
	( F_{\varepsilon} )
\Vert_{p}
=
\Vert
	( H \phi ) (F_{\varepsilon})
\Vert_{p,-2},
$$
where $H := x^{2} - \triangle$,
and for each $J \in \mathbb{D}^{\infty}$
and $p^{\prime\prime} \in (p, p^{\prime})$,
we have
\begin{equation*}
\begin{split}
\mathbf{E} [ ( H \phi ) ( F_{\varepsilon} ) J ]
=
\mathbf{E} [ ( ( x^{2} - \triangle ) \phi ) ( F_{\varepsilon} ) J ]
\leqslant
\Vert
	\phi ( F_{\varepsilon} )
\Vert_{p^{\prime\prime}}
\Vert
	l_{\varepsilon} (J)
\Vert_{q^{\prime\prime}},
\end{split}
\end{equation*}
where
$
1/p^{\prime\prime}
+
1/q^{\prime\prime}
= 1
$
and
$l_{\varepsilon}(J)$
is of the form
$
l_{\varepsilon} ( J )
=
\sum_{i=0}^{2}
\langle
	P_{i} ( \varepsilon , w ),
	D^{i}J
\rangle_{H^{\otimes i}}
$
for some
$
P_{i}(\varepsilon , w)
\in \mathbb{D}^{\infty} ( H^{\otimes i} )
$,
$i=0,1,2$,
all of which are polynomials in
$F_{\varepsilon}$,
its derivatives up to the second order
and $\Vert DF_{\varepsilon} \Vert_{H}^{-2}$.
Since
$\{ F_{\varepsilon} \}_{\varepsilon \in I}$
is
uniformly
non-degenerate, there exists $c_{0}^{\prime} > 0$
such that
$
\Vert
	l_{\varepsilon} (J)
\Vert_{q^{\prime\prime}}
\leqslant
c_{0}^{\prime}
\Vert J \Vert_{q,2}
$.
Hence we have
$$
\mathbf{E} [ ( H \phi )( F_{\varepsilon} ) J ]
\leqslant
c_{0}^{\prime}
\Vert
	\phi ( F_{\varepsilon} )
\Vert_{p^{\prime\prime}}
\Vert J \Vert_{q,2}
\leqslant
c_{0} c_{0}^{\prime}
\Vert
	\phi
\Vert_{L_{p^{\prime}}(\mathbb{R}, \mathrm{d}x)}
\Vert J \Vert_{q,2}
$$
for each $J \in \mathbb{D}^{\infty}$,
which implies
$
\Vert T_{-2} (\varepsilon) \phi \Vert_{p}
\leqslant
c_{1}
\Vert
	\phi
\Vert_{L_{p^{\prime}}(\mathbb{R}, \mathrm{d}x)}
$,
where $c_{1} := c_{0} c_{0}^{\prime}$.
Since $\mathscr{S}(\mathbb{R})$ is dense in
$
L_{p^{\prime}} ( \mathbb{R}, \mathrm{d}x )
$,
we obtain the desired estimate.

(b)
Next, we focus on the case of $\alpha = 1$.
We shall show that this operator actually
defines a continuous linear mapping
$
T_{1} ( \varepsilon ):
L_{p^{\prime}} ( \mathbb{R}, \mathrm{d}x )
\to
\mathbb{D}_{p}^{0} = L_{p}
$
with an estimate
$$
\Vert
	T_{1} ( \varepsilon )
\Vert_{L_{p^{\prime}} (\mathbb{R}, \mathrm{d}x) \to L_{p}}
\leqslant
c_{2}
\quad
\text{for any $\varepsilon \in I$,}
$$
for some constant $c_{2}>0$.
Let $\phi = \phi (x) \in \mathscr{S}(\mathbb{R})$ be arbitrary.
Then
\begin{equation*}
\begin{split}
&
\Vert T_{1}( \varepsilon ) \phi \Vert_{L_{p}}
=
\big\Vert
	( I-\mathcal{L} )^{1/2}
	\big[
		( x^{2} - \triangle )^{-1/2} \phi
	\big]
	( F_{\varepsilon} )
\big\Vert_{L_{p}}
=
\Vert
	( H^{-1/2} \phi	)
	( F_{\varepsilon} )
\Vert_{p,1}.
\end{split}
\end{equation*}
By Meyer's inequality, there exists a
positive constant
$c_{2}^{\prime} > 0$
such that
$
\Vert J \Vert_{p,1}
\leqslant
c_{2}^{\prime}
(
	\Vert J \Vert_{L_{p}}
	+
	\Vert DJ \Vert_{L_{p}(H)}
)
$
for every $J \in \mathbb{D}_{2}^{1}$.
Hence we have
\begin{equation*}
\begin{split}
&
\Vert
	( H^{-1/2} \phi )
	( F_{\varepsilon} )
\Vert_{p,1}
\leqslant
c_{2}^{\prime}
\Big(
\Vert
	( H^{-1/2} \phi )
	( F_{\varepsilon} )
\Vert_{L_{p}}
+
\Vert
	( H^{-1/2} \phi )^{\prime}
	( F_{\varepsilon} )
	D F_{\varepsilon}
\Vert_{L_{p}(H)}
\Big) .
\end{split}
\end{equation*}
We easily have
$
\Vert
	( H^{-1/2} \phi )
	( F_{\varepsilon} )
\Vert_{L_{p}}
\leqslant
c_{0}
\Vert
	( x^{2} - \triangle )^{-1/2} \phi
\Vert_{L_{p^{\prime}}(\mathbb{R}, \mathrm{d}x)}
$.
Take $q^{\prime} \in (1, \infty)$
so that
$1/p^{\prime} + 1/q^{\prime} = 1/p$.
Then there exists
$
c_{2}^{\prime\prime}
= c_{2}^{\prime\prime}(p^{\prime}, q^{\prime}) > 0
$
such that
$
\Vert J_{1} J_{2} \Vert_{L_{p}(H)}
\leqslant
c_{2}^{\prime\prime}
\Vert J_{1} \Vert_{L_{p^{\prime}}}
\Vert J_{2} \Vert_{L_{q^{\prime}}(H)}
$
for all
$
(J_{1}, J_{2})
\in
L_{p^{\prime}} \times L_{q^{\prime}} (H)
$.
Hence we have
\begin{equation*}
\begin{split}
\Vert
	( H^{-1/2} \phi )^{\prime}
	( F_{\varepsilon} )
	D F_{\varepsilon}
\Vert_{L_{p}(H)}
&\leqslant
c_{2}^{\prime\prime}
\Vert
	[ ( x^{2} - \triangle )^{-1/2} \phi ]^{\prime}
	( F_{\varepsilon} )
\Vert_{L_{p^{\prime}}}
\Vert
	D F_{\varepsilon}
\Vert_{L_{q^{\prime}}(H)} \\
&\leqslant
c_{2}^{\prime\prime}
c_{0}
\Big(
\sup_{\varepsilon \in I}
\Vert
	D F_{\varepsilon}
\Vert_{L_{q^{\prime}}(H)}
\Big)
\Vert
	[ ( x^{2} - \triangle )^{-1/2} \phi ]^{\prime}
\Vert_{L_{p^{\prime}}(\mathbb{R}, \mathrm{d}x)}.
\end{split}
\end{equation*}
By a result by
Bongioanni-Torrea
\cite[Lemma~3 and Theorem~4]{BoTo},
there exists a constant
$
c_{2}^{\prime\prime\prime}
=
c_{2}^{\prime\prime\prime} ( p^{\prime} ) > 0
$
such that
\begin{equation*}
\begin{split}
&
\Vert
	[ ( x^{2} - \triangle )^{-1/2} \phi ]^{\prime}
\Vert_{L_{p^{\prime}}(\mathbb{R}, \mathrm{d}x)}
+
\Vert
	( x^{2} - \triangle )^{-1/2} \phi
\Vert_{L_{p^{\prime}}(\mathbb{R}, \mathrm{d}x)} \\
&\leqslant
\Vert
	\Big( \frac{\mathrm{d}}{\mathrm{d}x} + x \Big)
	(
	x^{2}
	-
	\triangle
	)^{-1/2}
	\phi
\Vert_{L_{p^{\prime}}(\mathbb{R}, \mathrm{d}x)}
+
\Vert
	(
	x^{2}
	-
	\triangle
	)^{-1/2} \phi
\Vert_{L_{p^{\prime}}(\mathbb{R}, \mathrm{d}x)} \\
&\hspace{25mm}+
\Vert
	x
	( x^{2} - \triangle )^{-1/2}
	\phi
\Vert_{L_{p^{\prime}}(\mathbb{R}, \mathrm{d}x)} \\
&\leqslant
c_{2}^{\prime\prime\prime}
\Vert \phi \Vert_{L_{p^{\prime}}(\mathbb{R}, \mathrm{d}x)},
\end{split}
\end{equation*}
and hence we have obtained
$$
\Vert T_{1} ( \varepsilon ) \phi \Vert_{L_{p}}
\leqslant
c_{2}
\Vert \phi \Vert_{L_{p^{\prime}}(\mathbb{R}, \mathrm{d}x)}
\quad
\text{for all $\varepsilon \in I$}
$$
as desired, where
$
c_{2}
:=
c_{0}
c_{2}^{\prime}
(
	1
	+
	c_{2}^{\prime\prime}
	\sup_{\varepsilon \in I}
	\Vert D\widetilde{F}_{\varepsilon} \Vert_{ L_{q^{\prime}} (H) }
)
c_{2}^{\prime\prime\prime}
$.

(c)
For each
$\varepsilon \in I$,
$z \in \mathbb{C}$
and
$\phi \in \mathscr{S}(\mathbb{R})$,
we define
$$
T_{z} ( \varepsilon ) \phi
:=
( I - \mathcal{L} )^{z/2}
[ ( x^{2} - \triangle )^{-z/2} \phi ]
( F_{\varepsilon} ),
\quad
\phi \in \mathscr{S}( \mathbb{R} )
$$
(the operators $( I - \mathcal{L} )^{z/2}$
and $( x^{2} - \triangle )^{-z/2}$ are defined
by the same definition formulae for the real-exponent case).
For any
$
\phi \in \mathscr{S} ( \mathbb{R} )
$
and
$\psi \in \mathbb{D}^{\infty}$,
we shall show
$z \mapsto \mathbf{E} [ ( T_{z} ( \varepsilon ) \phi ) \psi ]$
is analytic.
For this, we note that
\begin{equation}
\label{referee-3} 
\begin{split}
\mathbf{E} [ ( T_{z} ( \varepsilon ) \phi ) \psi ]
&=
\mathbf{E}
[
	( H^{-z/2} \phi ) ( F_{\varepsilon} )
	( I - \mathcal{L} )^{z/2} \psi
] \\
&=
\sum_{m=0}^{\infty}
( 1 + m )^{z/2}
\mathbf{E}
\big[
	( H^{-z/2} \phi ) ( F_{\varepsilon} )
	J_{m} [ \psi ]
\big] .
\end{split}
\end{equation}
Letting
$
\{
	\phi_{n}
\}_{n \in \mathbb{Z}_{\geqslant 0}}
$
be as in Lemma~\ref{Hermite-facts},
we have
$
H^{-z/2} \phi_{n} = ( 2n+1 )^{-z/2} \phi_{n}
$
for $n \in \mathbb{Z}_{\geqslant 0}$
and
$
\{
	\widetilde{\phi}_{n}
	:=
	( \sqrt{\pi} n! )^{-1/2}
	\phi_{n}
\}_{n \in \mathbb{Z}_{\geqslant 0}}
$
is a complete orthonormal system of $L_{2} ( \mathbb{R}, \mathrm{d}x )$.
Then we have
$
\phi
=
\sum_{n=0}^{\infty}
\langle
	\phi ,
	\widetilde{\phi}_{n}
\rangle_{ L_{2} ( \mathbb{R}, \mathrm{d}x ) }
\widetilde{\phi}_{n}
$
in $L_{2} ( \mathbb{R}, \mathrm{d}x )$.
Therefore
$$
\mathbf{E}
[
	( H^{-z/2} \phi ) ( F_{\varepsilon} )
	J_{m} [ \psi ]
]
=
\sum_{n=0}^{\infty}
( 2n+1 )^{-z/2}
\langle
	\phi ,
	\widetilde{\phi}_{n}
\rangle_{ L_{2} ( \mathbb{R}, \mathrm{d}x ) }
\mathbf{E}
[
	\widetilde{\phi}_{n} ( F_{\varepsilon} )
	J_{m} [ \psi ]
]
$$
is an infinite sum of analytic functions in $z$.
Furthermore, this series converges absolutely
and locally-uniformly in $z$.
In fact, with recalling that
$
c_{0}
=
\sup_{\varepsilon \in I}
\sup_{x \in \mathbb{R}}
p_{F_{\varepsilon}} (x) < \infty
$
where $p_{F_{\varepsilon}}$ is
the probability density function of $F_{\varepsilon}$,
and by using the Cauchy-Schwartz inequality,
\begin{equation*}
\begin{split}
&
\sum_{n=0}^{\infty}
\big\vert
( 2n+1 )^{-z/2}
\langle
	\phi ,
	\widetilde{\phi}_{n}
\rangle_{ L_{2} ( \mathbb{R}, \mathrm{d}x ) }
\mathbf{E}
[
	\widetilde{\phi}_{n} ( F_{\varepsilon} )
	J_{m} [ \psi ]
]
\big\vert \\
&\leqslant
c_{0}
\Vert J_{m} [ \psi ] \Vert_{L_{2}}
\sum_{n=0}^{\infty}
\big\vert
( 2n+1 )^{-1}
\langle
	\phi ,
	H^{ 1-\mathrm{Re}(z/2) } \widetilde{\phi}_{n}
\rangle_{ L_{2} ( \mathbb{R}, \mathrm{d}x ) }
\big\vert \\
&\leqslant
c_{0}
\Vert J_{m} [ \psi ] \Vert_{L_{2}}
\Big(
	\sum_{n=0}^{\infty}
	( 2n+1 )^{-2}
\Big)^{1/2}
\Vert
	H^{ 1-\mathrm{Re}(z/2) } \phi
\Vert_{ L_{2} ( \mathbb{R}, \mathrm{d}x ) },
\end{split}
\end{equation*}
is finite because $\phi \in \mathscr{S} (\mathbb{R})$.
Therefore,
$
\mathbf{E}
\big[
	( H^{-z/2} \phi ) ( F_{\varepsilon} )
	J_{m} [ \psi ]
\big]
$
is indeed an analytic function in $z$.
Thus each term in the summation in (\ref{referee-3})
is analytic in $z$.
We see further that the each term has the estimate
\begin{equation*}
\begin{split}
&
\vert
\mathbf{E}
\big[
	( H^{-z/2} \phi ) ( F_{\varepsilon} )
	J_{m} [ \psi ]
\big]
\vert
\leqslant
\Vert
	( H^{-z/2} \phi ) ( F_{\varepsilon} )
\Vert_{L_{2}}
\Vert
	J_{m} [ \psi ]
\Vert_{L_{2}}
\leqslant
c_{0}
\Vert
	H^{-z/2} \phi
\Vert_{L_{2}(\mathbb{R}, \mathrm{d}y)}
\Vert
	J_{m} [ \psi ]
\Vert_{L_{2}}.
\end{split}
\end{equation*}
Therefore
$$
\sum_{m=0}^{\infty}
\vert
( 1 + m )^{z/2}
\mathbf{E}
\big[
	( H^{-z/2} \phi ) ( F_{\varepsilon} )
	J_{m} [ \psi ]
\big]
\vert
\leqslant
c_{0}
\Vert
	H^{-z/2} \phi
\Vert_{L_{2}(\mathbb{R}, \mathrm{d}y)}
\Vert
	\psi
\Vert_{2, \mathrm{Re}(z)} .
$$
Here, the two quantities
$
\Vert
	H^{-z/2} \phi
\Vert_{L_{2}(\mathbb{R}, \mathrm{d}y)}
$
and
$
\Vert
	\psi
\Vert_{2, \mathrm{Re}(z)}
$
are continuously depending on $z$
because
$\phi \in \mathscr{S} ( \mathbb{R} )$
and
$\psi \in \mathbb{D}^{\infty}$.
Hence the infinite sum in (\ref{referee-3})
converges absolutely and locally-uniformly in $z$,
so that
$
\mathbf{E}
[
	( T_{z} (\varepsilon) \phi )
	\psi
]
$
is analytic in $z$.

(d)
For
$\phi \in \mathscr{S}(\mathbb{R})$
and
$\psi \in \mathbb{D}^{\infty}$,
$
\Phi (z) := \mathbf{E} [ ( T_{z} ( \varepsilon ) \phi ) \psi ]
$
is analytic on $\mathbb{C}$.
By
the Marcinkiewicz $L_{p}$-multiplier theorem
(see e.g. \cite[Chapter 4, Section 4.2, Theorem 4.2.1]{Th}),
one gets
$
\sup_{\tau \in \mathbb{R}}
\Vert
( x^{2} - \triangle )^{i\tau}
\Vert_{
	L_{p^{\prime}} ( \mathbb{R}, \mathrm{d}x )
	\to
	L_{p^{\prime}} ( \mathbb{R} \to \mathbb{C}; \mathrm{d}x )
}
< +\infty
$,
where
$L_{p^{\prime}} ( \mathbb{R} \to \mathbb{C}; \mathrm{d}x )$
is the space of all complex-valued measurable functions that
are $p^{\prime}$-th order integrable.
On the other hand, by Meyer's $L_{p}$-multiplier theorem
(see e.g. \cite[Chapter V, Section 8, Lemma 8.2]{IW})
gives
$
\sup_{\tau \in \mathbb{R}}
\Vert
( I - \mathcal{L} )^{i\tau}
\Vert_{
	L_{p}
	\to
	\mathbb{D}_{p}^{0} ( \mathbb{C} )
}
< +\infty
$.
By using these, we can conclude the following:
\begin{itemize}
\item[(1)]
the complex function $\Phi$ has
the estimate
$$
\sup_{\tau \in \mathbb{R}}
\sup_{-2 \leqslant \alpha \leqslant 1}
\vert \Phi ( \alpha + i \tau ) \vert
<
+ \infty
$$
for each $\phi$ and $\psi$
(see Hirschman's Lemma in \cite{Ste56}).

\item[(2)]
for $\tau \in \mathbb{R}$, the operators
$T_{-2+i\tau} ( \varepsilon )$
and
$T_{1+i\tau} ( \varepsilon )$
uniquely extends to a bounded linear operators
$L_{p^{\prime}}( \mathbb{R}, \mathrm{d}x ) \to \mathbb{D}_{p}^{0} ( \mathbb{C} )$
with estimates
\begin{equation*}
\begin{split}
&
\sup_{\tau \in \mathbb{R}}
\Vert
	T_{-2+i\tau} ( \varepsilon )
\Vert_{
	L_{p^{\prime}}( \mathbb{R}, \mathrm{d}x )
	\to
	\mathbb{D}_{p}^{0} ( \mathbb{C} )
	}
< +\infty , \\
&
\sup_{\tau \in \mathbb{R}}
\Vert
	T_{1+i\tau} ( \varepsilon )
\Vert_{
	L_{p^{\prime}}( \mathbb{R}, \mathrm{d}x )
	\to
	\mathbb{D}_{p}^{0} ( \mathbb{C} )
}
< +\infty
\end{split}
\end{equation*}
\end{itemize}
(Note that one can take these upper bounds independent of $\varepsilon$
since the bounds $c_{1}$ and $c_{2}$ obtained in steps (a) and (b) are
independent of $\varepsilon$).
Then by Stein's interpolation theorem\footnote{Strictly speaking,
this theorem was stated in the case where
$\phi$ and $\psi$ are simple functions in \cite[Theorem 1]{Ste56}.
This is because just that his argument is based on
general $\sigma$-finite measure spaces where we can not speak
about `smoothness'.
The same arguments apply to our case.}
(see \cite[Theorem 1]{Ste56}),
we can conclude that
the operators
$T_{z} ( \varepsilon )$
for $-2 \leqslant \mathrm{Re}(z) \leqslant 1$
uniquely extends to a bounded linear operator
$
L_{ p^{\prime} } (\mathbb{R}, \mathrm{d}x)
\to
\mathbb{D}_{ p }^{0} ( \mathbb{C} )
$
and
$$
\sup_{\varepsilon \in I}
\sup_{-2 \leqslant \alpha \leqslant 1}
\Vert
	T_{\alpha} ( \varepsilon )
\Vert_{L_{p^{\prime}}( \mathbb{R}, \mathrm{d}x ) \to L_{p}}
< +\infty .
$$

(e)
Now, for each $\phi \in \mathscr{S}(\mathbb{R})$,
we have
\begin{equation*}
\begin{split}
&
\Vert
	\phi ( F_{\varepsilon} )
\Vert_{p,s}
=
\Vert
	( I - \mathcal{L} )^{s/2}
	[
	( x^{2} - \triangle )^{-s/2}
	( x^{2} - \triangle )^{s/2}
	\phi
	] ( F_{\varepsilon} )
\Vert_{L_{p}} \\
&=
\Vert
	T_{s}( \varepsilon )
	[ ( x^{2} - \triangle )^{s/2} \phi ]
\Vert_{L_{p}}
\leqslant
\Vert
	T_{s}( \varepsilon )
\Vert_{L_{p^{\prime}}(\mathbb{R}, \mathrm{d}x) \to L_{p}}
\Vert
	[ ( x^{2} - \triangle )^{s/2} \phi ]
\Vert_{L_{p^{\prime}}(\mathbb{R}, \mathrm{d}x)}.
\end{split}
\end{equation*}
Again by denseness of $\mathscr{S}(\mathbb{R})$,
this inequality extends to
$\Lambda \in \mathscr{S}^{\prime}(\mathbb{R})$
such that
$
( x^{2} - \triangle )^{s/2} \Lambda
\in
L_{p^{\prime}} (\mathbb{R}, \mathrm{d}x)
$.
\end{proof} 

\begin{Rm} 
\label{p'=p} 
In Lemma \ref{fractional_domin} and
Lemma \ref{frac_dom},
one can take $p^{\prime} = p$
when $F_{\varepsilon} = \varepsilon^{-1} w(\varepsilon^{2}T)$,
where $\varepsilon \in (0,T] =: I$,
$T>0$ and $w=(w^{1}, \cdots , w^{d})$
is the canonical process,
i.e.,
the $d$-dimensional Wiener process starting at zero
(assume $d=1$ if considering Lemma \ref{fractional_domin}),
because then
$
DF_{\varepsilon}
=
\varepsilon^{-1} Dw ( \varepsilon^{2}T )
=
\varepsilon^{-1}
( 1_{[0,\varepsilon^{2}T]} , \cdots , 1_{[0,\varepsilon^{2}T]} )
$
and
$
\langle
	D F_{\varepsilon}^{i}, D F_{\varepsilon}^{j}
\rangle_{H}
= \delta_{ij} T
$
are non-random.
\end{Rm} 

\subsection{Regularity of something related to resolvent kernel associated with elliptic operators}

Suppose {\rm (H1)}, {\rm (H3)} and
{\rm (H4)}.
We set, for
$f \in \mathscr{S}(\mathbb{R}^{d})$
and
$x \in \mathbb{R}^{d}$,
\begin{equation*}
\begin{split}
Lf (x)
&=
\frac{1}{2}
\sum_{i,j=1}^{d}
( \sigma \sigma^{*} )_{j}^{i} (x)
\frac{\partial^{2} f}{\partial x_{i}\partial x_{j}}
(x)
+
\sum_{i=1}^{d}
b^{i} (x)
\frac{\partial f}{\partial x_{i}}
(x) .
\end{split}
\end{equation*}
The fractional power
$( 1 - \triangle )^{s/2}$
is defined as a pseudo-differential operator:
$$
( 1 - \triangle )^{s/2} \phi (x)
=
\int_{\mathbb{R}^{d}}
(1+\vert \xi \vert^{2})^{s/2}
\widehat{\phi} ( \xi )
\mathrm{e}^{i \langle \xi , x \rangle }
\mathrm{d} \xi ,
\quad
\phi \in \mathscr{S}(\mathbb{R}^{d}),
$$
where
$
\widehat{\phi} (\xi)
=
( 2 \pi )^{-d/2}
\int_{\mathbb{R}^{d}}
\phi (y) \mathrm{e}^{- i\langle \xi , y \rangle}
\mathrm{d}y
$
is the Fourier transform of $\phi$,
and is also given by
\begin{equation*}
( 1 - \triangle )^{s/2} \phi (x)
=
\frac{1}{\Gamma (-s/2)}
\int_{0}^{\infty}
t^{-\frac{s}{2} - 1}
\mathrm{e}^{-t}
( \mathrm{e}^{ t \triangle } \phi )(x)
\mathrm{d}t,
\quad
\phi \in \mathscr{S}(\mathbb{R}^{d})
\end{equation*}
for $-d<s<0$,
where $\mathrm{e}^{ t \triangle }$ is
the heat semigroup associated to $\triangle$.
Actually, formulae
$
\phi (y)
=
\int_{\mathbb{R}^{d}}
\widehat{\phi} ( \xi )
\mathrm{e}^{i \langle \xi , y \rangle}
\mathrm{d}\xi
$
and
$
\int_{\mathbb{R}^{d}}
\mathrm{e}^{ - \frac{ \vert y \vert^{2} }{4t} }
\mathrm{e}^{ i \langle \xi , y \rangle }
\mathrm{d}y
=
( 4\pi t )^{d/2}
\mathrm{e}^{-\vert \xi \vert^{2} t}
$
give the above equivalence.

Before entering the following series of estimates,
we recall that
$$
\int_{\vert x \vert < 1}
\vert x \vert^{-s}
\mathrm{d}x < +\infty
\quad
\text{iff}
\quad
s<d,
$$
where the integral is over $\mathbb{R}^{d}$
and
$\vert x \vert = \vert x \vert_{\mathbb{R}^{d}}$.

\begin{Lem} 
\label{resol_lem} 
For every $p \in (1,\infty)$ and
$s \in (-d, -(p-1)d/p )$,
we have
$
( 1-\triangle )^{s/2} \delta_{y} \in L_{p} (\mathbb{R}^{d}, \mathrm{d}x)
$
for each $y \in \mathbb{R}^{d}$.
\end{Lem} 
\begin{proof} 
The transition density associated to
the semigroup generated by
$\triangle$
is given by
$
p_{t}(x,y)
=
(4\pi t)^{-d/2}
\exp ( - \vert y-x \vert^{2} / (4t) )
$.
Hence we have
\begin{equation*}
\begin{split}
&
( 1 - \triangle )^{s/2} \delta_{y} (x)
=
\frac{1}{\Gamma (-s/2)}
\int_{0}^{\infty}
t^{ -\frac{s}{2} - 1 }
\mathrm{e}^{-t}
p_{t} (x,y)
\mathrm{d}t \\
&=
\frac{ ( 4\pi )^{-d/2} }{\Gamma (-s/2)}
\int_{0}^{\infty}
t^{ -\frac{s+d}{2} - 1 }
\mathrm{e}^{-t}
\exp
\Big\{
	- \frac{ \vert x - y \vert^{2} }{4t}
\Big\}
\mathrm{d}t \\
&=
\frac{ ( 4\pi  )^{-d/2} }{4^{-\frac{s+d}{2}}\Gamma (-s/2)}
\vert x-y \vert^{-(d+s)}
\int_{0}^{\infty}
u^{ \frac{d+s}{2} - 1 }
\mathrm{e}^{-u}
\exp
\Big\{
	- \frac{ \vert x-y \vert^{2} }{ 4u }
\Big\}
\mathrm{d}u,
\end{split}
\end{equation*}
which behaves, up to a multiplicative constant, as
$\vert x-y \vert^{-(d+s)}$
when $\vert x - y \vert \to 0$
and rapidly decreasing as $\vert x - y \vert \to \infty$.
In fact,
by using the identity
$$
\mathrm{e}^{-u}
\exp
\Big\{
	-\frac{ \vert x-y \vert^{2} }{ 4u }
\Big\}
=
\exp \Big\{
	-
	\frac{
		( \vert x-y \vert - 2u )^{2}
	}{ 4u }
\Big\}
\mathrm{e}^{ - \vert x-y \vert },
$$
and putting
$a := \vert x-y \vert$,
the last integral can be written as
\begin{equation*}
\begin{split}
&
\int_{0}^{\infty}
u^{ \frac{d+s}{2} - 1 }
\mathrm{e}^{-u}
\exp
\Big\{
	- \frac{ a^{2} }{ 4u }
\Big\}
\mathrm{d}u
= 
\mathrm{e}^{ - a }
( I + I\!I )
\end{split}
\end{equation*}
where
\begin{equation*}
\begin{split}
I
&:=
\int_{0}^{a}
u^{ \frac{d+s}{2} - 1 }
\exp \Big\{
	-
	\frac{
		( a - 2u )^{2}
	}{ 4u }
\Big\} 
\mathrm{d}u,
\quad
I\!I
:=
\int_{a}^{\infty}
u^{ \frac{d+s}{2} - 1 }
\exp \Big\{
	-
	\frac{
		( a - 2u )^{2}
	}{ 4u }
\Big\} 
\mathrm{d}u .
\end{split}
\end{equation*}
Since $d+s > 0$,
we have
$
I
\leqslant
\int_{0}^{a}
u^{ \frac{d+s}{2} - 1 }
\mathrm{d}u
=
\frac{2}{d+s}
a^{ \frac{ d+s }{ 2 } }
$.
On the other hand,
by using the change of variable $u = \frac{a}{2}v$,
one has
$
I\!I
=
( a/2 )^{ \frac{d+s}{2} }
\int_{2}^{+\infty}
v^{ \frac{d+s}{2} - 1 }
\exp ( - \frac{ (1-v)^{2} }{2v} a )
\mathrm{d}v
$.
For
$a \geqslant 1$ and $v \geqslant 2$,
we have
$
\exp ( - \frac{ (1-v)^{2} }{2v} a )
=
\exp ( - \frac{a}{2} (v-2) )
\exp ( - \frac{a}{2v} )
\leqslant
\exp ( - \frac{a}{2} (v-2) )
\leqslant
\exp ( - \frac{1}{2} (v-2) )
=
\mathrm{e}
\cdot
\exp ( -v/2 )
$.
Thus we obtain
$
I\!I
\leqslant
c_{0}
\mathrm{e}
( a/2 )^{ \frac{d+s}{2} }
$
if $a \geqslant 1$,
where
$
c_{0}
:=
\int_{2}^{+\infty}
v^{ \frac{d+s}{2} - 1 }
\exp ( - v/2 )
\mathrm{d}v
< +\infty
$.
The arguments above shows that
$
\int_{0}^{\infty}
u^{ \frac{d+s}{2} - 1 }
\mathrm{e}^{-u}
\exp
(
	- \frac{ a^{2} }{ 4u }
)
\mathrm{d}u
$
decreases rapidly when $a \to +\infty$,
as claimed.

Therefore,
$
( 1 - \triangle )^{s/2} \delta_{y} \in L_{p} ( \mathbb{R}^{d}, \mathrm{d}x )
$
if
$(d+s)p < d$,
i.e.,
$s < -(p-1)d/p$.
\end{proof} 

Let
$x \in \mathbb{R}^{d}$.
Under the conditions
{\rm (H1)}, {\rm (H3)} and
{\rm (H4)},
we denote by
$X_{t} = X(t,x,w)$
a unique strong solution to
the stochastic differential equation
(\ref{ref-SDE}).
In the sequel, we fix
$y \in \mathbb{R}^{d}$
such that $y \neq x$ and define
$$
f_{y} (z) := ( 1-L )^{-1} \delta_{y} (z),
\quad
z \in \mathbb{R}^{d}.
$$
Let
$\phi : \mathbb{R}^{d} \to \mathbb{R}$
be a $C^{\infty}$-function such that
(1)
$
x
\notin \mathrm{supp} \phi
$,
(2)
$\phi \equiv 1$
on a neighbourhood of
$
y
$,
and
(3)
$\mathrm{supp} \phi$
is compact.

\begin{Lem} 
\label{der-conv} 

Assume {\rm (H1)}, {\rm (H3)} and {\rm (H4)}.
Let
$p \in (1,\infty)$
be arbitrary.
\begin{itemize}
\item[(i)]
For each
$
s < \min \{ 1-\frac{(p-1)d}{p}, 0 \}
$
and $t > 0$,
we have
$
(\nabla f_{y}) (X_{t})
\in
\mathbb{D}_{p}^{s} (\mathbb{R}^{d})
$
and
$
(0,T] \ni t
\mapsto
(\nabla f_{y}) (X_{t})
\in
\mathbb{D}_{p}^{s} (\mathbb{R}^{d})
$
is continuous.

\item[(ii)]
Assume $d \geqslant 2$.
Then for each
$
-\frac{d}{2} < s < \min \{ \frac{p}{p-1} - d, 0 \}
$,
we have
$
\lim_{t \to 0}
\Vert
	( \nabla ( \phi f_{y} )) ( X_{t} )
\Vert_{p,s}
= 0
$
and
$
\lim_{t \to 0}
\Vert
	( \phi f_{y}) ( X_{t} )
\Vert_{p,s}
= 0
$.
\end{itemize}
\end{Lem} 
\begin{proof} 
Let $p \in (1,\infty )$ and $t_{0}>0$ be arbitrary.
In the following, we write $f:=f_{y}$.

(i)
Take $p^{\prime} > p$ such that
$s < \min \{ 1-\frac{(p^{\prime}-1)d}{p^{\prime}}, 0 \}$.
Since
$\{ X_{t} \}_{t_{0} \leqslant t \leqslant T}$
is uniformly non-degenerate,
the assertion follows by
Lemma \ref{frac_dom}
once we show
$
( 1 - \triangle )^{s/2} ( \partial f / \partial z_{k} )
\in L_{p^{\prime}} (\mathbb{R}^{d}, \mathrm{d}z)
$.

Denote by
$p_{t}(z,z^{\prime})$
the transition density
associated to $L$.
By a standard estimate
(see e.g., \cite[Chapter 9, Section 6, Theorem 7]{Fr}),
there exist $c, C > 0$ such that
\begin{equation*}
\begin{split}
p_{t} (z,z^{\prime})
&\leqslant
C
(2\pi ct)^{-d/2}
\exp
\Big\{
	-
	\frac{
		\vert z - z^{\prime} \vert^{2}
	}{
		2ct
	}
\Big\}, \\
\big\vert
	\frac{\partial p_{t}}{\partial z_{k}}
	( z, z^{\prime} )
\big\vert
&\leqslant
C
t^{-1/2}
(2\pi ct)^{-d/2}
\exp
\Big\{
	-
	\frac{
		\vert z - z^{\prime} \vert^{2}
	}{
		2ct
	}
\Big\}
\end{split}
\end{equation*}
for every $k=1, 2, \cdots , d$ and $z,z^{\prime} \in \mathbb{R}^{d}$.
We may assume $c\geqslant 2$ by rearranging $C>0$.
Then we have
\begin{equation}
\label{der-of-f} 
\begin{split}
\big\vert
	\frac{\partial f}{\partial z_{k}}
	( z )
\big\vert
&\leqslant
\frac{
	1
}{
	\Gamma (1)
}
\int_{0}^{\infty}
\mathrm{e}^{-t}
\big\vert
\frac{\partial p_{t}}{\partial z_{k}} ( z  , y )
\big\vert
\mathrm{d}t \\
&\leqslant
C
\int_{0}^{\infty}
t^{-1/2}
\mathrm{e}^{-t}
(2\pi ct)^{-d/2}
\exp
\Big\{
	-
	\frac{
		\vert z - y \vert^{2}
	}{
		2ct
	}
\Big\}
\mathrm{d}t,
\end{split}
\end{equation}
so that
\begin{equation*}
\begin{split}
\Big\vert
\big[
( 1 - \triangle )^{s/2}
\frac{\partial f}{\partial z_{k}}
\big]
(z)
\Big\vert
& 
\leqslant
\frac{
	1
}{
	\Gamma (-s/2)
}
\int_{0}^{\infty}
t^{-\frac{s}{2}-1}
\mathrm{e}^{-t}
\int_{\mathbb{R}^{d}}
\frac{
	\mathrm{e}^{
			- \frac{ \vert z - z^{\prime} \vert^{2} }{ 4t }
		}
	}{
		\sqrt{ 4 \pi t }^{d}
	}
\big\vert
\frac{\partial f}{\partial z_{k}}
( z^{\prime} )
\big\vert
\mathrm{d} z^{\prime}
\mathrm{d}t \\
&\leqslant 
\frac{
	C^{\prime}
}{
	\Gamma (-s/2)
	\Gamma (1/2)
}
\int_{0}^{\infty}
\int_{0}^{\infty}
t^{-\frac{s}{2}-1}
u^{-1/2}
\mathrm{e}^{-(t+u)}
\frac{
	\mathrm{e}^{ - \frac{\vert z-y \vert^{2}}{2c(t+u)} }
}{
	( 2\pi c(t+u) )^{d/2}
}
\mathrm{d}u
\mathrm{d}t \\
&= 
\frac{
	C^{\prime}
}{
	\Gamma ( \frac{1-s}{2} )
}
\int_{0}^{\infty}
v^{\frac{1-s}{2} - 1}
\mathrm{e}^{-v}
\frac{
	\mathrm{e}^{ - \frac{ \vert z-y \vert^{2} }{ 2cv } }
}{
	( 2\pi cv )^{d/2}
}
\mathrm{d}v,
\end{split}
\end{equation*}
for some constant $C^{\prime} > 0$.
Hence
$
( 1 - \triangle )^{s/2}
(\partial_{k} f)
(z)
\in L_{p^{\prime}}(\mathbb{R}^{d}, \mathrm{d}z)
$.

(ii)
Suppose that
$
s < \min \{ \frac{p}{p-1} - d, 0 \}
$.
Take $\delta > 0$ so that
$
\{
	z \in \mathbb{R}^{d}:
	\vert
	z - x
	\vert
	< \delta
\}
\subset
( \mathrm{supp} \phi )^{c}
$.
Note that
$
\partial_{k} ( \phi f_{y} )
=
( \partial_{k} \phi ) f_{y}
+
\phi ( \partial_{k} f_{y} )
$,
$( \partial_{k} \phi )(x) f_{y}(x) = 0$
and
$
( \partial_{k} \phi ) f_{y} \in \mathscr{S}(\mathbb{R}^{d})
$.
Therefore
by bounded convergence theorem,
$$
\lim_{t \downarrow 0}
\Vert
[ ( \partial_{k} \phi ) f_{y} ] (X_{t})
\Vert_{p,s}
\leqslant
\lim_{t \downarrow 0}
\Vert
[ ( \partial_{k} \phi ) f_{y} ] (X_{t})
\Vert_{p}
= 0.
$$
So in the following, we investigate the behaviour of
$
( \phi ( \partial_{k} f_{y} ))(X_{t})
$
and
$( \phi f_{y} )(X_{t})$.
For this, we divide the proof into four steps.

(a)
Since $s \leqslant 0$, we notice that
\begin{equation*}
\begin{split}
\vert (I-\mathcal{L})^{s/2}F \vert
&=
\big\vert
\frac{1}{\Gamma (-\frac{s}{2})}
\int_{0}^{\infty}
u^{-\frac{s}{2}-1}
\mathrm{e}^{-u}
T_{u}F
\mathrm{d}u
\big\vert \\
&\leqslant
\frac{1}{\Gamma (-\frac{s}{2})}
\int_{0}^{\infty}
u^{-\frac{s}{2}-1}
\mathrm{e}^{-u}
T_{u} \vert F \vert
\mathrm{d}u
=
(I-\mathcal{L})^{s/2} \vert F \vert
\end{split}
\end{equation*}
for any $F \in L_{2}$,
where $T_{u}=\exp ( u\mathcal{L} )$, $u \geqslant 0$ is
the Ornstein-Uhlenbeck semigroup on the Wiener space.
Then, taking $p^{\prime}, q, r > 1$
so that
$
\frac{1}{p^{\prime}} + \frac{1}{q} + \frac{1}{r} < 1
$
and
with putting
$
F
:=
(\phi \partial_{k} f) (X_{t})
=
(\phi \partial_{k} f) (X_{t})
1_{\{ \vert X_{t} - x \vert > \delta \}}
$,
we have
\begin{equation*}
\begin{split}
&
\Vert
(\phi \partial_{k} f) (X_{t})
\Vert_{p,s}^{p}
=
\Vert
(I-\mathcal{L})^{s/2} F
\Vert_{p}^{p} \\
&\leqslant
\mathbf{E}
\big[
\Big\vert
	(I-\mathcal{L})^{s/2} F
\Big\vert^{p-1}
(I-\mathcal{L})^{s/2} \vert F \vert
\big] \\
&=
\mathbf{E}
\big[
\vert
(\phi \partial_{k} f) (X_{t})
\vert
\Big\{
(I-\mathcal{L})^{s/2}
\big\vert
	(I-\mathcal{L})^{s/2} F
\big\vert^{p-1}
\Big\}
1_{\text{{\small $\{ \vert X_{t} - x \vert > \delta \}$}}}
\big] \\
&\leqslant
c_{0}
\Vert
(\phi \partial_{k} f) (X_{t})
\Vert_{p^{\prime}}
\big\Vert
\vert
	(I-\mathcal{L})^{s/2} F
\vert^{p-1}
\big\Vert_{q,s}
\mathbf{P}
\big( \vert X_{t} - x \vert > \delta \big)^{1/r},
\end{split}
\end{equation*}
where
$
c_{0} = c_{0}(p^{\prime}, q, r) > 0
$
is a constant independent of $t$.
We easily have
$
\Vert
( \phi \partial_{k} f) (X_{t})
\Vert_{p^{\prime}}
\leqslant
\vert \phi \vert_{\infty}
\Vert
( \partial_{k} f) (X_{t})
\Vert_{p^{\prime}}
$
and
\begin{equation*}
\begin{split}
\big\Vert
\vert
	(I-\mathcal{L})^{s/2} F
\vert^{p-1}
\big\Vert_{q,s}
&\leqslant
\Vert
\{
	(I-\mathcal{L})^{s/2} F
\}^{p-1}
\Vert_{q} \\
&=
\Vert
(I-\mathcal{L})^{s/2} F
\Vert_{q(p-1)}^{p-1}
=
\Vert
(I-\mathcal{L})^{s/2} ( \phi \partial_{k} f ) (X_{t})
\Vert_{p^{\prime\prime}}^{p-1},
\end{split}
\end{equation*}
where $p^{\prime\prime} := q(p-1)$.
Thus we have obtained
\begin{equation}
\label{4factors} 
\begin{split}
&
\Vert
	( \phi \partial_{k} f ) (X_{t})
\Vert_{p,s}^{p} \\
&\leqslant
c_{0}
\vert \phi \vert_{\infty}
\Vert
( \partial_{k} f) (X_{t})
\Vert_{p^{\prime}}
\Vert
(I-\mathcal{L})^{s/2} ( \phi \partial_{k} f ) (X_{t})
\Vert_{p^{\prime\prime}}^{p-1}
\mathbf{P}
\big(
	\vert X_{t} - x \vert
	>
	\delta
\big)^{1/r}.
\end{split}
\end{equation}
Similarly, we have
\begin{equation}
\label{4factors'}
\begin{split}
\Vert
	( \phi f ) (X_{t})
\Vert_{p,s}^{p}
&\leqslant
c_{0}^{\prime}
\Vert
	f
	(X_{t})
\Vert_{p^{\prime}}
\Vert
	(I-\mathcal{L})^{s/2}
	( \phi f ) (X_{t})
\Vert_{p^{\prime\prime}}^{p-1}
\mathbf{P}
\big(
	\vert X_{t} - x \vert > \delta
\big)^{1/r}
\end{split}
\end{equation}
for some constant
$c_{0}^{\prime}>0$,
independent of $t$.

(b)
We will write $\varepsilon := \sqrt{t}$ in the sequel.
We shall give estimates for each factors
in (\ref{4factors}) and (\ref{4factors'}),
though the proof is presented in the next step.
Note that for the last
factors in
(\ref{4factors})
and
(\ref{4factors'}),
there exists $c_{3}, c_{3}^{\prime}, K > 0$ such that
\begin{equation}
\label{3rd-fac} 
\begin{split}
\mathbf{P}
\big(
	\vert X_{t} - x \vert > \delta
\big)
\leqslant
c_{3}
\exp \Big( - \frac{\delta^{2}}{ c_{3}^{\prime} \varepsilon^{2} } \Big)
\quad
\text{for $t=\varepsilon^{2} \in ( 0, K ].$}
\end{split}
\end{equation}
Let $p \in (1,\infty)$ anew be arbitrary.
To see
(\ref{4factors}), we shall prove
\begin{equation}
\label{(ii)-a}
\lim_{\varepsilon \downarrow 0}
\varepsilon^{d}
\Vert
(\partial_{k} f) (X_{t})
\Vert_{p}^{p}
= 0
\quad
\text{if
$
p < \frac{d}{d-1}
$,}
\end{equation}
\begin{equation}
\label{(ii)-b}
\begin{split}
&
\lim_{ \varepsilon \downarrow 0 }
\varepsilon^{d-sp}
\Vert
(I-\mathcal{L})^{s/2} ( \phi \partial_{k} f ) (X_{t})
\Vert_{p}^{p}
= 0
\quad
\text{if $p < \frac{d}{d+s-1}$.}
\end{split}
\end{equation}
On the other hand,
for (\ref{4factors'}), we shall prove
\begin{equation}
\label{(iii)-a}
\lim_{\varepsilon \downarrow 0}
\varepsilon^{d}
\Vert
	f ( X_{t} )
\Vert_{
	p
}^{p}
= 0
\quad
\text{if $p < \frac{d}{d-1}$,}
\end{equation}
\begin{equation}
\label{(iii)-b}
\begin{split}
&
\lim_{\varepsilon \downarrow 0}
\varepsilon^{d-sp}
\Vert
( I - \mathcal{L} )^{s/2} ( \phi f ) ( X_{t} )
\Vert_{p}^{p}
= 0 \\
&\hspace{20mm}
\text{if
$
p < \min \big\{ \frac{d}{d+s-1} , \frac{2d}{d-2} \big\}
= \frac{d}{d+s-1}
$,}
\end{split}
\end{equation}
where $\frac{2d}{d-2}$ is understood as $+\infty$ if $d=2$
and the last equality
follows from
$s > - \frac{d}{2}$.

(c)
Take $p^{\prime} > p$ arbitrary.
First we shall prove (\ref{(ii)-a}).
By using
Proposition \ref{trick1}
and
Lemma \ref{frac_dom}
(Recall that
$\varepsilon = \sqrt{t}$
and
$X_{t} = x + \varepsilon F_{\varepsilon}$,
where
$
F_{\varepsilon} = ( X^{\varepsilon}(1,x,w) - x ) / \varepsilon
$
is uniformly non-degenerate),
we have
\begin{equation*}
\begin{split}
&
\Vert
(\partial_{k} f) (X_{t})
\Vert_{p}^{ p^{\prime} }
=
\Vert
(\partial_{k} f) ( x + \varepsilon F_{\varepsilon} )
\Vert_{p}^{p^{\prime}}
\leqslant
c_{1}
\Vert
(\partial_{k} f) ( x + \varepsilon z)
\Vert_{L_{p^{\prime}}(\mathbb{R}^{d}, \mathrm{d}z)}^{ p^{\prime} } \\
&\leqslant
c_{1}^{\prime}
\int_{\mathbb{R}^{d}}
\Big\{
\int_{0}^{\infty}
u^{-1/2}
\mathrm{e}^{-u}
( 2\pi c u )^{-d/2}
\exp
\Big\{
	-
	\frac{\vert \varepsilon z - (y-x) \vert^{2} }{ 2cu }
\Big\}
\mathrm{d}u
\Big\}^{ p^{\prime} }
\mathrm{d}z \\
&=
c_{1}^{\prime}
\varepsilon^{-d}
\int_{\mathbb{R}^{d}}
\Big\{
\vert z - (y-x) \vert^{(1-d)}
\int_{0}^{\infty}
v^{\frac{d-1}{2}-1}
	\mathrm{e}^{-v}
	\mathrm{e}^{\text{{\small $
		- \frac{ \vert z- (y-x) \vert^{2} }{ 2cv }
	$}}}
\mathrm{d}v
\Big\}^{p^{\prime}}
\mathrm{d}z.
\end{split}
\end{equation*}
for some constants $c_{1}, c_{1}^{\prime} > 0$
independent of $\varepsilon$
(In the last equality,
we have applied the change of variables
$
v
=
\frac{
	\vert \varepsilon z - (y-x) \vert^{2}
}{
	2cu
}
$
).
The last factor can be further computed as follows:
\begin{equation*}
\begin{split}
&
\int_{\mathbb{R}^{d}}
\big\{
	\vert z \vert^{(1-d)}
	\int_{0}^{\infty}
	v^{ \frac{d-1}{2} - 1 }
	\mathrm{e}^{ -v }
	\mathrm{e}^{\text{{\small $
		- \frac{ \vert z \vert^{2} }{ 2cv }
	$}}}
	\mathrm{d} v
\big\}^{p^{\prime}}
\mathrm{d} z \\
&\leqslant 
\int_{ \vert z \vert < 1 }
\big\{
	\vert z \vert^{(1-d)}
	\int_{0}^{\infty}
	v^{ \frac{d-1}{2} - 1 }
	\mathrm{e}^{ -v }
	\mathrm{d} v
\big\}^{p^{\prime}}
\mathrm{d} z \\
&\hspace{15mm}+
\int_{ \vert z \vert \geqslant 1 }
\Big\{
\vert z \vert^{(1-d)}
	\int_{0}^{\infty}
	v^{ \frac{d-1}{2} - 1 }
\mathrm{e}^{\text{{\small $
	- \frac{ ( \vert z \vert - \sqrt{2c} v )^{2} }{ 2cv }
$}}}
	\mathrm{e}^{ - \sqrt{(2/c)} \vert z \vert }
	\mathrm{d} v
\Big\}^{p^{\prime}}
\mathrm{d} z .
\end{split}
\end{equation*}
(For getting the second term, note that
$
\mathrm{e}^{-v}
\mathrm{e}^{ - \vert z \vert^{2} / (2cv) }
=
\mathrm{e}^{\text{{\small $
	- \frac{ ( \vert z \vert - \sqrt{2c} v )^{2} }{ 2cv }
$}}}
	\mathrm{e}^{ - \sqrt{(2/c)} \vert z \vert }
$.)
The first term equals to
$
(
\int_{ \vert z \vert < 1 }
\vert z \vert^{ p^{\prime} (1-d) }
\mathrm{d}z
)
(
	\int_{0}^{\infty}
	v^{ \frac{d-1}{2} - 1 }
	\mathrm{e}^{ -v }
	\mathrm{d} v
)^{ p^{\prime} }
$.
The assumption $d\geqslant 2$
assures
$
\int_{0}^{\infty}
v^{\frac{d-1}{2}-1}
\mathrm{e}^{-v}
\mathrm{d}v
< +\infty
$
and hence
the first term is finite
if $p^{\prime} < \frac{d}{d-1}$.
For the second term,
we have
\begin{equation*}
\begin{split}
&
\int_{ \vert z \vert \geqslant 1 }
\Big\{
\vert z \vert^{(1-d)}
\int_{0}^{\infty}
v^{ \frac{d-1}{2} - 1 }
\mathrm{e}^{\text{{\small $
	- \frac{ ( \vert z \vert - \sqrt{2c} v )^{2} }{ 2cv }
$}}}
\mathrm{e}^{ - \sqrt{(2/c)} \vert z \vert }
\mathrm{d} v
\Big\}^{p^{\prime}}
\mathrm{d} z
\leqslant
c_{1}^{\prime\prime}
( I + I\!I ),
\end{split}
\end{equation*}
where
$c_{1}^{\prime\prime} > 0$
is a constant depending only on $p^{\prime}$,
and $I, I\!I$ are defined by
\begin{equation*}
\begin{split}
I
&:=
\int_{ \vert z \vert \geqslant 1 }
\Big\{
\vert z \vert^{(1-d)}
\mathrm{e}^{ - \sqrt{(2/c)} \vert z \vert }
\int_{0}^{ \vert z \vert }
v^{ \frac{d-1}{2} - 1 }
\mathrm{d} v
\Big\}^{p^{\prime}}
\mathrm{d} z, \\
I\!I
&:=
\int_{ \vert z \vert \geqslant 1 }
\Big\{
\vert z \vert^{(1-d)}
\mathrm{e}^{ - \sqrt{(2/c)} \vert z \vert }
\int_{ \vert z \vert }^{\infty}
v^{ \frac{d-1}{2} - 1 }
\mathrm{e}^{\text{{\small $
	- \frac{ ( \vert z \vert - \sqrt{2c} v )^{2} }{ 2cv }
$}}}
\mathrm{d} v
\Big\}^{p^{\prime}}
\mathrm{d} z
\end{split}
\end{equation*}
The term $I$ is estimated as
$
I
= 
( \frac{ 2 }{ d-1 } )^{p^{\prime}}
\int_{ \vert z \vert \geqslant 1 }
\{
\vert z \vert^{-\frac{d-1}{2}}
\mathrm{e}^{ - \sqrt{(2/c)} \vert z \vert }
\}^{p^{\prime}}
\mathrm{d} z
\leqslant
( \frac{ 2 }{ d-1 } )^{p^{\prime}}
\int_{ \vert z \vert \geqslant 1 }
\mathrm{e}^{ - p^{\prime} \sqrt{(2/c)} \vert z \vert }
\mathrm{d} z
< +\infty
$.
On the other hand,
by applying the change of variable
$v = \vert z \vert u$
and noting that
$\vert z \vert \geqslant 1$,
\begin{equation*}
\begin{split}
I\!I
&= 
\int_{ \vert z \vert \geqslant 1 }
\Big\{
\vert z \vert^{ -\frac{d-1}{2} }
\mathrm{e}^{ - \sqrt{(2/c)} \vert z \vert }
\int_{1}^{\infty}
u^{ \frac{d-1}{2} - 1 }
\mathrm{e}^{\text{{\small $
	- \frac{ \vert z \vert ( 1 - \sqrt{2c} u )^{2} }{ 2cu }
$}}}
\mathrm{d} u
\Big\}^{p^{\prime}}
\mathrm{d} z \\
&\leqslant 
\int_{ \vert z \vert \geqslant 1 }
\mathrm{e}^{ - p^{\prime} \sqrt{(2/c)} \vert z \vert }
\Big\{
\int_{1}^{\infty}
u^{ \frac{d-1}{2} - 1 }
\mathrm{e}^{\text{{\small $
	- \frac{ ( 1 - \sqrt{2c} u )^{2} }{ 2cu }
$}}}
\mathrm{d} u
\Big\}^{p^{\prime}}
\mathrm{d} z .
\end{split}
\end{equation*}
Since
$
\exp
(
	- \frac{ ( 1 - \sqrt{2c} u )^{2} }{ 2cu }
)
=
\mathrm{e}^{ - \frac{1}{2cu} }
\mathrm{e}^{ \sqrt{2/c} }
\exp ( - u )
\leqslant
\mathrm{e}^{ \sqrt{2/c} }
\exp ( - u )
$,
we get
$$
I\!I
\leqslant
\mathrm{e}^{ p^{\prime} \sqrt{2/c} }
\Big(
\int_{ \vert z \vert \geqslant 1 }
\mathrm{e}^{ - p^{\prime} \sqrt{2/c} \hspace{0.5mm} \vert z \vert }
\mathrm{d} z
\Big)
\Big(
\int_{1}^{\infty}
u^{ \frac{d-1}{2} - 1 }
\mathrm{e}^{-u}
\mathrm{d}u
\Big)^{p^{\prime}}
< +\infty
$$
(the arguments so far will be used repeatedly below).
Putting all together, we have obtained that
for $p^{\prime} \in (p, \frac{d}{d-1})$,
$
\limsup_{\varepsilon \downarrow 0}
\varepsilon^{d}
\Vert
	( \partial_{k} f ) ( X_{t} )
\Vert_{p}^{p^{\prime}}
< +\infty
$,
which implies (\ref{(ii)-a}).

For (\ref{(ii)-b}),
with assuming $\varepsilon \in (0,1]$,
we have
\begin{equation}
\label{phi-der-f} 
\begin{split}
&
\vert
( 1 - \triangle )^{s/2}
[ ( \phi \partial_{k} f ) ( x + \varepsilon \bullet ) ]
(z)
\vert
\\
&=
\Big\vert
\frac{1}{\Gamma (-\frac{s}{2})}
\int_{0}^{\infty}
u^{-\frac{s}{2}-1} \mathrm{e}^{-u}
\int_{\mathbb{R}^{d}}
\frac{
	\mathrm{e}^{\text{{\small $
			- \frac{ \vert z - z^{\prime} \vert^{2} }{ 4u }
		$}}}
}{
	( 4\pi u )^{d/2}
}
\big( \phi \partial_{k} f \big) (x+\varepsilon z^{\prime})
\mathrm{d}z^{\prime}
\mathrm{d}u
\Big\vert
\end{split}
\end{equation}
By using (\ref{der-of-f}),
change of variables $\varepsilon z^{\prime} = z^{\prime\prime}$
and the semigroup property of
$\mathrm{e}^{t \triangle /2}$
(recall that $c \geqslant 2$),
we have
\begin{equation*}
\begin{split}
&
\int_{\mathbb{R}^{d}}
\frac{
	\mathrm{e}^{\text{{\small $
			- \frac{ \vert z - z^{\prime} \vert^{2} }{ 4u }
		$}}}
}{
	( 4\pi u )^{d/2}
}
\vert
\big( \phi \partial_{k} f \big) (x+\varepsilon z^{\prime})
\vert
\mathrm{d}z^{\prime} \\
&\leqslant 
C
\int_{0}^{+\infty}
v^{-1/2}
\mathrm{e}^{-v}
\mathrm{d} v
\int_{\mathbb{R}^{d}}
\frac{
	\mathrm{e}^{\text{{\small$
		- \frac{ \vert z-z^{\prime} \vert^{2} }{ 4u }
	$}}}
}{
	( 4 \pi u )^{d/2}
}
\frac{
	\mathrm{e}^{\text{{\small$
		- \frac{ \vert ( x+\varepsilon z^{\prime} ) - y \vert^{2} }{ 2cv }
	$}}}
}{
	( 2 \pi c v )^{d/2}
}
\mathrm{d} z^{\prime} \\
&= 
C
\varepsilon^{d}
\int_{0}^{+\infty}
v^{-1/2}
\mathrm{e}^{-v}
\mathrm{d} v
\int_{\mathbb{R}^{d}}
\frac{
	\mathrm{e}^{\text{{\small$
		- \frac{ \vert \varepsilon z-\varepsilon z^{\prime} \vert^{2} }{ 4 \varepsilon^{2} u }
	$}}}
}{
	( 4 \pi \varepsilon^{2} u )^{d/2}
}
\frac{
	\mathrm{e}^{\text{{\small$
		- \frac{ \vert ( x+\varepsilon z^{\prime} ) - y \vert^{2} }{ 2cv }
	$}}}
}{
	( 2 \pi c v )^{d/2}
}
\mathrm{d} z^{\prime} \\
&= 
C
\int_{0}^{+\infty}
v^{-1/2}
\mathrm{e}^{-v}
\mathrm{d} v
\int_{\mathbb{R}^{d}}
\frac{
	\mathrm{e}^{\text{{\small$
		- \frac{ \vert \varepsilon z - z^{\prime\prime} \vert^{2} }{ 4 \varepsilon^{2} u }
	$}}}
}{
	( 4 \pi \varepsilon^{2} u )^{d/2}
}
\frac{
	\mathrm{e}^{\text{{\small$
		- \frac{ \vert ( x + z^{\prime\prime} ) - y \vert^{2} }{ 2cv }
	$}}}
}{
	( 2 \pi c v )^{d/2}
}
\mathrm{d} z^{\prime\prime} \\
&= 
C
\int_{0}^{+\infty}
v^{-1/2}
\mathrm{e}^{-v}
\frac{
	\mathrm{e}^{\text{{\small$
		-
		\frac{
			\vert \varepsilon z - (y-x) \vert^{2}
		}{
			2 ( 2 \varepsilon^{2} u + cv )
		}
	$}}}
}{
	( 2 \pi ( 2 \varepsilon^{2} u + cv ) )^{d/2}
}
\mathrm{d} v \\
&\leqslant 
C
( 2^{-1} c )^{d/2}
\int_{0}^{+\infty}
v^{-1/2}
\mathrm{e}^{-v}
\frac{
	\mathrm{e}^{\text{{\small$
		-
		\frac{
			\vert \varepsilon z - (y-x) \vert^{2}
		}{
			2 ( c \varepsilon^{2} u + cv )
		}
	$}}}
}{
	( 2 \pi ( c \varepsilon^{2} u + cv ) )^{d/2}
}
\mathrm{d} v,
\end{split}
\end{equation*}
where in the last inequality,
we have used that
$$
\sup_{u,v \geqslant 0}
\frac{
	( c \varepsilon^{2} + cv )^{d/2}
}{
	( 2 \varepsilon^{2} + cv )^{d/2}
}
\leqslant
\sup_{u,v \geqslant 0}
\frac{
	( c \varepsilon^{2} + cv )^{d/2}
}{
	( 2 \varepsilon^{2} + 2v )^{d/2}
}
=
(2^{-1}c)^{d/2}.
$$
Substituting this estimate into (\ref{phi-der-f}),
we get
\begin{equation*}
\begin{split}
&
\vert
( 1 - \triangle )^{s/2}
[ ( \phi \partial_{k} f ) ( x + \varepsilon \bullet ) ]
(z)
\vert \\
&\leqslant 
c_{2}
\int_{0}^{\infty}
\int_{0}^{\infty}
u^{-\frac{s}{2}-1} \mathrm{e}^{-u}
v^{-1/2} \mathrm{e}^{-v}
\frac{
	\mathrm{e}^{\text{{\small $
		- \frac{
			 \vert \varepsilon z - (y-x) \vert^{2}
		}{
			2c (\varepsilon^{2}u + v)
		}
	$}}}
}{
	( 2\pi c (\varepsilon^{2} u + v) )^{d/2}
}
\mathrm{d}u
\mathrm{d}v \\
&= 
c_{2}
\varepsilon^{s}
\int_{0}^{\infty}
\int_{0}^{\infty}
u^{-\frac{s}{2}-1}
v^{-1/2}
\mathrm{e}^{\text{{\small $
	- ( \varepsilon^{-2} u + v )
$}}}
\frac{
	\mathrm{e}^{\text{{\small $
		- \frac{
			 \vert \varepsilon z - (y-x) \vert^{2}
		}{
			2c (u + v)
		}
	$}}}
}{
	( 2\pi c ( u + v) )^{d/2}
}
\mathrm{d}u
\mathrm{d}v \\
&\leqslant 
c_{2}
\varepsilon^{s}
\int_{0}^{\infty}
\int_{0}^{\infty}
u^{-\frac{s}{2}-1}
v^{-1/2} \mathrm{e}^{-(u+v)}
\frac{
	\mathrm{e}^{\text{{\small $
		- \frac{
			 \vert \varepsilon z - (y-x) \vert^{2}
		}{
			2c (u + v)
		}
	$}}}
}{
	( 2\pi c ( u + v) )^{d/2}
}
\mathrm{d}u
\mathrm{d}v \\
&= 
c_{2}^{\prime}
\varepsilon^{s}
\int_{0}^{\infty}
u^{\frac{1-s}{2}-1}
\mathrm{e}^{-u}
\frac{
	\mathrm{e}^{\text{{\small $
		- \frac{
			 \vert \varepsilon z - (y-x) \vert^{2}
		}{
			2cu
		}
	$}}}
}{
	( 2\pi c u )^{d/2}
}
\mathrm{d}u,
\end{split}
\end{equation*}
for some constants $c_{2}, c_{2}^{\prime} > 0$
independent of $\varepsilon$,
and so that, by using
Lemma \ref{frac_dom},
\begin{equation*}
\begin{split}
&
\Vert
(I-\mathcal{L})^{s/2} ( \phi \partial_{k} f ) (X_{t})
\Vert_{p}^{p^{\prime}} \\
&\leqslant
c_{2}^{\prime\prime}
\Vert
( 1 - \triangle )^{s/2}
[
(
\phi
\partial_{k} f
)
( x + \varepsilon \bullet )
] (z)
\Vert_{L_{p^{\prime}}(\mathbb{R}^{d}, \mathrm{d}z)}^{p^{\prime}} \\
&\leqslant
c_{2}^{\prime\prime\prime}
\varepsilon^{sp^{\prime}}
\int_{\mathbb{R}^{d}}
\Big\{
\int_{0}^{\infty}
u^{\frac{1-s}{2}-1}
\mathrm{e}^{-u}
\frac{
	\mathrm{e}^{\text{{\small $
		- \frac{
			 \vert \varepsilon z - (y-x) \vert^{2}
		}{
			2cu
		}
	$}}}
}{
	( 2\pi c u )^{d/2}
}
\mathrm{d}u
\Big\}^{p^{\prime}}
\mathrm{d}z \\
&=
c_{2}^{\prime\prime\prime}
\varepsilon^{sp^{\prime}-d}
\int_{\mathbb{R}^{d}}
\Big\{
\vert
	z - (y-x)
\vert^{- ( d - (1-s) )}
\int_{0}^{\infty}
u^{\frac{d-(1-s)}{2}-1}
\mathrm{e}^{-u}
\mathrm{e}^{\text{{\small $
	- \frac{
		 \vert z - (y-x) \vert^{2}
	}{
		2cu
	}
$}}}
\mathrm{d}u
\Big\}^{p^{\prime}}
\mathrm{d}z.
\end{split}
\end{equation*}
for some constants
$c_{2}^{\prime\prime}, c_{2}^{\prime\prime\prime} > 0$.
Here, note that
$d-(1-s) > 0$
because
$d \geqslant 2$ and $s > -\frac{d}{2}$.
Hence
$
\int_{0}^{\infty}
u^{\frac{d-(1-s)}{2}-1}
\mathrm{e}^{-u}
\mathrm{d}u
< +\infty
$
and
by repeating the above argument,
we see that
if $p < p^{\prime} < \frac{d}{d+s-1}$,
$
\limsup_{\varepsilon \downarrow 0}
\varepsilon^{ - ( sp^{\prime} - d ) }
\Vert
	( I - \mathcal{L} )^{s/2}
	( \phi \partial_{k} f ) ( X_{t} )
\Vert_{p}^{ p^{\prime} }
< +\infty
$,
so that
\begin{equation*}
\begin{split}
&
\limsup_{\varepsilon \downarrow 0}
\Big(
\varepsilon^{ -(sp-d) }
\Vert
	( I - \mathcal{L} )^{s/2}
	( \phi \partial_{k} f ) ( X_{t} )
\Vert_{p}^{p}
\Big)^{ p^{\prime} / p } \\
&=
\limsup_{\varepsilon \downarrow 0}
\varepsilon^{ d ( \frac{ p^{\prime} }{ p } - 1 ) }
\Big(
\varepsilon^{ - ( sp^{\prime} - d ) }
\Vert
	( I - \mathcal{L} )^{s/2}
	( \phi \partial_{k} f ) ( X_{t} )
\Vert_{p}^{p^{\prime}}
\Big)
= 0.
\end{split}
\end{equation*}
This is nothing but (\ref{(ii)-b}).

Next we prove (\ref{(iii)-a}).
By virtue of
Lemma \ref{frac_dom},
it suffices to show
\begin{equation}
\label{(iii)-1}
\begin{split}
&
\Vert
	f ( x+ \varepsilon z )
\Vert_{L_{p^{\prime}}(\mathbb{R}^{d}, \mathrm{d}z)}^{p^{\prime}} \\
&\leqslant
c_{3}
\varepsilon^{-d}
\Big[
\int_{\mathbb{R}^{d}}
\Big\{
	\vert
	z - (x-y)
	\vert^{1-d}
	\int_{0}^{\infty}
	u^{ \frac{d}{2} - 1 }
	\mathrm{e}^{-u}
	\mathrm{e}^{\text{{\small $
		- \frac{ \vert z - (x-y) \vert^{2} }{ 2cu }
	$}}}
	\mathrm{d}u
\Big\}^{p^{\prime}}
\mathrm{d}z \\
&\hspace{20mm}+
\int_{\mathbb{R}^{d}}
\Big\{
	\vert
	z - (x-y)
	\vert^{-(d-1)/2}
	\int_{0}^{\infty}
	u^{-1/2}
	\mathrm{e}^{-u}
	\mathrm{e}^{\text{{\small $
		- \frac{ \vert z - (x-y) \vert^{2} }{ 2cu }
	$}}}
	\mathrm{d}u
\Big\}^{p^{\prime}}
\mathrm{d}z
\Big]
\end{split}
\end{equation}
for some constant
$
c_{3}
>0
$,
independent of $\varepsilon$.
Actually, then we have
$
\limsup_{\varepsilon \downarrow 0}
\varepsilon^{d}
\Vert
	f ( X_{t} )
\Vert_{p}^{p^{\prime}}
< +\infty
$
if $p < p^{\prime} < \frac{d}{d-1}$,
from which we can conclude (\ref{(iii)-a}).
To prove (\ref{(iii)-1}), we begin with the inequality
\begin{equation*}
\begin{split}
f ( x+ \varepsilon z )
\leqslant
\frac{ C (2\pi c)^{-d/2} }{ \Gamma (1) }
\int_{0}^{\infty}
u^{-d/2}
\mathrm{e}^{-u}
\mathrm{e}^{\text{{\small $
	- \frac{ \vert y - ( x + \varepsilon z ) \vert^{2} }{ 2cu }
$}}}
\mathrm{d}u .
\end{split}
\end{equation*}
We divide the integral as
\begin{equation*}
\begin{split}
&
\int_{0}^{\infty}
u^{-d/2}
\mathrm{e}^{-u}
\mathrm{e}^{\text{{\small $
	- \frac{ \vert y - ( x + \varepsilon z ) \vert^{2} }{ 2cu }
$}}}
\mathrm{d}u \\
&=
\int_{0}^{\vert y - ( x + \varepsilon z ) \vert}
u^{-d/2}
\mathrm{e}^{-u}
\mathrm{e}^{\text{{\small $
	- \frac{ \vert y - ( x + \varepsilon z ) \vert^{2} }{ 2cu }
$}}}
\mathrm{d}u
+
\int_{\vert y - ( x + \varepsilon z ) \vert}^{\infty}
u^{-d/2}
\mathrm{e}^{-u}
\mathrm{e}^{\text{{\small $
	- \frac{ \vert y - ( x + \varepsilon z ) \vert^{2} }{ 2cu }
$}}}
\mathrm{d}u .
\end{split}
\end{equation*}
The first term in the last equation is estimated as
\begin{equation*}
\begin{split}
&
\int_{0}^{\vert y - ( x + \varepsilon z ) \vert}
u^{-d/2}
\mathrm{e}^{-u}
\mathrm{e}^{\text{{\small $
	- \frac{ \vert \varepsilon z - (x-y) \vert^{2} }{ 2cu }
$}}}
\mathrm{d}u \\
&=
(2c)^{ \frac{d}{2} - 1 }
\vert
	\varepsilon z - (x-y)
\vert^{2-d}
\int_{
	\frac{ \vert y - ( x + \varepsilon z ) \vert }{ 2c }
}^{
	\infty
}
u^{ \frac{d}{2} - 2 }
\mathrm{e}^{-u}
\mathrm{e}^{\text{{\small $
	- \frac{ \vert \varepsilon z - (x-y) \vert^{2} }{ 2cu }
$}}}
\mathrm{d}u \\
&\leqslant
(2c)^{ \frac{d}{2} }
\vert
	\varepsilon z - (x-y)
\vert^{1-d}
\int_{
	\frac{ \vert y - ( x + \varepsilon z ) \vert }{ 2c }
}^{
	\infty
}
u^{ \frac{d}{2} - 1 }
\mathrm{e}^{-u}
\mathrm{e}^{\text{{\small $
	- \frac{ \vert \varepsilon z - (x-y) \vert^{2} }{ 2cu }
$}}}
\mathrm{d}u \\
&\leqslant
(2c)^{ \frac{d}{2} }
\vert
	\varepsilon z - (x-y)
\vert^{1-d}
\int_{0}^{\infty}
u^{ \frac{d}{2} - 1 }
\mathrm{e}^{-u}
\mathrm{e}^{\text{{\small $
	- \frac{ \vert \varepsilon z - (x-y) \vert^{2} }{ 2cu }
$}}}
\mathrm{d}u.
\end{split}
\end{equation*}
On the other hand, the second term is
\begin{equation*}
\begin{split}
&
\int_{\vert y - ( x + \varepsilon z ) \vert}^{\infty}
u^{-d/2}
\mathrm{e}^{-u}
\mathrm{e}^{\text{{\small $
	- \frac{ \vert y - ( x + \varepsilon z ) \vert^{2} }{ 2cu }
$}}}
\mathrm{d}u
\leqslant
\vert
	\varepsilon z
	-
	(x-y)
\vert^{-(d-1)/2}
\int_{0}^{\infty}
u^{-1/2}
\mathrm{e}^{-u}
\mathrm{e}^{\text{{\small $
	- \frac{ \vert y - ( x + \varepsilon z ) \vert^{2} }{ 2cu }
$}}}
\mathrm{d}u .
\end{split}
\end{equation*}
Now
a change of variable leads us to (\ref{(iii)-1})
and
thus
(\ref{(iii)-a}) is proved.

For (\ref{(iii)-b}),
it is sufficient to prove
\begin{equation}
\label{(iii)-2}
\begin{split}
&
\Vert
( 1 - \triangle )^{s/2} [ ( \phi f ) ( x + \varepsilon \bullet ) ] (z)
\Vert_{ L_{p^{\prime}} ( \mathbb{R}^{d}, \mathrm{d}z ) }^{p^{\prime}} \\
&\leqslant
c_{4}
\varepsilon^{
	s p^{\prime}
	- d
}
\Big[
	\int_{\mathbb{R}^{d}}
	\Big\{
	\vert z - (x-y) \vert^{ 1 - (s+d) }
	\int_{0}^{\infty}
	u^{ \frac{d+s}{2}-1 }
	\mathrm{e}^{-u}
	\mathrm{e}^{
		- \frac{ \vert z - (x-y) \vert^{2} }{ 2cu }
	}
	\mathrm{d}u
	\Big\}^{p^{\prime}}
	\mathrm{d}z \\
&\hspace{20mm}+
	\int_{\mathbb{R}^{d}}
	\Big\{
	\vert z - (x-y) \vert^{\frac{2-d}{2}}
	\int_{0}^{\infty}
	u^{ - \frac{s}{2} }
	\mathrm{e}^{-u}
	\mathrm{e}^{
		- \frac{ \vert z - (x-y) \vert^{2} }{ 2cu }
	}
	\mathrm{d}u
	\Big\}^{p^{\prime}}
	\mathrm{d}z
\Big]
\end{split}
\end{equation}
where
$
c_{4}
>0
$
is a constant independent of $\varepsilon$.
Note that $d+s > 0$ and $-s > 0$, so that
it is assured that
$
\int_{0}^{\infty}
u^{ \frac{d+s}{2}-1 }
\mathrm{e}^{-u}
\mathrm{d}u
< +\infty
$
and
$
\int_{0}^{\infty}
u^{ -\frac{s}{2} }
\mathrm{e}^{-u}
\mathrm{d}u
< +\infty
$,
respectively.
Note also that
$1 - (s+d) < 0$.
These imply
$
\limsup_{\varepsilon \downarrow 0}
\varepsilon^{d-s p^{\prime}}
\Vert
	( I - \mathcal{L} )^{ s/2 }
	( \phi f )
	( X_{t} )
\Vert_{p}^{p^{\prime}}
< +\infty
$
if
$
p < p^{\prime} < \min \{ \frac{d}{d+s-1}, \frac{2d}{d-2} \} \}
$,
and so (\ref{(iii)-b}).

To prove (\ref{(iii)-2}),
we apply a similar argument,
with assuming $\varepsilon \in (0,1]$,
which leads to
\begin{equation*}
\begin{split}
&
\vert
( 1 - \triangle )^{s/2} [ ( \phi f ) ( x + \varepsilon \bullet ) ] (z)
\vert
\leqslant
\mathrm{const.}
\varepsilon^{s}
\int_{0}^{\infty}
u^{ \frac{ 2-(s+d) }{ 2 } - 1 }
\mathrm{e}^{-u}
\mathrm{e}^{\text{{\small $
	-
	\frac{
		\vert \varepsilon z - (x-y) \vert^{2}
	}{
		2cu
	}
$}}}
\mathrm{d}u .
\end{split}
\end{equation*}
We divide the integral as
\begin{equation*}
\begin{split}
&
\int_{0}^{\infty}
u^{ \frac{ 2-(s+d) }{ 2 } - 1 }
\mathrm{e}^{-u}
\mathrm{e}^{\text{{\small $
	-
	\frac{
		\vert \varepsilon z - (x-y) \vert^{2}
	}{
		2cu
	}
$}}}
\mathrm{d}u \\
&=
\int_{0}^{\vert \varepsilon z - (x-y) \vert}
u^{ \frac{ 2-(s+d) }{ 2 } - 1 }
\mathrm{e}^{-u}
\mathrm{e}^{\text{{\small $
	-
	\frac{
		\vert \varepsilon z - (x-y) \vert^{2}
	}{
		2cu
	}
$}}}
\mathrm{d}u \\
&\hspace{40mm}+
\int_{\vert \varepsilon z - (x-y) \vert}^{\infty}
u^{ \frac{ 2-(s+d) }{ 2 } - 1 }
\mathrm{e}^{-u}
\mathrm{e}^{\text{{\small $
	-
	\frac{
		\vert \varepsilon z - (x-y) \vert^{2}
	}{
		2cu
	}
$}}}
\mathrm{d}u .
\end{split}
\end{equation*}
We estimate the first term as follows.
\begin{equation*}
\begin{split}
&
\int_{0}^{\vert \varepsilon z - (x-y) \vert}
u^{ \frac{ 2-(s+d) }{ 2 } - 1 }
\mathrm{e}^{-u}
\mathrm{e}^{\text{{\small $
	-
	\frac{
		\vert \varepsilon z - (x-y) \vert^{2}
	}{
		2cu
	}
$}}}
\mathrm{d}u \\
&=
(2c)^{ 1 - \frac{d+s-2}{2} }
\vert \varepsilon z - (x-y) \vert^{2-(s+d)}
\int_{ \frac{ \vert \varepsilon z - (x-y) \vert }{2c} }^{\infty}
u^{ \frac{ d+s-2 }{ 2 } - 1 }
\mathrm{e}^{-u}
\mathrm{e}^{\text{{\small $
	-
	\frac{
		\vert \varepsilon z - (x-y) \vert^{2}
	}{
		2cu
	}
$}}}
\mathrm{d}u \\
&\leqslant
(2c)^{ 1 - \frac{d+s-3}{2} }
\vert \varepsilon z - (x-y) \vert^{1-(s+d)}
\int_{ \frac{ \vert \varepsilon z - (x-y) \vert }{2c} }^{\infty}
u^{ \frac{ d+s }{ 2 } - 1 }
\mathrm{e}^{-u}
\mathrm{e}^{\text{{\small $
	-
	\frac{
		\vert \varepsilon z - (x-y) \vert^{2}
	}{
		2cu
	}
$}}}
\mathrm{d}u \\
&\leqslant
(2c)^{ 1 - \frac{d+s-3}{2} }
\vert \varepsilon z - (x-y) \vert^{1-(s+d)}
\int_{ 0 }^{\infty}
u^{ \frac{ d+s }{ 2 } - 1 }
\mathrm{e}^{-u}
\mathrm{e}^{\text{{\small $
	-
	\frac{
		\vert \varepsilon z - (x-y) \vert^{2}
	}{
		2cu
	}
$}}}
\mathrm{d}u .
\end{split}
\end{equation*}
The second term is estimated as
\begin{equation*}
\begin{split}
&
\int_{\vert \varepsilon z - (x-y) \vert}^{\infty}
u^{ \frac{ 2-(s+d) }{ 2 } - 1 }
\mathrm{e}^{-u}
\mathrm{e}^{\text{{\small $
	-
	\frac{
		\vert \varepsilon z - (x-y) \vert^{2}
	}{
		2cu
	}
$}}}
\mathrm{d}u
\leqslant
\vert \varepsilon z - (x-y) \vert^{\frac{2-d}{2}}
\int_{0}^{\infty}
u^{ - \frac{ s }{ 2 } }
\mathrm{e}^{-u}
\mathrm{e}^{\text{{\small $
	-
	\frac{
		\vert \varepsilon z - (x-y) \vert^{2}
	}{
		2cu
	}
$}}}
\mathrm{d}u.
\end{split}
\end{equation*}
Combining these, and a change of variable,
we reach the estimate (\ref{(iii)-2}),
and hence (\ref{(iii)-b}) is proved.

(d)
In
view of
(\ref{(ii)-a}),
(\ref{(ii)-b}),
(\ref{(iii)-a})
and
(\ref{(iii)-b}),
what we have to do now is to find
$p^{\prime}, q, r > 1$
such that
\begin{itemize}
\item[$\bullet$]
$
\frac{1}{p^{\prime}} + \frac{1}{q} + \frac{1}{r} < 1
$;

\item[$\bullet$]
$
p^{\prime} < \frac{d}{d-1}
$
and
$
p^{\prime\prime} := q(p-1) < \frac{d}{d-(1-s)}
$.
\end{itemize}
In fact, since
$s < \frac{p}{p-1} - d$,
we can take $\varepsilon \in (0, \frac{1}{d} )$ such that
$
s < \frac{p}{p-1} - \frac{d}{1-\varepsilon d}
<
\frac{p}{p-1} - d
$.
Then take $p^{\prime}, q > 1$ so that
$$
\frac{d-1}{d}
< \frac{1}{p^{\prime}}
< \frac{d-1}{d} + \varepsilon
\quad
\text{and}
\quad
\frac{1}{q}
=
\frac{1}{d} - \varepsilon
\hspace{2mm}
\Big(
	< \frac{1}{d}
\Big) .
$$
These conditions imply
$
1/p^{\prime} + 1/q < \frac{d-1}{d} + \frac{1}{d} = 1
$
and hence one can take $r > 0$ such that
$
1/p^{\prime} + 1/q + 1/r < 1
$.
Finally, we have
if $p^{\prime\prime} - 1 > 0$,
then
$
1 - \frac{ ( p^{\prime\prime} - 1 ) d }{ p^{\prime\prime} }
= 1 - \frac{(p^{\prime\prime}-1)d}{(p-1)q}
>
1 - \frac{(p^{\prime\prime}-1)d}{(p-1)d}
=
1 - \frac{p^{\prime\prime}-1}{p-1}
=
\frac{p}{p-1} - q
=
\frac{p}{p-1} - \frac{d}{1-\varepsilon d}
> s
$,
which implies
$
p^{\prime\prime}
< \frac{d}{d-(1-s)}
$.
If $p^{\prime\prime} - 1 \leqslant 0$,
then
$
1 - \frac{ ( p^{\prime\prime} - 1 ) d }{ p^{\prime\prime} }
> 0 > s
$,
which also implies
$
p^{\prime\prime}
< \frac{d}{d-(1-s)}
$.

Therefore, by taking
$p^{\prime}, q, r > 1$
as above,
we conclude from
(\ref{4factors}), (\ref{4factors'})
and
(\ref{3rd-fac})--(\ref{(iii)-b})
that
$
\Vert
	( \phi \partial_{k} f ) (X_{t})
\Vert_{p,s}
$
and
$
\Vert
	( \phi f ) (X_{t})
\Vert_{p,s}
$
converge to zero as $t = \varepsilon^{2} \downarrow 0$.

\end{proof} 


\end{document}